\documentclass[sn-mathphys,Numbered]{sn-jnl}

\usepackage{graphicx}%
\usepackage{multirow}%
\usepackage{amsmath,amssymb,amsfonts}%
\usepackage{amsthm}%
\usepackage{dsfont,esint,mathtools}
\usepackage{mathrsfs}%
\usepackage[title]{appendix}%
\usepackage{xcolor}%
\usepackage{textcomp}%
\usepackage{manyfoot}%
\usepackage{booktabs}%
\usepackage{algorithm}%
\usepackage{algorithmicx}%
\usepackage{algpseudocode}%
\usepackage{listings}%
\usepackage{enumerate}
\usepackage{hyperref}

\makeatletter
 
\newtheorem{theorem}{Theorem}[section]
\newtheorem{proposition}[theorem]{Proposition}%
\newtheorem{remark}[theorem]{Remark} %

\newtheorem{definition}[theorem]{Definition} %
\newtheorem{lemma}[theorem]{Lemma}
\newtheorem{assumption}[theorem]{Assumption}
\numberwithin{equation}{section}

  \def\sC {{\mathcal C}}
 \def\sE {{\mathcal E}} \def\sF {{\mathcal F}}

  \def\sX {{\mathcal X}}

 \def\bN {{\mathbb N}} 
  \def\bR {{\mathbb R}}
\def\bS {{\mathbb S}}

\def\wt{\widetilde}
\def\wh{\widehat}
\def\vp{\varphi}
\def\Cap{\operatorname{Cap}}
\newcommand\restr[2]{{
		\left.\kern-\nulldelimiterspace 
		#1 
		\vphantom{\big|} 
		\right|_{#2} 
}} 
\def\eps{\epsilon}
\newcommand{\contfunc}{C}
\newcommand{\abs}[1]{{\left\vert\kern-0.25ex #1
		\kern-0.25ex\right\vert}}
\newcommand\norm[1]{\left\lVert#1\right\rVert} 
\def\grad{\nabla}
\def\q{\quad} \def\qq{\qquad}
\newtheorem*{notation}{Notation}
\newcommand{\one}{\mathds{1}} 
\def\diam{{\mathop{{\rm diam }}}}
\def\ol{\overline}      
\newcommand{\arxiv}[1]{{\tt \href{http://arxiv.org/abs/#1}{arXiv:#1}}}     
\newcommand{\on}[1]{\operatorname{ #1}}
\def\supp{\mathop{{\rm supp}}}
\newcommand{\loc}[0]{\operatorname{loc}}
\newcommand{\set}[1]{\left\{ #1 \right\}}

\begin{document}

\title[Reflected diffusion   on uniform domains]{Heat kernel for reflected diffusion and extension property on uniform domains}


\author{\fnm{Mathav} \sur{Murugan} \protect\footnotemark[0]}\email{mathav@math.ubc.ca}

\affil{\orgdiv{Department of Mathematics}, \orgname{The University of British Columbia}, \orgaddress{\street{1984 Mathematics Road}, \city{Vancouver}, \postcode{V6T 1Z2}, \state{British Columbia}, \country{Canada}}}

\renewcommand{\thefootnote}{}
\footnotetext{The author is partially supported by NSERC and the Canada research chairs program.}
\renewcommand{\thefootnote}{\arabic{footnote}}
\setcounter{footnote}{0}


\abstract{We study reflected diffusion on uniform domains where the underlying space admits a symmetric diffusion that satisfies sub-Gaussian heat kernel estimates.
	A celebrated theorem of Jones (Acta Math. 1981) states that uniform domains in      Euclidean space are extension domains for Sobolev spaces. In this work, we obtain a similar extension property for  metric spaces equipped with a Dirichlet form   whose heat kernel satisfies a sub-Gaussian estimate. 
	We introduce a scale-invariant version of this extension property and  apply it to show that the reflected diffusion process on such a uniform domain inherits various properties from the ambient space, such as Harnack inequalities, cutoff energy inequality, and sub-Gaussian heat kernel bounds. 
	In particular, our work extends  Neumann heat kernel estimates of Gyrya and Saloff-Coste (Ast\'erisque 2011) beyond the Gaussian space-time scaling. 
	Furthermore, our estimates on the extension operator imply that the energy measure of the boundary of a uniform domain is always zero. This property of the energy measure is a broad generalization of Hino's result (PTRF 2013) that proves the vanishing of the energy measure on the outer square boundary of the standard Sierpi\'nski carpet equipped with the self-similar Dirichlet form.}

\keywords{Reflected diffusion, sub-Gaussian heat kernel estimate, extension problem, uniform domains, Whitney cover, Dirichlet form.}


\pacs[MSC Classification]{31C25, 31E05, 35K08,  60J46, 60J60}

\maketitle

\section{Introduction}\label{sec1}

The goal of this work is obtaining heat kernel  (transition probability) estimates for reflected diffusion on `nice domains' when the underlying space admits a symmetric diffusion process with sub-Gaussian heat kernel estimates. 
We wish to understand if the corresponding reflected diffusion on these domains also satisfy similar heat kernel estimates as the underlying space. 
A more general guiding question is the following: {\em what properties are inherited by a domain from the ambient space?}

Gaussian estimates of the heat kernel  for symmetric diffusion have been known to hold in a broad range of settings including manifolds with non-negative Ricci curvature \cite{LY}, uniformly elliptic operators \cite{Aro,Sal92}, weighted manifolds \cite[Theorem 7.1]{GrS}, Lie groups of polynomial growth \cite[Chapter IV]{VSC} and many other examples \cite[Section 3.3]{Sal10}. 
This Gaussian heat kernel estimate is congruous with the  \emph{usual} space-time scaling  property that the expected exit time from a ball of radius $r$ grows like $r^2$. Sub-Gaussian estimate on the heat kernel allows for a richer possibility of space-time scaling where the    expected exit time from a ball of radius $r$ grows like $\Psi(r)$, where $\Psi:(0,\infty) \to (0,\infty)$ is a function that governs the space-time scaling. Sub-Gaussian heat kernel estimates were first established by Barlow and Perkins for the Brownian motion on Sierpi\'nski gasket \cite{BP} and was later shown to hold in various  fractals by several authors \cite{FHK, BB96, BH, Kum, Lin}. We refer to  Barlow's monograph for an introduction to diffusion on fractals \cite{Bar98}.

The `nice domains' we consider in this work are uniform domains. 
Uniform domains were introduced independently by Martio and Sarvas \cite{MS} and
Jones \cite{Jon81}. This class includes Lipschitz domains, and more generally non-tangentially accesible (NTA) domains. Uniform domains are relevant in various contexts  such as   extension property \cite{Jon81, HeK}, Gromov hyperbolicity \cite{BHK}, boundary Harnack principle \cite{Aik}, geometric function theory \cite{MS,GH,Geh}, and heat kernel estimates \cite{GS, CKKW}.  Uniform domains are abundant. Every bounded domain   can be approximated by a uniform domain on a large class of metric spaces \cite[Theorem 1.1]{Raj}.

A novel feature of the work is to use  \emph{extension problem} to obtain  heat kernel estimates for reflected diffusion on a domain.
Given a space of functions $\sF(U)$ in a domain $U$ with $U \subset X$  and a  space of functions $\sF(X)$ in the underlying space $X$, the \emph{extension problem} asks if every function in $\sF(U)$ can be extended to a function belonging to $\sF(X)$. Often there are additional requirements on the extension such as linearity and boundedness. We refer the reader to Stein's book \cite[Chapter VI]{Ste} for a nice introduction to extension problem of functions in Sobolev space  on Lipschitz domains in $\mathbb{R}^n$.
In our work, the relevant function spaces will be domain of Dirichlet forms. Our function spaces can be thought of as an abstraction of the $W^{1,2}$ Sobolev space (function and its first order distributional derivatives are in $L^2$).   While heat kernel estimates and extension problems have been extensively studied,  using the extension property to obtain heat kernel estimates is new, at least to the author's knowledge. 

To explain the relationship between reflected diffusion and the extension problem, we recall the constructions of reflected Brownian motion on Euclidean domains. The   different constructions   using stochastic differential equation (SDE)   and Dirichlet forms   have their origins in the works of Skorokhod \cite{S} and Fukushima \cite{F} respectively. These approaches were studied by various authors \cite{LS,Che,Tan}. We refer to the introduction of \cite{Che} for a nice overview of these two approaches and a more complete list of references. 
For a smooth domain $U$ in $\bR^n$, the SDE approach involves solving the stochastic differential equation 
\[
Y(t) = Y(0)+ B(t)+ \int_0^t \vec{n}(Y(s)) \, dL_s,
\]
where $B(t)$ is the standard Brownian motion on $\mathbb{R}^n$, $L_s$ is the `boundary local time' of the process $Y(s)$ and $\vec{n}(x)$ is the inward pointing unit normal vector at $x \in \partial U$. Heuristically, the last term is responsible for `pushing the diffusion $Y(t)$ back into the domain' when it hits the boundary so that it stays in $\overline{U}$.

Given a smooth domain $U$, the Dirichlet form approach involves the bilinear form
\[
\sE_U(f,f):= \frac{1}{2} \int_U \abs{\grad f}^2(x) \, dx,
\]
for all $f \in W^{1,2}(U)$, where $\nabla f$ denotes the distributional gradient of $f$ and $W^{1,2}(U)$ denotes the subspace of functions in $L^2(U)$  whose distributional first order partial derivatives are also in $L^2(U)$. Using the theory of Dirichlet forms, Fukushima \cite{F} constructs a Markov process with continuous sample paths in some abstract closure of $U$ (called the Martin-Kuramochi compactification).

If $U$ is a smooth domain (or more generally, a uniform domain) this abstract closure can be identified with $\overline{U}$.
As pointed out in \cite[p.5]{BCR}, if $U$ is a $W^{1,2}$-Sobolev extension domain (that is, there is a bounded linear operator $E: W^{1,2}(U) \to W^{1,2}(\bR^n)$), then the Dirichlet form approach yields a Markov process with continuous paths in $\overline{ U}$. These two approaches lead to the \emph{same process} on smooth domains.
The Dirichlet form approach has the advantage that it also  works when the domain is not smooth and more importantly when the ambient space is not smooth. 

Gyrya and Saloff-Coste show that for any symmetric diffusion satisfying Gaussian heat kernel bounds, the reflected diffusion on any uniform domain (or more generally, inner uniform domains) also satisfies Gaussian heat kernel estimates \cite[Theorem 3.10]{GS}. A natural question is whether a similar result is true for more general space-time scaling given by sub-Gaussian heat kernel bounds.  By a celebrated theorem of Jones \cite[Theorem 1]{Jon81}, there is a bounded linear extension map $E: W^{1,2}(U) \to W^{1,2}(\mathbb{R}^n)$ for any uniform domain (Jones' theorem is valid for a more general Sobolev space $W^{k,p}$). Since $W^{1,2}(U)$ is the domain of the Dirichlet form for the reflected diffusion on $U$, one could ask if this is a general phenomenon for any reflected diffusion on a uniform domain where the diffusion on the ambient space satisfies   sub-Gaussian heat kernel estimate. Our main result answers these questions. As mentioned earlier, we use the extension result to obtain heat kernel bounds. Stated informally, our main results are: 
\begin{enumerate}[(i)]
	\item For any  symmetric diffusion satisfying sub-Gaussian heat kernel bounds, the corresponding reflected diffusion on any uniform domain  also satisfies sub-Gaussian heat kernel estimates with the same space-time scaling (Theorem \ref{t:hkeunif}).
	\item In the same setting as (i), there is a bounded linear \emph{extension} map from the domain of the Dirichlet form for the reflected diffusion on a uniform domain to the domain of the Dirichlet form for the diffusion on the ambient space. This extension map is   bounded \emph{at all locations and scales} (Theorem \ref{t:extend}).
	\item In the same setting as (i), the energy measure of any function in the domain of the Dirichlet form on the boundary of any uniform domain is identically zero (Theorem \ref{t:benergy}).
\end{enumerate}
As mentioned above the results (i) and (ii) above can be viewed as analogues of  \cite[Theorem 3.10]{GS} and  \cite[Theorem 1]{Jon81} respectively. For a specific uniform domain  on the Sierpi\'nski carpet, the result (iii) was obtained by Hino \cite[Proposition 4.15]{Hin13}.

Since the domain of the Dirichlet form for reflected diffusion is the Sobolev space $W^{1,2}(U)$, it suggests that the extension problem for the Sobolev space $W^{1,2}$  could be relevant for studying reflected diffusion.
To explain this connection to the extension problem, we recall that sub-Gaussian   heat kernel estimate for symmetric diffusion can be characterized by the volume doubling property and functional inequalities such as the Poincar\'e inequality and cutoff Sobolev inequality. This characterization goes back to the works of Grigor'yan\cite{Gri} and Saloff-Coste\cite{Sal92} in the Gaussian case and Barlow and Bass in the sub-Gaussian case \cite{BB04} with many other important contributions \cite{Stu96,BBK, AB,GHL15} (see Theorem \ref{t:hke}).

While the proof of Poincar\'e inequality for the reflected diffusion follows the same line of reasoning as \cite{GS}, the proof of the cutoff Sobolev inequality for the Dirichlet form corresponding to the reflected diffusion requires new ideas. Indeed, the characterization of Gaussian heat kernel estimate does not require cut-off Sobolev inequality.
We use the extension property to obtain cutoff Sobolev inequality for the Dirichlet form corresponding to the reflected diffusion using the similar property in the larger ambient space (Proposition \ref{p:cs}). 
In other words, we obtain a functional inequality in a domain using the corresponding  inequality in the ambient space using the extension property. Since the cut-off Sobolev has a local and scale-invariant nature, we need to show certain scale-invariant estimates on the extension operator (cf. \eqref{e:ebndloc1}, \eqref{e:ebndloc2} in Theorem \ref{t:extend}).
These \emph{scale-invariant} bounds on the extension operator seem to be new and we believe  is of independent interest.
By scale-invariant bounds, we mean that the $L^2$ norm and energy measure of the extended function on all balls $B(x,r)$ centered in the domain $U$ can be bounded by the corresponding quantities for comparable balls $B(x,Kr) \cap U$ contained in the domain. The constants involved in these bounds are uniform in the location $x$ and scale $r$.
Although our motivation behind obtaining the extension property is to prove heat kernel estimates, there are several works over a long period of time on the extension property for its intrinsic interest and other applications;  cf. \cite{Whi,Cal61,Ste, GV,HeK, Shv,R, HKT} and their references.

The cut-off Sobolev inequality obtained using  scale-invariant bounds on the extension operator along  Poincar\'e inequality for the reflected 
diffusion  generalizes the result of Gyrya and Saloff-Coste to sub-Gaussian heat kernel on uniform domains (Theorem \ref{t:hkeunif}). 
However, drawback of this work compared to \cite{GS} is that we cannot handle \emph{inner} uniform domains because the extension property can fail on such domains. It would therefore be   desirable to develop an intrinsic approach to the proof of cutoff Sobolev inequality that does not rely on the extension property.

Our construction of the extension operator is based on Whitney covers of the domain $U$ and $V= (\overline{U})^c$. 
The use Whitney cover to extend functions has a long history. Constructing differentiable extensions was the original motivation behind Whitney's construction of his eponymous cubes \cite{Whi}. 
This was later adapted by Calderon and Stein \cite{Cal61, Ste} to construct extension of Sobolev functions on Lipschitz domains, and by Jones \cite{Jon81} on locally uniform domains.  Bj\"orn and Shanmugalingam use it to extend Newton-Sobolev functions  on uniform domains in metric spaces satisfying doubling property and  Poincar\'e inequality \cite{BS}.

While our construction of the extension operator is similar to earlier works of Jones \cite{Jon81} and Bj\"orn-Shanmugalingam \cite{BS}, we need a new approach to prove that the extension operator is bounded.
This is because the previous approaches to the Sobolev extension problem  in \cite{Ste, Jon81, BS} relied on point-wise upper bounds on gradient (or upper gradient) of the extension to control the Sobolev norm. However in the setting of  Dirichlet forms such \emph{point-wise estimates on gradient are not meaningful} in general since the energy measure can be singular with respect to the symmetric measure \cite{BST, Hin05, KM20, Kus89}.

To explain the difficulty that arises due to the singularity of energy measure, we recall the definition of energy measure. Given a Dirichlet form
$(\mathcal{E},\mathcal{F})$ on $L^2(X,m)$, the \emph{energy measure}
of a function $f\in\mathcal{F}\cap L^{\infty}(X,m)$ is defined as
the unique Borel measure $\Gamma(f,f)$ on $X$ such that
\[
\int_{X} g \, d\Gamma(f,f) = \mathcal{E}(f,fg)-\frac{1}{2}\mathcal{E}(f^{2},g)\q \mbox{for all $g\in\mathcal{F}\cap C_{c}(X)$.}
\]
For the Brownian motion on $\mathbb{R}^n$, we have $\sE(f,f):= \frac{1}{2} \int_{\mathbb{R}^n} \abs{\grad f}^2(x)\,m(dx)$, where $m$ is the Lebesgue (symmetric) measure for all $f \in W^{1,2}(\mathbb{R}^n)$. By a simple calculation using product rule, the energy measure of any $f \in W^{1,2}(\mathbb{R}^n)$ is given by $\Gamma(f,f)(A) =\frac{1}{2} \int_A \abs{\nabla f}^2(x)\,m(dx)$ for any Borel set $A$. In this case $\Gamma(f,f)$ is absolutely continuous with respect to the symmetric measure but this need not be true in general as mentioned above.

Since our approach is new even in the Euclidean setting, we explain it   on $\mathbb{R}^n$.
Since we cannot rely on a point-wise estimate on gradient, we need to estimate $\int_{\bR^n} \abs{\nabla f}^2(x)\,dx$ \emph{without using point-wise estimates on $\abs{\nabla f}$}, where $f \in W^{1,2}(\mathbb{R}^n)$. 
Our approach is based on a result of Korevaar and Schoen\footnote{A   related expression for Sobolev norm was obtained earlier by Calder\'on \cite{Cal72} using the sharp maximal function \cite[Proposition 3]{HKT}. } \cite[Theorem 1.6.2]{KS}, which implies that for any $f \in W^{1,2}(\bR^n)$, the Dirichlet energy $\int_{\bR^n} \abs{\nabla f}^2(x) dx$ is comparable to 
\[
\limsup_{r\downarrow 0} \int_{\bR^n} \frac{1}{r^{n+2}}\int_{\{y: \abs{y-x}<r\}} \abs{f(x)-f(y)}^2\,dy \,dx.
\]
To estimate the above integral, we introduce a  Poincar\'e type inequality for the extension operator that is new even on $\mathbb{R}^n$ (Proposition \ref{p:eqpi}(a)).
There is a suitable generalization of  \cite[Theorem 1.6.2]{KS} in our framework as shown in Theorem \ref{t:besov} building upon earlier works \cite{Jons,GHL03, KuSt}. Since we need to obtain scale invariant boundedness of the extension operator (see \eqref{e:ebndloc1}, \eqref{e:ebndloc2} in Theorem \ref{t:extend}), we need to obtain estimates for energy measures that do not rely on point-wise estimates on the gradient (see Theorem \ref{t:besov}(c) and \eqref{e:se} in Proposition \ref{p:eqpi}(e)). 

Our estimates on the extension operator imply that the energy measure of any function in the domain of the Dirichlet form vanishes identically on the boundary of any uniform domain (Theorem \ref{t:benergy}). This property is trivial in the Gaussian space-time scaling case because the energy measure is absolutely continuous with respect to the symmetric measure  \cite[Theorem 2.13]{KM20} (see also \cite[Lemma 2.11]{ABCRST}).
It is easy to see that the symmetric measure vanishes on the boundary of any uniform domain, since it is doubling. On the other hand, if the space-time scaling is not Gaussian the energy measure is typically singular with respect to the symmetric measure \cite[Theorem 2.13(b)]{KM20}. Hino shows that the energy measure of any function in the domain of the Dirichlet form is always zero on the outer square boundary of the Sierpi\'nski carpet using a fairly intricate analysis that relies on the self-similarity of the Dirichlet form and symmetries of the carpet \cite[Section 5 and Proposition 4.15]{Hin13}. Our proof is   different and is based on more general principles such as the above-mentioned Poincar\'e inequality on the extension operator (Proposition \ref{p:eqpi}(a)). The vanishing of the energy measure on the boundary of any uniform domain is obtained as a general consequence of sub-Gaussian heat kernel estimates and thereby we obtain a new proof of \cite[Proposition 4.15]{Hin13}. Understanding energy measures has  applications to computing martingale dimension \cite{Hin13} and the attainment problem for conformal walk dimension \cite[\textsection 6]{KM23}. Energy measures are not  well-understood in general \cite[Problems 7.5 amd 7.6]{KM23}, \cite[Conjecture 2.15]{KM20}.

\begin{notation}
	Throughout this paper, we use the following notation and conventions.
	\begin{enumerate}[(i)]
		\item  The symbols $\subset$ and $\supset$ for set inclusion
		\emph{allow} the case of the equality.
		\item The cardinality (the number of elements) of a set $A$ is denoted by $\#A$.
		\item We set $\infty^{-1}:=0$. We write
		$a\vee b:=\max\{a,b\}$, $a\wedge b:=\min\{a,b\}$.
		\item Let $X$ be a non-empty set. We define $\one_{A}=\one_{A}^{X}\in\mathbb{R}^{X}$ for $A\subset X$ by
		\[\one_{A}(x):=\one_{A}^{X}(x):= \begin{cases}
			1 & \mbox{if $x \in A$,}\\
			0 & \mbox{if $x \notin A$.}
		\end{cases} \]
		
		\item We use the notation  $A \lesssim B$ for quantities $A$ and $B$ to indicate the existence of an
		implicit constant $C \ge 1$ depending on some inessential parameters such that $A \le CB$. We write $A \asymp B$, if $A \lesssim B$ and $B \lesssim A$.
		
		\item Let $X$ be a topological space. We set
		$C(X):=\{f\mid\textrm{$f:X\to\mathbb{R}$, $f$ is continuous}\}$ and
		$C_c(X):=\{f\in C(X)\mid\textrm{$X\setminus f^{-1}(0)$ has compact closure in $X$}\}$.
		\item In a metric space $(X,d)$, $B(x,r)$ is the open ball centered at $x \in X$ of radius $r>0$.
		For a subset $A \subset X$, we use the notation $B_A(x,r) := A \cap B(x,r)$ for $x \in X, r>0$.
		\item Given a ball $B:=B_U(x,r)$ (respectively $B:=B(x,r)$) and $K>0$, by $KB$ we denote the ball $B_U(x,Kr)$ (resp. $B(x,Kr)$).We denote the radius of $B$ by $r(B)$.
		\item For a set $A \subset X$, we write $\ol{A}, A^\circ, \partial A= \ol{A} \setminus A^\circ$ to denote its closure, interior and boundary respectively.
		\item For a measure $m$ and $f \in L^1(m)$, we denote by $\supp_m(f)$ the support of the measure $A \mapsto \int_A f\,dm$.
	\end{enumerate}
\end{notation}

\section{Framework and Main results}
In order to state the main results, we recall the definitions of doubling measure, uniform domains, Dirichlet form, energy measure and sub-Gaussian heat kernel estimates.
\subsection{Doubling metric space and doubling measures}
Throughout this paper, we consider a  metric space $(X,d)$ in which
$B(x,r):=B_{d}(x,r):=\{y\in X \mid d(x,y)<r\}$ is relatively compact
(i.e., has compact closure) for any $(x,r)\in X\times(0,\infty)$,
and a Radon measure $m$ on $X$ with full support, i.e., a Borel measure
$m$ on $X$ which is finite on any compact subset of $X$ and strictly positive on any
non-empty open subset of $X$. Such a triple $(X,d,m)$ is referred to as a \emph{metric measure space}.
We set $\diam(A):=\sup_{x,y\in A}d(x,y)$ for $A\subset X$ ($\sup\emptyset:=0$).

In much of this work, we will be in the setting for a doubling metric space equipped with a doubling measure.
\begin{definition}\label{d:mdoubling}
	A metric $d$ on $X$ is said to be a doubling metric (or equivalently, $(X,d)$ is a doubling metric space), if there exists $N \in \bN$ such that every ball $B(x,R)$ can be covered by $N$ balls of radii $R/2$ for all $x \in X, R>0$. 
\end{definition}
Next, we recall the closely related notion of doubling measures on subsets of $X$.
\begin{definition} \label{d:doubling}
	Let $(X,d)$ be a metric space and let $V \subset X$.
	We say that a Borel measure $m$ is \emph{doubling on $V$} if $m(V) \neq 0$ and there exists $D_0 \ge 1$ such that 
	$$  m(B(x,2r) \cap V) \le D_0 m(B(x,r)\cap V), \quad
	\mbox{for all $x \in V$ and all $r>0$.}$$
	We say that a non-zero Borel measure $m$ on $X$ is \emph{doubling}, if $m$ is  doubling  on $X$.
\end{definition}
The basic relationship between these notions is that if there is a (non-zero) doubing measure on a metric space $(X,d)$, then $(X,d)$ is a doubling metric space. Conversely, every complete doubling metric space admits a doubling measure \cite[Chapter 13]{Hei}. 
\subsection{Uniform  domains}
We recall the definition of a \emph{length uniform domain} and \emph{uniform domain}.  There are different definitions of uniform domains in the literature \cite{Mar,Vai}. 
We note that our definition  of length uniform domain is what is usually called a uniform domain.

Let $U \subset X$ be an open set. A \emph{curve in $U$} is a continuous function  $\gamma:[a,b] \to U$  such that $\gamma(0)=x, \gamma(b)=y$. We sometimes identify $\gamma$ with it its image $\gamma([a,b])$, so that $\gamma \subset U$. The \emph{length} of a curve $\gamma: [a,b] \to X$ is $$\ell(\gamma):= \sup \left\{ \sum_{i=0}^{n-1} d(\gamma(t_i),\gamma(t_{i+1})): a \le t_0 < t_1 \ldots < t_n \le b\right\}.$$ Define
\begin{equation}\label{e:dU}
	\delta_U(x):= \operatorname{dist}(x, X \setminus U )= \inf \{d(x,y):y \in X \setminus U\}, \quad \mbox{for all $x \in U$}.
\end{equation}
\begin{definition} \label{d:uniform}
	Let $A \ge 1$. A connected, non-empty, proper open set $U \subsetneq X$ is said to be a \textbf{length $A$-uniform domain} if   for every pair of points $x,y \in U$, there exists a curve $\gamma$ in $U$ from $x$ to $y$ such that its length $\ell(\gamma) \le A d(x,y)$ and for all $z \in \gamma$, 
	\[
	\delta_U(z) \ge A^{-1} \min \left( \ell(\gamma_{x,z}), \ell(\gamma_{z,y}) \right),
	\]
	where $\gamma_{x,z}, \gamma_{z,y}$ are subcurves of $\gamma$ from $x$ and $z$ and from $z$ to $y$ respectively. Such a curve $\gamma$ is called a \emph{length $A$-uniform curve}.
	
	A connected, non-empty, proper open set $U \subsetneq X$ is said to be a \textbf{$A$-uniform domain} if   for every pair of points $x,y \in U$, there exists a curve $\gamma$ in $U$ from $x$ to $y$ such that its diameter $\diam(\gamma) \le A d(x,y)$ and for all $z \in \gamma$, 
	\[
	\delta_U(z) \ge A^{-1} \min \left(d(x,z), d(y,z) \right).
	\]
	Such a curve $\gamma$ is called a \emph{$A$-uniform curve}.
\end{definition}
Since every  length $A$-uniform curve is    $A$-uniform curve, every length uniform domain is  a uniform domain. 
The converse fails in general because a snowflake\footnote{replacing the metric $d$ with $d^\alpha$ for some $\alpha \in (0,1)$} transform of the metric   makes it impossible for non-trivial rectifiable curves to exist, but on the other hand, the property of being a uniform domain is preserved under such transformation. 
In the Euclidean space (with the Euclidean distance) these two definitions coincide due to an argument of Marito and Sarvas \cite[Lemma 2.7]{MS}.  More generally, these two notions of uniform domains  coincide in any complete length space satisfying the metric doubling property. This can been seen by following the argument  in \cite[Proposition 3.3]{GS}, and \cite[Lemma 2.7]{MS}.  
Our reason to choose this particular definition of uniform domains is that the property of {\em being a uniform domain is   preserved under a quasisymmetric change of metric}. Since quasisymmetric changes of metric has recently played an important role in the understanding of heat kernel estimates and Harnack inequalities, we choose the weaker definition \cite{Kig12,BM,KM23}.

We discuss a few examples of uniform domains.
\begin{enumerate}[(i)]
	\item In $\bR^2$, there is a rich family of uniform domains that arise from quasiconformal mappings. By \cite[Theorem 3.4.5]{GH} every quasidisk is a uniform domain. In particular, von Koch snowflake domain and its variants are uniform domains \cite[Theorem 7.5.2]{GH}.
	\item Half-spaces, balls and cubes in the Heisenberg group equipped with the Carnot metric are uniform domains \cite[p. 132]{GS}.
	\item The complement of the outer square boundary and the domain formed by removing the bottom line of the Sierpi\'nski carpet are uniform domains \cite[Proposition 4.4]{Lie} \cite[Proposition 2.4]{CQ}.
	\item A large family of uniform domains is due to a construction of T.~Rajala \cite{Raj}. We say that a metric space $(X,d)$ is \emph{quasiconvex} if there exists $C_q \in (1,\infty)$ such that for any $x,y \in X$, there is a curve $\gamma$ connecting $x,y$ such that $\ell(\gamma) \le C_q d(x,y)$. For any quasiconvex, doubling metric space $(X,d)$, for any bounded domain $\Omega \subset X$  and for any $\epsilon >0$, there exist \emph{uniform} domains
	$\Omega_i$ and $\Omega_o$ such that $\Omega_i \subset \Omega \subset \Omega_o$ and
	$$\Omega_o \subset [\Omega]_\epsilon, \quad  \Omega_i^c \subset [\Omega^c]_\epsilon, \quad \mbox{where $[A]_\epsilon$ denotes the $\epsilon$-neighborhood of $A$.}$$
	Informally, every bounded domain can be $\epsilon$-approximated by uniform domains from outside and inside for any $\epsilon>0$.
\end{enumerate}

\subsection{Metric measure Dirichlet space and energy measure}
Let $(\mathcal{E},\mathcal{F})$ be a \emph{symmetric Dirichlet form} on $L^{2}(X,m)$;
that is, $\mathcal{F}$ is a dense linear subspace of $L^{2}(X,m)$, and
$\mathcal{E}:\mathcal{F}\times\mathcal{F}\to\mathbb{R}$
is a non-negative definite symmetric bilinear form which is \emph{closed}
($\mathcal{F}$ is a Hilbert space under the inner product $\mathcal{E}_{1}:= \mathcal{E}+ \langle \cdot,\cdot \rangle_{L^{2}(X,m)}$)
and \emph{Markovian} ($f^{+}\wedge 1\in\mathcal{F}$ and $\mathcal{E}(f^{+}\wedge 1,f^{+}\wedge 1)\leq \mathcal{E}(f,f)$ for any $f\in\mathcal{F}$).
Recall that $(\mathcal{E},\mathcal{F})$ is called \emph{regular} if
$\mathcal{F}\cap \contfunc_{\mathrm{c}}(X)$ is dense both in $(\mathcal{F},\mathcal{E}_{1})$
and in $(\contfunc_{\mathrm{c}}(X),\|\cdot\|_{\mathrm{sup}})$, and that
$(\mathcal{E},\mathcal{F})$ is called \emph{strongly local} if $\mathcal{E}(f,g)=0$
for any $f,g\in\mathcal{F}$ with $\supp_{m}[f]$, $\supp_{m}[g]$ compact and
$\supp_{m}[f-a\one_{X}]\cap\supp_{m}[g]=\emptyset$ for some $a\in\mathbb{R}$. Here
$\contfunc_{\mathrm{c}}(X)$ denotes the space of $\mathbb{R}$-valued continuous functions on $X$ with compact support, and
for a Borel measurable function $f:X\to[-\infty,\infty]$ or an
$m$-equivalence class $f$ of such functions, $\supp_{m}[f]$ denotes the support of the measure $|f|\,dm$,
i.e., the smallest closed subset $F$ of $X$ with $\int_{X\setminus F}|f|\,dm=0$,
which exists since $X$ has a countable open base for its topology; note that
$\supp_{m}[f]$ coincides with the closure of $X\setminus f^{-1}(0)$ in $X$ if $f$ is continuous.
The pair $(X,d,m,\mathcal{E},\mathcal{F})$ of a metric measure space $(X,d,m)$ and a strongly local,
regular symmetric Dirichlet form $(\mathcal{E},\mathcal{F})$ on $L^{2}(X,m)$ is termed
a \emph{metric measure Dirichlet space}, or an \emph{MMD space} in abbreviation. By Fukushima's theorem about regular Dirichlet forms, the MMD space corresponds to a symmetric Markov processes on $X$ with continuous sample paths \cite[Theorem 7.2.1 and 7.2.2]{FOT}.
We refer to \cite{FOT,CF} for details of the theory of symmetric Dirichlet forms.

We recall the definition of energy measure.
Note that $fg\in\mathcal{F}$
for any $f,g\in\mathcal{F}\cap L^{\infty}(X,m)$ by \cite[Theorem 1.4.2-(ii)]{FOT}
and that $\{(-n)\vee(f\wedge n)\}_{n=1}^{\infty}\subset\mathcal{F}$ and
$\lim_{n\to\infty}(-n)\vee(f\wedge n)=f$ in norm in $(\mathcal{F},\mathcal{E}_{1})$
by \cite[Theorem 1.4.2-(iii)]{FOT}.

\begin{definition}\label{d:EnergyMeas}
	Let $(X,d,m,\mathcal{E},\mathcal{F})$ be an MMD space.
	The \emph{energy measure} $\Gamma(f,f)$ of $f\in\mathcal{F}$
	associated with $(X,d,m,\mathcal{E},\mathcal{F})$ is defined,
	first for $f\in\mathcal{F}\cap L^{\infty}(X,m)$ as the unique ($[0,\infty]$-valued)
	Borel measure on $X$ such that
	\begin{equation}\label{e:EnergyMeas}
		\int_{X} g \, d\Gamma(f,f)= \mathcal{E}(f,fg)-\frac{1}{2}\mathcal{E}(f^{2},g) \qquad \textrm{for all $g \in \mathcal{F}\cap \contfunc_{\mathrm{c}}(X)$,}
	\end{equation}
	and then by
	$\Gamma(f,f)(A):=\lim_{n\to\infty}\Gamma\bigl((-n)\vee(f\wedge n),(-n)\vee(f\wedge n)\bigr)(A)$
	for each Borel subset $A$ of $X$ for general $f\in\mathcal{F}$. 
\end{definition}
Associated with a Dirichlet form is a \textbf{strongly continuous contraction semigroup} $(P_t)_{t > 0}$; that is, a family of symmetric bounded linear operators $P_t:L^2(X,m) \to L^2(X,m)$ such that
\[
P_{t+s}f=P_t(P_sf), \q \norm{P_tf}_2 \le \norm{f}_2, \q \lim_{t \downarrow 0} \norm{P_tf-f}_2 =0, 
\]
for all $t,s>0, f \in L^2(X,m)$. In this case, we can express $(\sE,\sF)$ in terms of the semigroup as
\begin{equation} \label{e:semigroup}
	\sF=\{f \in L^2(X,m): \lim_{t \downarrow 0}  \frac{1}{t}\langle f - P_t f, f \rangle < \infty \}, \q \sE(f,f)= \lim_{t \downarrow 0}  \frac{1}{t}\langle f - P_t f, f \rangle, 
\end{equation}
for all $f \in \sF$, where $\langle \cdot, \cdot \rangle$ denotes the inner product in $L^2(X,m)$ \cite[Theorem 1.3.1 and Lemmas 1.3.3 and 1.3.4]{FOT}. It is known that  $P_t$ restricted to  $L^2(X,m) \cap L^\infty(X,m)$ extends to a linear contraction on $L^\infty(X,m)$ \cite[pp. 5 and 6]{CF}. If $P_t 1 = 1$ ($m$ a.e.) for all $t >0$, we say that the corresponding Dirichlet form $(\sE,\sF)$ is \emph{conservative}.
\begin{definition}[Local Dirichlet space and its energy measure] \label{d:local}
	For an open set $U \subset X$ of an MMD space $(X,d,m,\sE,\sF)$, we define the local Dirichlet space $\mathcal{F}_{\loc}(U)$ as
	\begin{equation}\label{e:Floc}
		\mathcal{F}_{\loc}(U) := \Biggl\{ f \Biggm|
		\begin{minipage}{285pt}
			$f$ is an $m$-equivalence class of $\mathbb{R}$-valued Borel measurable functions
			on $U$ such that $f \one_{V} = f^{\#} \one_{V}$ $m$-a.e.\ for some $f^{\#}\in\mathcal{F}$
			for each relatively compact open subset $V$ of $U$
		\end{minipage}
		\Biggr\} 
	\end{equation}
	and the energy measure $\Gamma_U(f,f)$ of $f\in\mathcal{F}_{\loc}(U)$ associated with
	$(X,d,m,\mathcal{E},\mathcal{F})$ is defined as the unique Borel measure on $U$
	such that $\Gamma_U(f,f)(A)=\Gamma(f^{\#},f^{\#})(A)$ for any relatively compact
	Borel subset $A$ of $U$ and any $V,f^{\#}$ as in \eqref{e:Floc} with $A\subset V$;
	note that $\Gamma(f^{\#},f^{\#})(A)$ is independent of a particular choice of such $V,f^{\#}$. We define
	\begin{equation} \label{e:defFU}
		\mathcal{F}(U) := \{f \in 	\mathcal{F}_{\loc}(U): \int_U f^2\,dm + \int_U \Gamma_U(f,f) < \infty \},
	\end{equation}
	and the bilinear form $(\sE_U,\sF(U))$  as 
	\begin{equation} \label{e:neumann}
		\sE_U(f,f)= \int_U \Gamma_U(f,f), \q \mbox{for all $f \in \sF(U)$.}
	\end{equation}
\end{definition}
The form $(\sE_U,\sF(U))$ need not be a regular Dirichlet form on $L^2(\overline{U},m)$ in general. A   sufficient condition for  $(\sE_U,\sF(U))$ to be a regular Dirichlet form on $L^2(\overline{U},m)$ is given in Lemma \ref{l:regular}.  If $(\sE_U,\sF(U))$ is a regular Dirichlet form on $L^2(\overline{U},m)$, then this is the Dirichlet form corresponding to the \emph{reflected diffusion} on $\overline{U}$.

\subsection{Sub-Gaussian heat kernel estimates}
Let $\Psi:(0,\infty) \to (0,\infty)$ be a continuous increasing bijection of $(0,\infty)$ onto itself, such that for all $0 < r \le R$,
\begin{equation}  \label{e:reg}
C^{-1} \left( \frac R r \right)^{\beta_1} \le \frac{\Psi(R)}{\Psi(r)} \le C \left( \frac R r \right)^{\beta_2}, 
\end{equation}
for some constants $1 < \beta_1 < \beta_2$ and $C>1$. 
If necessary, we extend $\Psi$ by setting $\Psi(\infty) = \infty$.
Such a function $\Psi$ is said to be a \textbf{scale function}.
For $\Psi$ satisfying \eqref{e:reg}, we define
\begin{equation} \label{e:defPhi}
\Phi(s)= \sup_{r>0} \left(\frac{s}{r}-\frac{1}{\Psi(r)}\right).
\end{equation}

\begin{definition}[\hypertarget{hke}{$\on{HKE(\Psi)}$}]\label{d:HKE}
	Let $(X,d,m,\mathcal{E},\mathcal{F})$ be an MMD space, and let $\set{P_t}_{t>0}$
	denote its associated Markov semigroup. A family $\set{p_t}_{t>0}$ of non-negative
	Borel measurable functions on $X \times X$ is called the
	\emph{heat kernel} of $(X,d,m,\mathcal{E},\mathcal{F})$, if $p_t$ is the integral kernel
	of the operator $P_t$ for any $t>0$, that is, for any $t > 0$ and for any $f \in L^2(X,m)$,
	\[
	P_t f(x) = \int_X p_t (x, y) f (y)\, dm (y) \qquad \mbox{for $m$-almost all $x \in X$.}
	\]
	We say that $(X,d,m,\mathcal{E},\mathcal{F})$ satisfies the \textbf{heat kernel estimates}
	\hyperlink{hke}{$\on{HKE(\Psi)}$}, if there exist $C_{1},c_{1},c_{2},c_{3},\delta\in(0,\infty)$
	and a heat kernel $\set{p_t}_{t>0}$ such that for any $t>0$,
	\begin{align}\label{e:uhke}
		p_t(x,y) &\le \frac{C_{1}}{m\bigl(B(x,\Psi^{-1}(t))\bigr)} \exp \left( -c_{1} t \Phi\left( c_{2}\frac{d(x,y)} {t} \right) \right)
		\qquad \mbox{for $m$-a.e.~$x,y \in X$,}\\
		p_t(x,y) &\ge \frac{c_{3}}{m\bigl(B(x,\Psi^{-1}(t))\bigr)}
		\qquad \mbox{for $m$-a.e.~$x,y\in X$ with $d(x,y) \le \delta\Psi^{-1}(t)$,}
		\label{e:nlhke}
	\end{align}
	where $\Phi$ is as defined in \eqref{e:defPhi}.
\end{definition}

\subsection{Main results}
We are now ready to state the main results. The setting is   an MMD space that satisfies sub-Gaussian heat kernel bounds.
Our first result is the existence of a bounded extension operator. This operator is bounded globally (see \eqref{e:bndglob1}, \eqref{e:bndglob2}) and also satisfies good \emph{scale-invariant} bounds (see \eqref{e:ebndloc1}, \eqref{e:ebndloc2})
\begin{theorem}[Extension property] \label{t:extend} 
	Let $(X,d,m,\sE,\sF)$ be an MMD space that satisfies the heat kernel estimate \hyperlink{hke}{$\on{HKE(\Psi)}$} for some scale function $\Psi$ and let $m$ be a doubling measure.  Let $U$ be a uniform domain $U$ and let $(\sE_U,\sF(U))$ denote the bi-linear form in Definition \ref{d:local}. There is a linear operator $E: \sF(U) \to \sF$ such that the restriction of $E(f)$ to $U$ is $f$ for all $f \in \sF$ (that is, $E$ is an extension operator).
	Furthermore, there exist  $C,K \in (1,\infty), c \in (0,1)$ such that for all $x \in  \overline{U}$, and $f \in \sF(U)$,  we have
	\begin{align}
		\label{e:ebndloc1}				\Gamma(E(f),E(f))(B(x,r)) &\le C \Gamma_U(f,f)(B_U(x,K r)),  \quad\mbox{for all  $0 <r< c \diam(U)$;}\\  
		\label{e:ebndloc2}	 \int_{B(x,r)} \abs{E(f)}^2\,dm &\le C \int_{B_U(x,Kr)} f^2\,dm \q \mbox{for all  $r>0$;}\\
		\label{e:bndglob1}		\sE(E(f),E(f))  &\le C \left(\sE_U(f,f) + \frac{1}{\Psi(\diam(U))} \int_U f^2 \,dm \right); \\
		\int_X \abs{Ef}^2 \,dm &\le C \int_U f^2 \, dm. \label{e:bndglob2}
	\end{align}
	Here $\Gamma, \Gamma_U$ denote the energy measures of $(\sE,\sF)$ and $(\sE_U,\sF(U))$ respectively. 
\end{theorem}
In \eqref{e:bndglob1} above, we interpret $\frac{1}{\Psi(\infty)}=0$ in case $\diam(U)= \infty$.

Our second main result is that the reflected diffusion on any uniform domain satisfies a sub-Gaussian heat kernel estimate and that the Dirichlet form approach defines  a symmetric Markov process on $\overline{U}$. 
\begin{theorem}[Heat kernel estimate  for reflected diffusion] \label{t:hkeunif}
	Let $(X,d,m,\sE,\sF)$ be an MMD space that satisfies the heat kernel estimate \hyperlink{hke}{$\on{HKE(\Psi)}$} for some scale function $\Psi$ and let $m$ be a doubling measure. Then for any uniform domain $U$, the bi-linear form $(\sE_U,\sF(U))$ is a strongly-local regular Dirichlet form on $L^2(\overline{U},m)$. Moreover, the corresponding MMD space $(\overline{U},d,m,\sE_U,\sF(U))$  satisfies the heat kernel estimate  \hyperlink{hke}{$\on{HKE(\Psi)}$}.
\end{theorem} 
Finally, we show that the energy measure of any function vanishes on the boundary of any uniform domain.
\begin{theorem}[Energy measure of the boundary] \label{t:benergy}
	Let $(X,d,m,\sE,\sF)$ be an MMD space that satisfies the heat kernel estimate \hyperlink{hke}{$\on{HKE(\Psi)}$} for some scale function $\Psi$ and let $m$ be a doubling measure. Then for any uniform domain $U$ and any $f \in \sF$, we have 
	\[
	\Gamma(f,f)(\partial U) =0,
	\]
	where $\Gamma(f,f)$ denotes the corresponding energy measure.
\end{theorem}
We briefly mention a  probabilistic consequence of the above property of the energy measure. Let $\mu$ be a smooth measure whose quasi-support is   $\partial U$. In \cite{KM23+},  Theorem \ref{t:benergy} is used to show  that the trace process corresponding to $\mu$ on $\partial U$ is a pure jump process  by \cite[Theorems 5.2.2, 5.2.15, and Corollary 5.6.1]{CF}. Hence Theorem \ref{t:benergy} is a starting point to   study    jump process on $\partial U$ that is a trace of reflected diffusion on $\overline{U}$. We obtain heat kernel estimates for the trace (jump) process on the boundary of reflected diffusion on uniform domains in  \cite{KM23+}. 

\subsection{Outline of the work}
The rest of the paper is organized as follows.
In \textsection \ref{s:geometry}, we recall some useful facts about geometry of Whitney covers, uniform domains and Whitney cover of a uniform domain. The main result in \textsection \ref{s:geometry} is Proposition \ref{p:reflect} which provides a `reflection map' that maps   Whitney cover of $V= (\overline{ U})^c$ to  Whitney cover of $U$ at all scales less than diameter of $U$. This reflection of Whitney balls is used to define the extension map from $\sF(U)$ to $\sF$ in \textsection \ref{ss:extend} for any  uniform domain $U$. In \textsection \ref{s:sghke}, we recall the simple result that extension property implies the regularity  of the Dirichlet form corresponding to reflected diffusion on the closure of the domain (Lemma \ref{l:regular}). We recall the characterization of heat kernel estimates using functional inequalities in Theorem \ref{t:hke}.  In this setting, there is a  Koreevar-Schoen type estimate for the Dirichlet energy that is shown in Theorem \ref{t:besov}.  After these somewhat lengthy preparations,  we define the extension map using the reflection map of Whitney balls in \textsection \ref{s:extend}. We obtain a Poincar\'e inequality on the local Dirichlet space corresponding the uniform domain in \textsection \ref{s:poin}.
The boundedness of the extension map in $L^2$ is fairly easy to establish (Lemma \ref{l:l2b}). The heart of the work is \textsection \ref{s:heart} where bounds on energy of the extended function is obtained. A Poincar\'e-type inequality for the extended function (Proposition \ref{p:eqpi}(a)) along with Theorem \ref{t:besov} is used to obtain bounds on the energy (and energy measure) of the extended function. These bounds on energy measure in Proposition \ref{p:eqpi} along with Lemma \ref{l:l2b} implies extension property of uniform domain with scale-invariant bounds stated in Theorem \ref{t:extend}.
Using the estimates of energy measure for the extended function obtained in Proposition \ref{p:eqpi}, we complete the proof of  Theorem \ref{t:benergy} in \textsection \ref{ss:energyb}. Finally in \textsection \ref{s:hke}, we introduce a simpler version of cutoff Sobolev inequality (Definition \ref{d:css}) and show that is equivalent to earlier version. We obtain the simplified version of cutoff Sobolev inequality using the bounds on extension operator in Theorem \ref{t:extend} and the cutoff Sobolev inequality in the ambient space in Proposition \ref{p:cs}. This along with Poincar\'e inequality (Theorem \ref{t:pi}) and the characterization of sub-Gaussian heat kernel bounds (Theorem \ref{t:hke}) is used to conclude the proof of Theorem \ref{t:hkeunif}. The key ingredients  are outlined in Figure \ref{f:outline}.
\begin{figure}[h] \label{f:outline}
	\caption{Outline of the work}
	\centering
	\includegraphics[width=0.99\textwidth]{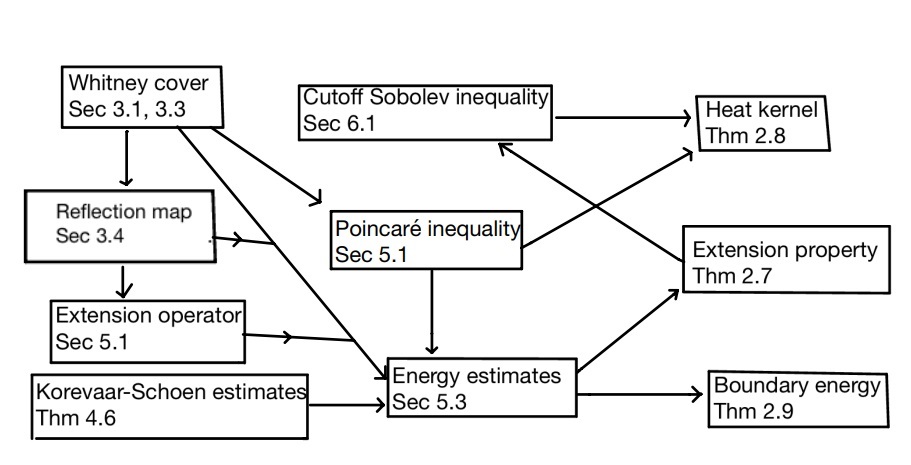}
\end{figure}
\section{Whitney covering and uniform domains}  \label{s:geometry}
We recall geometric properties of Whitney cover, uniform domains, and Whitney covers on a uniform domain.  Finally, we introduce a reflection map similar to that of Jones' reflection of Whitney cubes in $\mathbb{R}^n$ (Proposition \ref{p:reflect}).
\subsection{Whitney covering}

We recall the notion of a $\epsilon$-Whitney cover from \cite[Definition 3.16]{GS}.
\begin{definition} \label{d:whitney}
	Let $\epsilon \in (0,1/2)$ and $U \subsetneq X$. 
	We say a collection of balls $\mathfrak{R}:= \set{B_U(x_i,r_i): x_i \in U, r_i >0, i \in I}$ is an $\epsilon$-Whitney cover if it satisfies the following properties:
	\begin{enumerate}[(i)]
		\item The collection of sets $\{B_U(x_i,r_i), i \in I \}$ are pairwise disjoint.
		\item The radii $r_i$ satisfy $$r_i= \frac{\epsilon}{1+\epsilon}\delta_U(x_i), \quad \mbox{for all $i \in I$.}$$
		\item $\bigcup_{i \in I} B_U(x_i, K_\eps r_i) = U$, where $K_\epsilon=2 (1+\epsilon) \in (2,3)$.
	\end{enumerate}
	Note that since  $K_\epsilon < 3$   by (iii), we have $$\bigcup_{i \in I} B_U(x_i, 3 r_i) = U.$$
\end{definition}
The existence of such a Whitney cover follows from Zorn's lemma as we recall below. We also record a few elementary  geometric properties of the Whitney cover for future use.
\begin{proposition} \label{p:whitney}
	Let $U \subsetneq X$ be a non-empty open set and $\epsilon \in (0,1/2)$. 
	\begin{enumerate}[(a)]
		\item There exists an $\epsilon$-Whitney cover $\mathfrak{R}= \set{B_U(x_i,r_i): x_i \in U, r_i >0, i \in I}$ of $U$ such that the following properties hold.
		\item (distance to boundary) For any $B_U(x_i,r_i) \in \mathfrak{R}$ and for any $y \in B_U(x_i,3r_i)$, we have
		\begin{equation} \label{e:whitb}
		\frac{1-2\epsilon}{1+\epsilon}\delta_U(x_i) <\delta_U(y) < \frac{1+4\epsilon}{1+\epsilon}  \delta_U(x_i)
		\end{equation}
		\item (radius comparison) For any $\lambda>1,\epsilon \in (1,1/2)$ such that $(\lambda-1)\epsilon<1$, for any $\epsilon$-Whitney cover  $\mathfrak{R}= \set{B_U(x_i,r_i): x_i \in U, r_i >0, i \in I}$ of $U$, and for any $i,j \in I$ such that $B_U(x_i,\lambda r_i) \cap B_U(x_j,\lambda r_j) \neq \emptyset$, we have 
		\begin{equation} \label{e:whitc}
		\frac{1-(\lambda-1)\epsilon}{1+(\lambda+1)\epsilon}r_j \le	 r_i \le \frac{1+(\lambda+1)\epsilon}{1-(\lambda-1)\epsilon} r_j.
		\end{equation}
		In particular, if $i,j \in I, i \neq j, B_U(x_i,\lambda r_i) \cap B_U(x_j,\lambda r_j) \neq \emptyset$ implies
		\begin{equation}\label{e:distnei}
		r_i \vee r_j \le	d(x_i,x_j) \le \lambda \left(1+ \frac{1+(\lambda+1)\epsilon}{1-(\lambda-1)\epsilon}\right) (r_i \wedge r_j).
		\end{equation}
		\item (bounded overlap) Let  $(X,d)$ be a doubling metric space and  let $\epsilon$-Whitney cover  $\mathfrak{R}= \set{B_U(x_i,r_i): x_i \in U, r_i >0, i \in I}$ of $U \subsetneq X$. Then there exists a constant $C\in [1,\infty)$ such that
		$$\sum_{i \in I} \one_{B_U(x_i,r_i/\epsilon)} \le C.$$
		\item Let $(X,d)$ is a doubling metric space and $\epsilon \in (0,1/5)$. Then for any $\epsilon$-Whitney cover $\mathfrak{R}$ of $U$  there exists $N \in \bN$ such that  
		\[
		\# \{ B_U(x_i,r_i) \in \mathfrak{R}: B_U(x_i,6r_i) \cap B_U(x,6r) \neq \emptyset \} \le N \q \mbox{for all $B_U(x,r) \in \mathfrak{R}$.}
		\]
	\end{enumerate}
\end{proposition}
\begin{proof}
	\begin{enumerate}[(a)]
		\item Let $\Omega$ denote the partially ordered (by inclusion) set  consisting of collection of balls $\{B_U(x_i,r_i): x_i \in U, r_i >0, i \in I\}$ that satisfies the conditions (i) and (ii) in Definition \ref{d:whitney}. 
		If $\sC$ is a chain, then it easy to see that $\bigcup_{\mathfrak{A} \in \sC} \mathfrak{A} \in \Omega$. 
		So by Zorn's lemma, there exists a maximal element $\mathfrak{R}=\set{B_U(x_i,r_i): x_i \in U, r_i >0, i \in I} \in \Omega$.
		
		Clearly, $\mathfrak{R}$ satisfies properties (i) and (ii) in Definition \ref{d:whitney}. Next, we show that  Definition \ref{d:whitney}(iii) also holds. 
		Suppose to the contrary that (iii) does not hold, then there exists $y \notin U \setminus \bigcup_{i \in I} B_U(x_i,K_\epsilon r_i)$ and hence
		\begin{equation}\label{e:wh1}
			d(x_i,y) \ge K_\epsilon r_i = \frac{K_\epsilon \epsilon}{1+\epsilon} \delta_U(x_i) \quad \mbox{for all $i \in I$.}
		\end{equation}
		Since $\mathfrak{R}$ is maximal, there exists $j \in I$ such that $B(x_j,r_j) \cap B(y,\epsilon \delta_U(y)/(1+\epsilon)) \neq \emptyset$. Therefore by the triangle inequality and Definition \ref{d:whitney}(ii), we have
		\begin{equation}
			\label{e:wh2}
			d(x_j,y) < \frac{\epsilon}{1+\epsilon} \left(\delta_U(x_j)+ \delta_U(y)\right).
		\end{equation}
		By \eqref{e:wh2},
		\begin{equation*}
			\delta_U(y) \le d(y,x_j) + \delta_U(x_j) < \frac{\epsilon}{1+\epsilon}\left(\delta_U(x_j)+ \delta_U(y)\right) + \delta_U(x_j)
		\end{equation*}
		and hence
		\begin{equation}
			\label{e:wh3}
			\delta_U(y) < (1+2 \epsilon) \delta_U(x_j).
		\end{equation}
		Combining \eqref{e:wh1} and \eqref{e:wh2}, 
		$\frac{K_\epsilon \epsilon}{1+\epsilon}  \delta_U(x_j) \le d(x_j,y) < \frac{\epsilon}{1+\epsilon}  \left(\delta_U(x_j)+ \delta_U(y)\right)$
		which implies
		\begin{equation}
			\label{e:wh4}
			(K_\epsilon - 1) \delta_U(x_j) <  \delta_U(y) <  (1+2 \epsilon) \delta_U(x_j).
		\end{equation}
		This yields the desired contradiction since $K_\epsilon = 2(1+\epsilon)$.
		\item 
		Since $y \in B_U(x_i,3r_i)$, we have $$\delta_U(y) < \delta_U(x_i) + 3r_i= \delta_U(x_i)\left(1+ \frac{3 \epsilon}{1+\epsilon}\right)$$ which implies the upper bound on $\delta_U(y)$.
		For the lower bound, we use $$\delta_U(x_i) < \delta_U(y) + 3r_i= \delta_U(y)+ \frac{3 \epsilon}{1+\epsilon} \delta_U(x_i), \quad \delta_U(y) > \frac{1-2\epsilon}{1+\epsilon} \delta_U(x_i).$$
		\item 
		If $B(x_i,\lambda r_i) \cap B(x_j , \lambda r_j) \neq \emptyset$, then by the triangle inequality 
		\begin{align*}
			\delta_U(x_j) \le \delta_U(x_i) + d(x_i,x_j) < \delta_U(x_i) +  \lambda \frac{\epsilon}{1+\epsilon} (\delta_U(x_i)+\delta_U(x_j))
		\end{align*}
		and hence $(1-(\lambda -1)\epsilon)\delta_U(x_j)<(1+(\lambda +1)\epsilon)\delta_U(x_j)$. This is equivalent to \eqref{e:whitc}.  
		The lower bound on $d(x_i,x_j)$ follows from $B(x_i,r_i) \cap B(x_j,r_j)=\emptyset$ while the upper bound follows from $B(x_i,\lambda r_i) \cap B(x_j,\lambda r_j) \neq \emptyset$, the triangle inequality and \eqref{e:whitc}.
		\item Suppose that $y \in B_U(x_i,r_i/\epsilon)$ for some $i \in I$. By the triangle inequality, 
		$$\delta_U(x_i) \le d(x_i,y) +\delta_U(y) < \frac{r_i}{\epsilon} +  \delta_U(y)\le \frac{1}{1+\epsilon}\delta_U(x_i)+\delta_U(y).$$
		Similarly, 
		$$\delta_U(y) < \delta_U(x_i)+ \frac{r_i}{\epsilon} = \delta_U(x_i)   \frac{2+\epsilon}{1+\epsilon} $$
		Therefore 
		\begin{equation} \label{e:wc1}
			\frac{\epsilon}{2+\epsilon} \delta_U(y) < r_i < \delta_U(y)  \quad \mbox{whenever $y \in  B_U(x_i,r_i/\epsilon), i \in I$.}
		\end{equation}
		Therefore the set $A_y:=\{x_i : i \in I, y \in B(x_i,r_i/\epsilon)\}$ is contained in $B_U(y,\delta_U(y)/\epsilon)$ and any two distinct points in $A_y$ is separated by a distance of at least $ {\epsilon} \delta_U(y)/(2+\epsilon)$. The desired conclusion follows from the metric doubling property.
		\item This is an easy consequence of (d) and the metric doubling property. \qedhere
	\end{enumerate}
\end{proof}

\subsection{Basic properties of uniform domains}
An important property of uniform domains is that is satisfies a  {corkscrew condition} whose definition we recall below.
\begin{definition}\label{d:corkscrew}
	Let $V \subset X$. We say that $V$ satisfies the \emph{corkscrew condition} if there exists $\epsilon>0$ such that for all $x \in \overline{V}$ and $0 < r \le \diam(V)$, the set $B(x,r) \cap V$ contains a ball of with radius $\epsilon r$. 
\end{definition}
Every uniform domains satisfies the corkscrew condition as we recall now. 
The same argument presented in \cite[Lemma 4.2]{BS} for length uniform domains also shows that every   $A$-uniform domain (with our weaker definition of uniform domains) also satisfies the corkscrew condition.
\begin{lemma} \cite[Lemma 4.2]{BS} \label{l:corkscrew}
	Let $U \subsetneq X$ be an  $A$-uniform domain. For any $x \in \overline{U}, r>0$ such that $U \setminus B(x,r) \neq \emptyset$ (in particular, if $r<\diam(U,d)/2$), there exists a ball $B(y,r/(3A)) \subset U \cap B(x,r)$ with radius $r/(3A)$. 
\end{lemma}
The doubling property of a measure $m$  is preserved under restriction to uniform domains. This is the content of the following lemma.
\begin{lemma}\cite[Theorem 2.8]{BS} \label{l:doubling}
	Let $m$ be a doubling measure on $X$ and let $U \subsetneq X$ be a non-empty uniform domain. Then 
	$$m(\partial U)=0$$ and
	$m$ is doubling on $\overline{U}$ and doubling on $U$.
\end{lemma}
\begin{proof}
	By the corkscrew condition (Lemma \ref{l:corkscrew}) and \cite[Theorem 2.8]{BS} we have that $m$ is doubling on $\overline{U}$ and on $U$. 
	
	To see that $m(\partial U)=0$, note that by  the corkscrew condition and the doubling property on $\overline{U}$ we have
	\[
	\limsup_{r \downarrow 0} \fint_{B(x,r)} \one_{\partial U}(y)\,m(dy) \le 1- \liminf_{r \downarrow 0}  \frac{m( U \cap B(x,r) )}{m(B(x,r))}< 1, 
	\]
	for all $x \in \partial U$. By the Lebesgue differentiation theorem \cite[Theorem 1.8]{Hei}, we conclude that $m(\partial U)=0$.
\end{proof}

We recall some  well known properties of  doubling measures on a metric space $(X,d)$.
\begin{lemma} \label{l:vd}
	Let $m$ be a doubling measure on $V$ where $V \subset X$ with doubling constant $D_0$ as given in Definition \ref{d:doubling}. Then 
	\begin{equation}\label{e:vd}
	m(V \cap B(x,s)) \le D_0^2 \left( \frac{d(x,y)+s}{r}\right)^\alpha m(V\cap B(y,r)), \quad \mbox{for all $x \in V, 0<r<s<\infty$,}
	\end{equation}
	where $\alpha=\log_2 D_0$.
	If $m(V)>0$, then the metric space $(V,d)$ satisfies the metric doubling; that is there exists $N \in \bN$ such that every ball $B_V(x,r)$ for $x \in V, r>0$ can be covered by at most $N$ balls of radii $r/2$.
\end{lemma}

\subsection{Whitney cover on a uniform domain}

Let $U \subset X$ be a  $A$-uniform domain for some $A \ge 1$ and let $\mathfrak{R}$ be an $\epsilon$-Whitney cover of $U$ for some $\epsilon \in (0,1/2)$.  For any ball $B_U(x,r)$, we define
\begin{equation}
	\label{e:defRB}
	\mathfrak{R}(B_U(x,r)) = \set{B_U(x_i,r_i) \in \mathfrak{R}: B_U(x_i,3r_i) \cap B_U(x,r) \neq \emptyset}.
\end{equation}
We think of $\mathfrak{R}(B_U(x,r))$ as the Whitney balls near $B_U(x,r)$. In the following lemma we show some basic properties of $\mathfrak{R}(B_U(x,r))$.
\begin{lemma}\label{l:central}
	Let $U \subsetneq X$ be a  $A$-uniform domain for some $A \ge 1$.
	Suppose that $\mathfrak{R}= \set{ B_U(x_i,r_i): i \in I}$ be an $\epsilon$-Whitney cover of $U$ for some $\eps \in (0,1/14)$.
	Let $B_U(x,r)$ be a ball such that $x \in \overline{U},   \delta_U(x) \le  r < \diam(U)/2$. Then 
	\begin{equation} B_U(x,r) \subset \bigcup_{B_U(x_i,r_i) \in \mathfrak{R}(B_U(x,r))} B_U(x_i,3r_i) \subset B_U(x,2r),\label{e:incRB} \end{equation}
	and there exists a ball $B_U(x_0,r_0) \in  \mathfrak{R}(B_U(x,r))$ such that  
	\begin{equation}\label{e:bndr0}
		\frac{\eps}{3A(4+\eps)}r \le 	r_0 \le  \frac{2 \eps}{1-2\eps} r. 
	\end{equation}
\end{lemma}
\begin{proof}
	The inclusion $B_U(x,r) \subset \bigcup_{B_U(x_i,r_i) \in \mathfrak{R}(B_U(x,r))} B_U(x_i,3r_i)$ follows from Definition \ref{d:whitney}(iii) and \eqref{e:defRB}.
	
	Since $r \ge \delta_U(x)$ for any $y \in B_U(x,r)$, we have
	\begin{equation} \label{e:rb1}
		\delta_U(y) \le \delta_U(x)+ d(x,y) < 2r \quad \mbox{for any $y \in B_U(x,r)$.}
	\end{equation}
	For any $B(x_i,r_i) \in \mathfrak{R}(B_U(x,r))$ there exists $y_i \in B(x_i, 3r_i) \cap B_U(x,r)$ and hence
	by Proposition \ref{p:whitney}(b), we have $$\delta_U(x_i) < (1+\epsilon) \delta_U(y_i)/(1-2\epsilon).$$
	Combining this with \eqref{e:rb1}, we obtain
	\[
	r_i= \frac{\epsilon}{1+\epsilon} \delta_U(x_i) < \frac{ \epsilon}{1-2\epsilon} \delta_U(y_i) <  \frac{ 2\epsilon}{1-2\epsilon} r.
	\]
	Therefore for any $z \in B_U(x_i, 3r_i)$ with $B_U(x_i,r_i) \in \mathfrak{R}(B_U(x,r))$, we have (using $\epsilon<1/14$)
	\[
	d(z,x) \le d(x,y_i) + d(y_i,x_i) + 3r_i <  d(x,y_i) + 6 r_i < \left(1+ \frac{12 \epsilon}{1-2\epsilon}\right)r<2 r,
	\]
	where $y_i \in  B(x_i, 3r_i) \cap B_U(x,r)$ is as above. This completes the proof of \eqref{e:incRB}.
	
	By Lemma \ref{l:corkscrew}, there exists $B(y,r/3A) \subset B_U(x,r)$. By \eqref{e:incRB}, there exists a ball $B_U(x_0,r_0) \in \mathfrak{R}(B_U(x,r))$ be such that $y \in B_U(x_0,3r_0)$. Therefore
	\[
	\frac{r}{3A} \le \delta_U(y) < 3 r_0 + \delta_U(x_0)= \left(3 + \frac{ 1+\epsilon}{\epsilon}\right) r_0,
	\]
	which is equivalent the lower bound  of $r_0$ in \eqref{e:bndr0}. 
	Since $\delta_U(y) \le \delta_U(x)+ d(x,y) < 2r$, we have
	\[
	\frac{1+\epsilon}{\epsilon} r_0=\delta_U(x_0) \le 3 r_0+ \delta_U(y) < 3 r_0 + 2r,
	\]
	which is equivalent to the upper bound in \eqref{e:bndr0}.
\end{proof}

We recall the notion of a central ball in a Whitney cover (cf. \cite[Definition 3.21]{GS}).
\begin{definition} \cite[Lemma 3.23]{GS}
	Let $U \subsetneq X$ be a  $A$-uniform domain for some $A \ge 1$,
	and let $\mathfrak{R}= \set{ B_U(x_i,r_i): i \in I}$ be an $\epsilon$-Whitney cover of $U$ for some $\epsilon \in (0,1/14)$.
	Let $B_U(x,r)$ such that $x \in \overline U,   \delta_U(x) \le  r < \diam(U)/2$, and let  $B_0=B_U(x_0,r_0) \in \mathfrak{R}(B_U(x,r))$ be any ball that satisfies \eqref{e:bndr0} in Lemma \ref{l:central}. Then we say that $B_0$ is a \emph{central ball} in $\mathfrak{R}(B_U(x,r))$.
\end{definition}
The following is an analogue of \cite[Lemma 3.23]{GS}.
\begin{lemma} \label{l:string}	
	Suppose that $\mathfrak{R}= \set{ B_U(x_i,r_i): i \in I}$ be an $\epsilon$-Whitney cover of $U$ for some $\epsilon \in (0,1/14)$ and let $B_U(x,r), x \in \ol U, \delta_U(x) \le r < \diam(U)/2$. Let $B_0=B_U(x_0,r_0) \in \mathfrak{R}(B_U(x,r))$ be a central ball. For any $D=B_U(x_D,r_D) \in  \mathfrak{R}(B_U(x,r))$, let $\gamma$ be a   $A$-uniform curve from $x_0$ to $x_D$. 
	\begin{enumerate}[(a)]
		\item Then there exists a finite collection of distinct balls $\mathbb{S}(D)= \{B_0^D,B_1^D,\ldots,B_l^D\}$ of length $l=l(D)$ such that $B_0=B_0^D, B_l^D=D$, and
		\begin{equation} \label{e:wu1}
			B_j^D \in \mathfrak{R}, \q  3 B_j^D \cap 3 B_{j-1}^D \neq \emptyset, \q 3 B_j^D \cap \gamma \neq \emptyset, \q \mbox{for all $j=1,\ldots,l$.}
		\end{equation}
		Here $3 B_j^D$ denotes the ball $B_U(x_j^D, 3 r_j^D)$, where $B_j^D=B_U(x_j^D,  r_j^D)$.
		\item For all $j=0,1,\ldots,l$, let $B_j^D=B_U(x_j^D,r_j^D)$ be the balls as given in (a). Then for all $j=0,1,\ldots,l$,
		\begin{equation}\label{e:wu2}
		r_j^D \le \frac{(A(4 \epsilon +1) +1 -2 \epsilon)\epsilon}{(1-2\epsilon)^2} r, \quad B_U(x_j^D,r_j^D) \subset  B_U(x,C_0r),
		\end{equation}
		and
		\begin{equation}\label{e:wu3}
		d(x_j^D,x_l^D) \le C_1 r_j^D, \quad D \subset  B_U(x_j^D,(2C_1+1)r_j^D)
		\end{equation}
		where   $C_0,C_1$ depends only on $A$ and $\epsilon$.
	\end{enumerate}
\end{lemma}
\begin{proof}
	\begin{enumerate}[(a)]
		\item

		By Definition \ref{d:whitney}(iii), the balls $\{B_U(y,3s): B_U(y,s) \in \mathfrak{R}\}$ cover $\gamma$.

		Let the $A$-uniform curve $\gamma$ be parameterized as $\gamma:[a,b] \to U$. We choose finite subcover $\{B_U(y_i,3 s_i): B_U(y_i,s_i) \in \mathfrak{R}, 1 \le i \le N\}$ of $\gamma$ that contains $B_0$ and $D$. We pick $B_0^D=B_0$ and choose $B_j^D$ inductively as follows. Suppose we have choose $B_0^D,\ldots,B_{j-1}^D$ for some $j \in \bN$. If $B_{j-1}^D=D$, then we set $l=j-1$. If $B^{D}_{j-1} \neq D$ and $3B_{j-1}^D \cap 3D \neq \emptyset$, then we set $B_j^D=D$. If   $B^{D}_{j-1} \neq D$ and $3B_{j-1}^D \cap 3D = \emptyset$, then  we set $s= \sup \{ t \in [a,b]: \gamma(t) \in 3B_{j-1}^D \}$ and pick some $1 \le i \le N$ such that $\gamma(s) \in B_U(y_i,3s_i)$ and set $B_{j}^D= B_U(y_i,s_i)$.
		Evidently, this construction satisfies the desired properties in \eqref{e:wu1}
		\item 
		Note that since $B_U(x_D,3r_D)\cap B_U(x,r) \neq \emptyset$ and $r \ge \delta_U(x)$, we have 
		$$\frac{1+\epsilon}{\epsilon} r_D=	\delta_U(x_D) < 3 r_D+\delta_U(x)+r < 3 r_D + 2r,$$ which implies
		\[
		r_D \le \frac{2 \epsilon}{1-2\epsilon} r, \quad \mbox{ and } d(x_D,x) < 3r_D+r < \frac{4 \epsilon+1}{1-2\epsilon} r.
		\]
		Since $\gamma$ is a   $A$-uniform curve between $x$ and $x_D$, we have $$\gamma \subset \ol{B_U}(x,A d(x,x_D)) \subset B_U(x, \frac{A(4\epsilon+1)}{1-2\epsilon} r).$$
		Since $3B_j^D \cap \gamma \neq \emptyset$, we have 
		\begin{equation} \label{e:wu4}
		d(x_j^D,x) < 3 r^D_j + \frac{A(4\epsilon+1)}{1-2\epsilon} r,
		\end{equation}and hence $$\frac{1+\epsilon}{\epsilon} r_j^D= \delta_U(x_j^D) \le \delta_U(x)+d(x_j^D,x)\le r+3 r^D_j + \frac{A(4\epsilon+1)}{1-2\epsilon} r.$$ This implies 
		\[
		r_j^D \le \frac{(A(4 \epsilon +1) +1 -2 \epsilon)\epsilon}{(1-2\epsilon)^2} r,\, B_U(x_j^D,r_j^D) \subset  B_U(x, d(x,x_j^D)+r_j^D) \subset B_U(x,C_0r),
		\]
		where $$C_0=\frac{A(1+4 \epsilon)}{1-2\epsilon}+\frac{4(A(4 \epsilon +1) +1 -2 \epsilon)\epsilon}{(1-2\epsilon)^2}$$ by \eqref{e:wu4} and the upper bound on $r_j^D$ above. 
		This completes the proof of \eqref{e:wu2}.
		
		For the proof of \eqref{e:wu3}, we first show that the second inclusion follows from the first inequality. Note that
		\[
		\frac{1+\epsilon}{\epsilon} r_l^D = \delta_U(x_l^D) \le d(x_l^D,x_j^D)+\delta_U(x_j^D) \le \left(C_1+ \frac{1+\epsilon}{\epsilon}\right) r_j^D.
		\]
		Therefore $ r_l^D < (C_1+1)r_j^D$ and hence
		\[
		B_l^D=B_U(x_l^D,r_l^D) \subset B_U(x_j^D,  r_l^D + d(x_l^D,x_j^D)) \subset  B_U(x_j^D,  (C_1+1)r_j^D +C_1 r_j^D)).
		\]

		By (a), there exists $z_j \in B_U(x_j^D,3r_j^D) \cap \gamma$. 
		By  \eqref{e:incRB}, \eqref{e:wu4} and  \eqref{e:wu2}, we have
		\begin{align} \label{e:wu5}
			d(x_j^D,x_l^D) &\stackrel{\eqref{e:incRB}}{<} d(x_j^D,x) +2r  \stackrel{\eqref{e:wu4}}{<} 2r+ 3r_j^D + \frac{A(4 \epsilon +1)}{(1-2\epsilon)} r \nonumber \\
			&\stackrel{\eqref{e:wu2}}{<} r \left(2+ \frac{3(A(4 \epsilon +1) +1 -2 \epsilon)\epsilon}{(1-2\epsilon)^2}+ \frac{A(4 \epsilon +1)}{(1-2\epsilon)} \right).
		\end{align}
		We consider two cases whether or not $z_j \in B_U(x_0^D,3r_0^D)$.
		If  $$z_j \in B_U(x_0^D,3r_0^D) \cap  B_U(x_j^D,3r_j^D) \cap \gamma,$$ then by Proposition \ref{p:whitney}(c) we obtain
		\begin{equation}\label{e:wu6}
		r_j^D \stackrel{\eqref{e:whitc}}{\ge} \frac{1-2\epsilon}{1+4\epsilon} r_0^D \overset{\eqref{e:bndr0}}{\ge} \frac{(1-2\epsilon)\epsilon}{3A (4+\epsilon)(1+4\epsilon)} r.
		\end{equation}
		Combining \eqref{e:wu5} and \eqref{e:wu6}, we obtain \eqref{e:wu3} in this case.
		
		If $z_j \notin B_U(x_0^D,3r_0^D)$, by Proposition \ref{p:whitney}(b) and the weak $A$-uniformity of $\gamma$, we have 
		\begin{align} \label{e:wu7}
			\frac{1+4 \epsilon}{\epsilon} r_j^D=\frac{1+4\epsilon}{1+\epsilon}\delta_U(x_j^D) &\overset
			{\eqref{e:whitb}}{>} \delta_U(z_j) \ge \frac{1}{A} \min(d(x_0^D,z_j),d(x_l^D,z_j)) \nonumber\\
			&\ge \frac{1}{A} \min(3 r_0^D,d(x_l^D,z_j))\q \mbox{(since $z_j \notin B_U(x_0^D,3 r_0^D)$)} \nonumber \\
			&\ge \frac{1}{A} \min(3 r_0^D,d(x_l^D,x_j^D)-3 r_j^D)  \\
			&\qquad  \qquad \nonumber \mbox{ (since $z_j \in B_U(x_j^D,3 r_j^D)$).}
		\end{align}
		We divide this case into two subcases depending on whether or not $3 r_0^D \le d(x_l^D,x_j^D)-3 r_j^D$. If $3 r_0^D \le d(x_l^D,x_j^D)-3 r_j^D$, we have 
		\begin{equation} \label{e:wu8}
		r_j^D \stackrel{\eqref{e:wu7}}{\ge} \frac{3\epsilon}{A(1+4\epsilon)}  r_0 \stackrel{\eqref{e:bndr0}}{\ge}\frac{\epsilon^2}{A^2(4+\epsilon)(1+4\epsilon)}   r.
		\end{equation}
		In this subcase, we obtain the desired conclusion by \eqref{e:wu8} and \eqref{e:wu5}.
		
		Finally, if $3 r_0^D > d(x_l^D,x_j^D)-3 r_j^D$, we have 
		\begin{equation} \label{e:wu9}
			d(x_l^D,x_j^D) \stackrel{\eqref{e:wu7}}{\le} \left(\frac{A(1+4\epsilon)}{\epsilon}+3\right) r_j^D
		\end{equation}
		which concludes the proof of \eqref{e:wu3}, as it follows from \eqref{e:wu9} along with \eqref{e:wu8} and \eqref{e:wu6}. \qedhere
		
		%
		
	\end{enumerate}

\end{proof}

For any ball $B_U(x,r)$ with $r \ge \delta_U(x)$,  define
\begin{equation} \label{e:defR1}
	\mathfrak{R}_1(B_U(x,r)):= \bigcup_{D \in \mathfrak{R}(B_U(x,r))} \bS(D)= \{ B \in \mathfrak{R}: B \in \bS(D)  \mbox{ and   $D \in \mathfrak{R}(B_U(x,r))$} \},
\end{equation}
where $\bS(D)$ is given in Lemma \ref{l:string}. 
\begin{remark} \label{r:central1}
	We note that the assumption that $B_0= B_U(x_0,r_0)\in \mathfrak{R}(B_U(x,r))$ is central can be replaced with the condition that $r_0 \ge c_0 r$ for some $c_0 >0$. In the case, the constants in the conclusion will also depend on $c_0$ and the proof can be modified by replacing the use of \eqref{e:bndr0} with the estimate $r_0 \ge c_0 r$.
\end{remark}
\subsection{Reflection of Whitney balls \`a la Jones}
Let $U \subsetneq X$ be a non-empty   $A$-uniform domain and let $\mathfrak{R}$ be an $\epsilon$-Whitney cover of $U$, where $\epsilon \in (0,1/14)$. Let $V:=(U^c)^\circ$ denote the interior of $U^c$. We record a simple topological fact. 
\begin{lemma} \label{l:top}
	Let $U$ be a non-empty open set and let $V= (U^c)^\circ$. Then 
	$\partial V \subseteq \partial U$.
\end{lemma}
\begin{proof}
	Note that $\overline{V} \subset U^c$ (since $U^c$ is closed and contains $V$) and $V^c=  \left( (U^c)^\circ \right)^c = \ol{ \left(U^c\right)^c}= \ol{U}$ (since $(A^\circ)^c= \ol{A^c}$). Therefore
	\[
	\partial V= \overline V \cap V^c \subseteq U^c \cap \ol{U}= \partial U. \qedhere
	\]
\end{proof}

If $V \neq \emptyset$, let $\mathfrak{S}$ be a $\epsilon$-Whitney cover of $V$.
Let 
\begin{equation} \label{e:defStil}
	\wt {\mathfrak{S}}= \left\{ B_V(y,s) \in \mathfrak{S}: s < \frac{\epsilon}{6A(1+\epsilon)}\diam(U) \right\}.
\end{equation}
In particular $\wt {\mathfrak{S}}= {\mathfrak{S}}$, if $\diam(U)=\infty$.
Next, we define a `reflection' map $Q:\wt {\mathfrak{S}} \to \mathfrak{R}$ that maps a ball  from  $\wt {\mathfrak{S}}$ to a ball in the Whitney cover $\mathfrak{R}$ that is similar to a construction of Jones \cite{Jon80,Jon81} motivated by quasiconformal reflection.
The following proposition is a modification of \cite[Lemmas 2.4, 2.5, 2.6, 2.7 and 2.8]{Jon81}. 
\begin{proposition} \label{p:reflect} Let $(X,d)$ be a doubling metric space and let $\eps\in (0,1/5)$. Let $U$ be a   $A$-uniform domain with $V= (U^c)^\circ \neq \emptyset$. Let $\mathfrak{R},\mathfrak{S}$ denote the $\epsilon$-Whitney covers of $U,V$ respectively and let $\wt{\mathfrak{S}}$ be as defined in \eqref{e:defStil}.
	Then there exists a map $Q: \wt {\mathfrak{S}} \to \mathfrak{R}$ such that the following properties hold:
	\begin{enumerate}[(a)]
		\item For any $B=B_{V}(y,s) \in  \wt {\mathfrak{S}}$, the ball $Q(B)= B_U(x,r)$ satisfies 
		\begin{equation}\label{e:prefa}
		\frac{1+\eps}{1+4 \eps}s < r <    \frac{1+\eps}{1-2 \eps}s, \quad d(x,y) \le \left(2+ \frac{3A}{2}\right) \frac{1+\eps}{\eps} s.
		\end{equation}
		\item  There exists $K \in \bN$ which depends only on $\eps, A$ and the doubling constant of $(X,d)$ such that the map $Q$ is at most $K$ to $1$; that is 
		\[
		\sup \{ \# Q^{-1}(B): B \in \mathfrak{R} \} \le K.
		\]
		\item If $B_V(y_i,s_i) \in \wt{ \mathfrak{S}}, i=1,2$ satisfy $B_V(y_1,6s_1)  \cap B_V(y_2,6s_2) \neq \emptyset$, then 
		$B_U(x_i,r_i) = Q (B_V(y_i,s_i)), i =1,2$ satisfy
		\begin{equation}\label{e:rcomp}
		\frac{(1-2\eps)(1-5 \eps)}{(1+4 \eps)(1+7 \eps)} r_2 \le	r_1 \le   \frac{(1+4 \eps)(1+7 \eps)}{(1-2\eps)(1-5 \eps)} r_2.
		\end{equation}
		Furthermore, there is a chain of distinct balls $\{B_U(z_i,t_i) \in \mathfrak{R}: 1 \le i \le N \}$ such that $z_1=x_1, t_1= r_1, z_N=x_2, t_N= r_2$, $B_U(z_i, 3t_i) \cap B_U(z_{i+1},3t_{i+1}) \neq \emptyset$  for all $i=1,\ldots,N-1$, and $N$ satisfies the bound
		\[
		N \le  N_0,
		\]
		where $N_0$ depends only on $\eps, A$ and the doubling constant of $(X,d)$.
	\end{enumerate}
	
\end{proposition} 
\begin{proof}
	For any $B_{V}(y,s) \in \wt{\mathfrak{S}}$, choose $z \in \partial V \subset \partial U$ (by Lemma \ref{l:top}) such that $\frac{1+\eps}{\eps}s=\delta_{V}(y)= d(y,z)$. Choose  points $z_1,z_2 \in U$ such that $$d(z,z_2) \ge \frac{5A s}{2} \frac{1+\eps}{\eps} \ge \frac{5 }{12}\diam(U), \mbox{ and }  d(z,z_1) \le \frac{(1+\eps)s}{2 \eps}.$$ By considering a   $A$-uniform curve $\gamma$ in $U$ from $z_1$ to $z_2$, we pick $z_3 \in \gamma$ to be the first point along the curve from $z_1$ to $z_3$  such that $\delta_U(z_3) = \frac{1+\eps}{\eps} s$. We claim  that such a point $z_3$ exists and satisfies
	\begin{equation}\label{e:db}
		d(z_1,z_3) \le \frac{A(1+\eps)}{\eps} s. 
	\end{equation}  To see this, note that if $z_4$ is the first point along $\gamma$ so that $d(z_1,z_4) = \frac{A(1+\eps)}{\eps} s$, then $$\delta_U(z_4) \ge  A^{-1}\min(d(z_1,z_4),d(z_2,z_1)-d(z_1,z_4)) = \frac{1+\epsilon}{\epsilon}s.$$ Therefore $d(z_1,z_3) \le d(z_1,z_4) \le \frac{A(1+\eps)}{\eps} s$.

	By Definition \ref{d:whitney}(iii), there exists $B_U(x,r) \in \mathfrak{R}$ such that $z_3 \in B(x, 3r)$. For each $B_V(y,s) \in \wt{\mathfrak{S}}$, we set $B_U(x,r) \in \mathfrak{R}$  as $Q(B_V(y,s))$. This defines a map $Q: \wt {\mathfrak{S}} \to \mathfrak{R}$. We will now verify that it satisfies the desired properties (a)-(c).
	\begin{enumerate}[(a)]
		\item 
		By Proposition \ref{p:whitney}(b), we have 
		\begin{equation}\label{e:rsbnd}
		\frac{1+\eps}{1+4 \eps}s =\frac{\eps}{1+4 \eps} \delta_U(z_3)< r < \frac{\eps}{1-2 \eps} \delta_U(z_3)=    \frac{1+\eps}{1-2 \eps}s.
		\end{equation}
		Furthermore, by the choice of $z,z_1,z_3$ and \eqref{e:db}, we have
		\begin{align} \label{e:dxybnd}
			d(x,y) &< d(y,z)+ d(z,z_1)+d(z_1,z_3) + d(z_3,x) \nonumber \\
			&\le \frac{1+\eps}{\eps} s + \frac{A(1+\eps)}{2 \eps} s+  \frac{A(1+\eps)}{ \eps} s+ 3r   \nonumber\\
			& \le \left(2+ \frac{3A}{2}\right) \frac{1+\eps}{\eps} s \q \mbox{(by \eqref{e:rsbnd} and $\eps<1/5$).}
		\end{align}
		The property (a) follows form \eqref{e:rsbnd} and \eqref{e:dxybnd}.
		\item  Let $B_V(y_1,s_1), B_V(y_2,s_2)$ be two distinct balls in $\wt{\mathfrak{S}}$ such that $Q(B_V(y_1,s_1))=Q(B_V(y_2,s_2))= B_U(x,r) \in \mathfrak{R}$. Then
		\begin{equation}\label{e:kto1-1}
		d(y_i,x) \stackrel{\eqref{e:dxybnd}}{\le} \left(2+ \frac{3A}{2}\right) \frac{1+\eps}{\eps} s_i \stackrel{\eqref{e:rsbnd}}{\le} \left(2+ \frac{3A}{2}\right) \frac{1+4\eps}{\eps} r, \q \mbox{for $i=1,2$.}
		\end{equation}
		Since $B_V(y_1,s_1) \cap B_V(y_2,s_2) = \emptyset$, we have
		\begin{equation}\label{e:kto1-2}
		d(y_1,y_2) \ge s_1 \stackrel{\eqref{e:rsbnd}}{\ge} \frac{1-2\eps}{1+\eps}r.
		\end{equation}
		By \eqref{e:kto1-1}, \eqref{e:kto1-2} and   \cite[Exercise 10.17]{Hei}, there exists $K \in \bN$ that depends only on $\eps, A, K$ such that $Q$ is at most $K$ to $1$.
		\item The estimate \eqref{e:rcomp} is an immediate consequence of \eqref{e:rsbnd} and Proposition \ref{p:whitney}(c).
		By triangle inequality $d(x_1,x_2) \le d(x_1,y_1) + d(y_1,y_2) + d(x_2,y_2)$, $d(y_1,y_2) \le 6(s_1+s_2)$ (since $B_V(y_1,6s_1)  \cap B_V(y_2,6s_2) \neq \emptyset$),  \eqref{e:dxybnd}, and Proposition \ref{p:whitney}(c), we obtain
		\begin{equation}\label{e:ch1}
		d(x_1,x_2) <  \left[ 6+ \left(2+ \frac{3A}{2}\right) \frac{1+\eps}{\eps} \right](s_1+s_2) \le \left[6+ \left(2+ \frac{3A}{2}\right) \frac{1+\eps}{\eps} \right] \frac{2(1+\eps)}{1-5\eps}s_1.
		\end{equation}
		If $Q (B_V(y_1,s_1))= Q (B_V(y_2,s_2))$, we choose the obvious chain with $N=1$ ball. 
		Otherwise, we connect $x_1$ and $x_2$ with a   $A$-uniform curve $\wt \gamma$, whose diameter $\diam (\wt \gamma)$ satisfies the bound
		\begin{equation}\nonumber
		r_1 \le d(x_1,x_2) \le \diam(\wt \gamma) \le A d(x_1,x_2) \stackrel{\eqref{e:ch1}}{\le} \left[ 6+ \left(2+ \frac{3A}{2}\right) \frac{1+\eps}{\eps} \right] \frac{2A(1+\eps)}{1-5\eps}s_1.
		\end{equation}
		This along with \eqref{e:rsbnd} yields
		\begin{equation}\label{e:ch2}
		\frac{1+\eps}{1+4\eps} s_1 \le \diam (\wt \gamma)\le \left[ 6+ \left(2+ \frac{3A}{2}\right) \frac{1+\eps}{\eps} \right] \frac{2A(1+\eps)}{1-5\eps}s_1.
		\end{equation}
		For any $z \in \wt \gamma$, we claim that 
		\begin{equation}\label{e:ch3}
		\delta_U(z) \ge    \frac{1}{A} (r_1 \wedge r_2).
		\end{equation}
		The proof of \eqref{e:ch3} is divided into two cases depending on whether or not $z \in B_U(x_1,r_1) \cup B_U(x_2,r_2)$. If $z \in B_U(x_i, r_i)$ for some $i=1,2$, then 
		\[
		\delta_U(z) > \delta_U(x_i)-r_i= \frac{1+\eps}{\eps} r_i - r_i = \eps^{-1} r_i \ge  (r_1 \wedge r_2).
		\]
		On the other hand, if  $z \in B_U(x_1,r_1)^c \cap  B_U(x_1,r_1)^c \cap \wt \gamma$, then by the weak $A$-uniformity of $\wt \gamma$, we have
		\[
		\delta_U(z) \ge \frac{1}{A} \min(d(x_1,z), d(x_2,z)) \ge \frac{1}{A}  (r_1 \wedge r_2)
		\]
		which completes the proof of \eqref{e:ch3}. By \eqref{e:ch1}, \eqref{e:rcomp}, and \eqref{e:rsbnd}, there exists $C_1$ which depends only on $\eps,A$ such that 
		\begin{equation} \label{e:ch4}
		\delta_U(z) \le C_1 (r_1 \wedge r_2).
		\end{equation}

		By the same argument as in Lemma \ref{l:string}(a), we construct    a chain of distinct balls $\{B_U(z_i,t_i) \in \mathfrak{R}: 1 \le i \le N \}$ such that \[  z_1=x_1, t_1= r_1, z_N=x_2, t_N= r_2, \quad  B_U(z_i, 3t_i) \cap B_U(z_{i+1},3t_{i+1}) \neq \emptyset  \]  for all $i=1,\ldots,N-1$ and $\wt \gamma \cap  B_U(z_i, 3t_i) \neq \emptyset$ for all $i=1,\ldots,N$. 
		It remains to obtain an  upper bound on $N$ that   depends only on $\eps, A$ and the doubling constant of $(X,d)$.
		By Proposition \ref{p:whitney}(b), \eqref{e:ch3}, and \eqref{e:ch4}, we obtain
		\begin{equation}\label{e:ch5}
		\frac{C_1(1+\eps)}{1-2\eps} (r_1 \wedge r_2) \stackrel{\eqref{e:whitb}, \eqref{e:ch4}}{>}  \frac{1+\eps}{\eps}t_i = \delta_U(z_i) \stackrel{\eqref{e:whitb}, \eqref{e:ch3}}{>} \frac{1+\eps}{A(1+4 \eps)} (r_1 \wedge r_2).
		\end{equation}
		By \eqref{e:ch2}, \eqref{e:rsbnd}, \eqref{e:rcomp},  \eqref{e:ch4}, and \eqref{e:ch5},  there exists $C_2$ that depends only on $\eps, A$ such that
		\begin{equation} \label{e:ch6}
		d(x_1,z_i) \le \diam(\wt \gamma) + 3 \max_{1\le  i \le N} t_i \le C_2(r_1 \wedge r_2), \q \mbox{for all $i=1,\ldots,N$.}
		\end{equation}
		Since the balls $B_U(z_i,t_i) \in \mathfrak{R}, 1 \le i \le N$ are pairwise disjoint, by \eqref{e:ch5} the points $\{z_i: 1 \le i \le N\}$ have mutual distance of at least $$\frac{\eps}{A(1+4\eps)} (r_1 \wedge r_2).$$ Combining this with \eqref{e:ch6} and   \cite[Exercise 10.17]{Hei}, we obtain the desired bound on $N$. \qedhere
	\end{enumerate}
\end{proof}

We record a few more geometric properties of the Whitney covers $\mathfrak{R}, \mathfrak{S}$  that will be used.
We define graphs $G_{\mathfrak{R}}, G_{\mathfrak{S}}$ with vertices  $\mathfrak{R}, \mathfrak{S}$ respectively below.
\begin{definition} \label{d:graph}
	Let $G_{\mathfrak{R}}$ denote the graph (undirected) whose vertex set is $\mathfrak{R}$ and such that $B_U(x_1,r_1), B_U(x_2,r_2)\in \mathfrak{R}$ are connected by an edge if $B_U(x_1,3r_1) \cap B_U(x_2, 3r_2) \neq \emptyset$ and $B_U(x_1,r_1) \neq B_U(x_2,r_2)$. In this case, we denote it by $B_U(x_1,r_1) \stackrel{\mathfrak{R}}{\sim}B_U(x_2,r_2)$.
	
	Similarly, we 	define $G_{\mathfrak{S}}$ to be the graph whose vertex set is $\mathfrak{S}$ and such that $B_V(y_1,s_1), B_V(y_2,s_2)\in \mathfrak{S}$ are connected by an edge if $B_V(y_1,6s_1) \cap B_V(y_2, 6r_2) \neq \emptyset$ and $B_V(y_1,s_1) \neq B_V(y_2,s_2)$. In this case, we denote it by $B_V(y_1,s_1) \stackrel{\mathfrak{S}}{\sim}B_V(y_2,s_2)$.
	
	Let $D_{\mathfrak{R}}, D_{\mathfrak{S}}$ denote the (combinatorial) graph distance defined on $\mathfrak{R}, \mathfrak{S}$  induced by the graphs  $G_{\mathfrak{R}}, G_{\mathfrak{S}}$ respectively.
\end{definition}
For $B \subset X$, we define 
\begin{equation} \label{e:defSB}
	\mathfrak{S}(B)= \{B_V(y,s) \in \mathfrak{S} : B_V(y,6s) \cap B \neq \emptyset \}, \quad \wt{\mathfrak{S}}(B)= \mathfrak{S}(B) \cap \wt{\mathfrak{S}}. 
\end{equation}
This definition is a slight variation of  \eqref{e:defRB} as the constant $3$ is replaced by $6$ in the current definition.
\begin{lemma} \label{l:reflect}
	Let $(X,d)$, $U,V,\mathfrak{R},\mathfrak{S}, Q: \wt{\mathfrak{S}} \to \mathfrak{R}$, $\eps \in (0,1/5)$ be as given in Proposition \ref{p:reflect}.  Let $G_{\mathfrak{R}},G_{\mathfrak{S}}$ denote the graphs as given in Definition \ref{d:graph}. 
	\begin{enumerate}[(a)]
		\item For any $\xi \in \partial U, r >0$ and $B_V(y, s) \in \mathfrak{S}(B(\xi,r))$, we have
		\begin{equation} \label{e:srbnd}
			s < \frac{\epsilon}{1-5 \epsilon} r.
		\end{equation}
		In particular, if $r \le \frac{1-5\eps}{6A(1+\eps)} \diam(U)$ then $B_V(y, s) \in \wt{\mathfrak{S}}$.
		
		Similarly, if  $\xi \in \partial U, r >0$ and $B_U(x_0, r_0) \in \mathfrak{R}(B(\xi,r))$, we have
		\begin{equation} \label{e:srbnd1}
			r_0 < \frac{\epsilon}{1-2 \epsilon} r.
		\end{equation}
		\item $G_{\mathfrak{R}}, G_{\mathfrak{S}}$ are bounded degree graphs.
		\item The map $Q: \wt{\mathfrak{S}} \to \mathfrak{R}$ is Lipschitz with respect to the distances $D_{\mathfrak{R}}, D_{\mathfrak{S}}$ on $\mathfrak{R}, \wt{\mathfrak{S}}$ respectively.
		\item There exists $K_0 \in (0,\infty)$ depending only on $A, \epsilon$ such that for any $\xi \in \partial U, r>0$ and for any $B \in \wt{\mathfrak{S}}(B_V(\xi,r))$, we have $Q(B) \in \mathfrak{R}(B_U(\xi,K_0 r))$. 
		\item For any $L$, there exists $K$ such that for any 
		$B_1 \in \mathfrak{R}(B_U(\xi,r)), B_2 \in \mathfrak{R}$ such that $D_{\mathfrak{R}}(B_1,B_2) \le L$, we have $B_2 \in \mathfrak{R}(B_U(\xi, K r))$. 
	\end{enumerate}	
\end{lemma}
\begin{proof}
	\begin{enumerate}[(a)]
		\item Since $$B_V(y,6s) \cap B_V(\xi,r) \neq \emptyset, \quad 
		\frac{1+\eps}{\eps}s=\delta_V(y) \le d(y,\xi)<6s +  r,$$  we obtain \eqref{e:srbnd}. The second conclusion follows from \eqref{e:defStil} and  \eqref{e:srbnd}. The estimate \eqref{e:srbnd1} follows from the same argument as the proof of \eqref{e:srbnd}.
		\item Proposition \ref{p:whitney}(e) provides an uniform upper bound on the degree of the graphs.
		\item This is just a restatement of Proposition \ref{p:reflect}(c).
		\item  Let $B_V(y,s) \in \wt{\mathfrak{S}}(B(\xi,r))$, where $\xi \in \partial U, r>0$. Let $Q(B_V(y,s))= B_U(x_i,r_i) \in \mathfrak{R}$. 
		Hence, we obtain
		$$d(x_i,y) \stackrel{\eqref{e:prefa}}{ \le} \left(2+ \frac{3A}{2}\right) \frac{1+\eps}{\eps} s \stackrel{ \eqref{e:srbnd}}{<} \left(2+ \frac{3A}{2}\right) \frac{1+\eps}{1-5\eps} r. $$
		This along with $d(x_i,\xi) < d(x_i,y)+d(y,\xi) < d(x_i,y) + 6s+r  < K_0 r$ where 
		$$ K_0 = \left(2+ \frac{3A}{2}\right) \frac{1+\eps}{1-5\eps}+  \frac{6 \eps}{1-5 \eps}+ 1.$$
		In particular, $Q(B) \in \mathfrak{R}(B_U(\xi,K_0 r))$.
		\item We note that for any $B_U(x_i,r_i) \in \mathfrak{R}(B_U(\xi,r))$ with $\xi \in \partial U , r>0$, we have 
		$$r_i < \frac{\epsilon}{1-2\epsilon}r, \quad 
		d(\xi,x_i)< 3r_i+r  < \frac{1+\eps}{1-2\eps} r.$$
		This follows from the same argument as the proof of \eqref{e:srbnd}. Combining the above estimate with Proposition \ref{p:whitney}(c) and triangle inequality, we obtain the desired result. \qedhere
	\end{enumerate}
\end{proof}

\section{Sub-Gaussian heat kernel estimates} \label{s:sghke}
In this section, we recall some background material on Dirichlet forms, extension domains, and sub-Gaussian heat kernel estimates. We show that in a  extension domain the reflected Dirichlet form is regular.

\subsection{Extension domain}

Note that we always have the inclusion
\begin{equation}\label{e:trivialinc}
\sF(U) \subseteq \{f \in L^2(U,m): \mbox{there exists $\wt f \in \sF$ such that $f=\wt f$, $m$-a.e. on $U$}\}.
\end{equation}
A natural question is if the above inclusion is an equality.
The following notion plays a central role in this work.
\begin{definition} [Extension domain] \label{d:extend}
	Let $(X,d,m,\sE,\sF)$ be an MMD space and let $U \subset X$ be open. We say that $U$ is an \emph{extension domain} for  $(X,d,m,\sE,\sF)$,  if there is a bounded linear map $E: \sF(U) \to \sF$ such that $E$ is an extension map; that is, $E(f) =f$ $m$-a.e. on $U$. Here the boundedness of $E$ is with respect to the inner products $	\sE_U(\cdot,\cdot) + \langle \cdot, \cdot \rangle_{L^2(U,m)}, \sE(\cdot,\cdot) + \langle \cdot, \cdot \rangle_{L^2(X,m)}$ on $\sF(U)$ and $\sF$ respectively.
\end{definition}
Note that the inclusion \eqref{e:trivialinc} is an equality for any extension domain $U$.
An immediate consequence of the extension property is that the regular Dirichlet form on $L^2(X,m)$ induces a regular Dirichlet form on $L^2(\overline{U},m)$. 
By the correspondence between regular Dirichlet forms and symmetric Markov processes \cite[Theorem 7.2.1 and 7.2.2]{FOT}, this corresponds to a $m$-symmetric diffusion process on $(\overline{U},m)$ which is called the \emph{reflected diffusion} on $\overline{U}$.
\begin{lemma} \label{l:regular}
	Let $(X,d)$ be a complete metric space such that all bounded sets are precompact. 
	Let $(\sE,\sF)$ be strongly local regular Dirichlet form. If $U$ is an extension domain such that $m(\partial U)=0$, then $(\sE_U,\sF(U))$ defines a strongly local regular Dirichlet form on $\overline{U}$, where $(\sE_U,\sF(U))$  is the local Dirichlet space of Definition \ref{d:local}.
\end{lemma}
\begin{proof}
	We equip $\sF(U), \sF$ with the corresponding inner products $\sE_U(\cdot,\cdot) + \langle \cdot, \cdot \rangle_{L^2(U,m)}$, and  $\sE(\cdot,\cdot) + \langle \cdot, \cdot \rangle_{L^2(X,m)}$  respectively.
	Let $E: \sF(U) \to \sF$ be an extension operator as given in Definition \ref{d:extend}. If $f_n$ is a Cauchy sequence in $\sF(U)$, then $E(f_n)$ is a Cauchy sequence in $\sF$ (since $E$ is a bounded operator). Since $\sF$ is complete, $E(f_n)$ converges to a limit, say $g \in \sF$. The restriction of $g$ to $U$ yields the desired limit of $f_n$ in $\sF(U)$. Hence $(\sE_U,\sF(U))$ is a Dirichlet form on $L^2(\overline{U},m)$ (note that since $m(\partial U)=0$, $L^2(U,m)$ can be identified with $L^2(\overline{U},m)$).
	
	The regularity of $(\sE_U,\sF(U))$ is an easy consequence of the regularity of $(\sE,\sF)$ and the extension property. To this end, note that function $f \in \sF(U)$ has an extension $E(f) \in \sF$ which is a limit of functions in $C_c(X) \cap \sF$ and hence by restricting this sequence of functions to $\overline{U}$, we obtain that $C_c(\overline{U}) \cap \sF(U)$ is dense in $\sF(U)$.  Since any function in $C_c(\overline{U})$ can be extended to a function on $C_c(X)$, by the same argument as above, we obtain that any function on $C_c(\overline{U})$ is a limit (with respect to the uniform norm) of functions in  $C_c(\overline{U}) \cap \sF(U)$.
\end{proof}

In general, we note that the Dirichlet form $(\sE_U,\sF(U))$ need not be regular on $\overline{U}$. In particular, $C_c(\overline{U}) \cap \sF(U)$ need not be dense in $\sF(U)$. This can be seen by considering the slit domain $U=\mathbb{R}^2 \setminus \{(t,0): t \in (-\infty,0]\}$ for the Brownian motion on $\mathbb{R}^2$.
\subsection{Sub-Gaussian heat kernel estimates and its consequences}
In this subsection, we recall some previous results concerning the heat kernel and its sub-Gaussian estimates.   We start with recall the definition of  capacity. 
For disjoint Borel sets $B_1,B_2$ such that $B_2$ is closed and $\overline{B_1} \Subset B_2^c$ (by $B_1 \Subset B_2^c$, we mean that $\ol B_1$ is compact and $\ol B_1 \subset B_2^c$), we
define $\sF(B_1,B_2)$ as the set of function $\phi \in \sF$ such that $\phi \equiv 1$ in an open neighborhood of $B_1$, and $\supp_m(\phi) \subset B_2^c$.
For such sets $B_1$ and $B_2$, we define the capacity between them as
\[
\Cap(B_1,B_2)= \inf \set{\sE(f,f) \mid f \in \sF(B_1,B_2)}.
\]

The sub-Gaussian heat kernel estimates are known to be equivalent to properties that are known to be stable under perturbations. To recall this characterization, we recall the relevant properties. 	

\begin{definition}[\hypertarget{pi}{$\operatorname{PI}(\Psi)$}, \hypertarget{cs}{$\operatorname{CS}(\Psi)$}, and \hypertarget{cap}{$\operatorname{cap}(\Psi)$}]\label{d:PI-CS}
	Let $(X,d,m,\mathcal{E},\mathcal{F})$ be an MMD space, $\Psi:[0,\infty) \to [0,\infty)$ be a scale function, and let $\Gamma(\cdot,\cdot)$ denote the corresponding energy measure.
	\begin{enumerate}[(a)]
		\item 
		We say that $(X,d,m,\sE,\sF)$ satisfies the \textbf{Poincar\'e inequality} \hyperlink{pi}{$\operatorname{PI}(\Psi)$},
		if there exist constants $C_{P},A_P\ge 1$ such that 
		for all $(x,r)\in X\times(0,\infty)$ and all $f \in \mathcal{F}_{\loc}(B(x,A_Pr))$,
		\begin{equation} \tag*{$\operatorname{PI}(\Psi)$}
			\int_{B(x,r)} (f -   f_{B(x,r)})^2 \,dm  \le C_{P} \Psi(r) \int_{B(x,A_P r)}d\Gamma(f,f),
		\end{equation}
		where $f_{B(x,r)}:= m(B(x,r))^{-1} \int_{B(x,r)} f\, dm$.
		
		\item 
		For open subsets $U,V$ of $X$ with $\overline{U} \subset V$, we say that
		a function $\phi \in \sF$ is a \emph{cutoff function} for $U \subset V$
		if $0 \le \phi \le 1$, $\phi=1$ on a neighbourhood of $\overline{U}$ and $\supp_m[\phi] \subset V$.
		Then we say that $(X,d,m,\sE,\sF)$ satisfies the \textbf{cutoff Sobolev inequality} \hyperlink{cs}{$\operatorname{CS}(\Psi)$},
		if there exists $C_S>0$ such that the following holds: for all $x \in X$ and $R,r>0$,
		there exists a cutoff function $\phi \in \sF$ for $B(x,R) \subset B(x,R+r)$ such that for all $f \in \sF$,
		\begin{equation}\tag*{$\operatorname{CS}(\Psi)$}
			\begin{split}
				&\int_{B(x,R+r) \setminus B(x,R)} \wt{f}^2\, d\Gamma(\phi,\phi)\\
				&\mspace{40mu}\le \frac{1}{8} \int_{B(x,R+r) \setminus B(x,R)} \wt{\phi}^2 \, d\Gamma(f,f)
				+ \frac{C_S}{\Psi(r)} \int_{B(x,R+r) \setminus B(x,R)} f^2\,dm;
			\end{split}
		\end{equation}
		where $\wt{f},\wt{\phi}$ are the quasi-continuous versions\footnote{We refer the reader \textsection \ref{ss:compare} for the definition of quasi-continuous functions with respect to a
			regular symmetric Dirichlet form.} of $f,\phi \in \sF$ so that $\wt{\phi}$ is uniquely determined $\Gamma(f,f)$-a.e.\ for
		any $f\in\mathcal{F}$; see \cite[Theorem 2.1.3, Lemmas 2.1.4 and 3.2.4]{FOT}.
		
		\item We say that an MMD space $(X,d,m,\sE,\sF)$ satisfies the \emph{capacity upper bound} \hypertarget{cap}{$\operatorname{cap}(\Psi)_\le$} 	
		if there exist $C_1,A_1,A_2>1$ such that for all $R\in (0,\diam(X,d)/A_2)$, $x \in X$, we have
		\begin{equation}\tag*{$\operatorname{cap}(\Psi)_\le$}
		\Cap(B(x,R),B(x,A_1R)^c) \le C_1 \frac{m(B(x,R))}{\Psi(R)}.
		\end{equation}
	\end{enumerate}
\end{definition}

The following theorem was first proved in the context of random walks on graphs by Barlow and Bass \cite{BB04}. It was later extended to MMD spaces by Barlow, Bass and Kumagai \cite{BBK}. Following a simplification of cutoff Sobolev inequality by Andres and Barlow \cite{AB}, the following characterization was proved by Grigor'yan, Hu, and Lau \cite{GHL15} for unbounded MMD spaces. The same arguments also apply with minor changes to the bounded setting as pointed out in \cite[Remark 2.9]{KM20} and \cite[Theorem 4.9]{KM23}.

\begin{theorem} (\cite{BB04,BBK,GHL15}) \label{t:hke}
	Let $(X,d,m,\sE,\sF)$ be an MMD space
	such that $m$ is a doubling measure on $(X,d)$ and let $\Psi$ be a scale function. Then the following are equivalent.
	\begin{enumerate}[(a)]
		\item $(X,d,m,\sE,\sF)$  satisfies   \hyperlink{hke}{$\on{HKE(\Psi)}$}. 
		\item  $(X,d,m,\sE,\sF)$  satisfies   \hyperlink{pi}{$\on{PI(\Psi)}$} and \hyperlink{cs}{$\on{CS(\Psi)}$}.
	\end{enumerate}
\end{theorem}

The capacity upper bound \hyperlink{cap}{$\on{cap}(\Psi)_\le$} and the regularity of the Dirichlet form $(\sE,\sF)$
implies the existence of a partition of unity on $V$ with controlled energy by a standard argument.
\begin{lemma} \label{l:partition}
	Let $(X,d,m,\sE,\sF)$ be an MMD space that satisfies \hyperlink{cap}{$\on{cap}(\Psi)_\le$} and let $m$ be a doubling measure. There exists $\epsilon_0 \in (0,1/6),C_1,c_1 > 0$ such that 
	for any open set $V \subsetneq X$ be open and for any $\epsilon$-Whitney cover  $\mathfrak{S}$   of $V$ where $0<\epsilon < \epsilon_0$, there exists a partition on unity $\{\psi_B: B \in \mathfrak{S}\}$ of $V$ such that the following properties hold:
	\begin{enumerate}[(a)]
		\item (partition of unity) $\sum_{B \in \mathfrak{S}} \psi_B \equiv \one_V$ and $0 \le \psi_B \le 1$ for all $B \in \mathfrak{S}$. 
		\item (controlled energy) For each $B=B_V(y,s) \in \mathfrak{S}$, then $\psi_B \in \contfunc_c(X) \cap \sF$ such that $\psi_B \ge c_1$ on $B_V(y,3s)$, $\psi_B \equiv 0$ on $B_V(y,6s)^c$ and 
		\begin{equation}\label{e:capbnd}
		\sE(\psi_B,\psi_B) \le C_1 \frac{m(B(y,s))}{\Psi(s)}.
		\end{equation}
		
	\end{enumerate}
\end{lemma}
\noindent {\em Proof sketch.} The proof follows from the argument in \cite[p. 504]{BBK} or \cite[Proof of Lemma 2.5]{Mur20} where the  required bounded overlap property for that argument follows from Proposition \ref{p:whitney}(d).

In Theorem \ref{t:besov}, we estimate the Dirichlet energy and energy measure of a function via a Besov energy type expression. The statement  is inspired by a similar result due to Grigor'yan, Hu and Lau \cite[Theorem 4.2]{GHL03} and its extension due to Kumagai and Sturm \cite[Theorem 4.1]{KuSt} (see also \cite[Theorem 1]{Jons} for this result on the Sierpi\'nski gasket and \cite[Theorem 1.6.2]{KS} on the Euclidean space).  One difference from these works is that the earlier results contain some technical conditions \cite[(4.2)]{KuSt} and \cite[(4.9)]{GHL03} that turn out to be unnecessary. 
Furthermore, another improvement is that we obtain estimates on the \emph{energy measure} along with  estimates on energy.
We use the following notation in Theorem \ref{t:besov}.  Let $\delta>0, K \subset X$ and denote by $K_\delta:= \{y \in X : \mbox{there exists $z \in K$ such that $d(y,z)<\delta$}\}$ the $\delta$-neighborhood of $K$. Part (b) of the lemma will provide a sufficient condition to verify if a given function $f \in L^2(X,m)$ belongs to $\sF$. If the assumption of both part (a) and (b) hold then this condition given by \eqref{e:fcrit} is both necessary and sufficient.
\begin{theorem} \label{t:besov}
	Let $(X,d,m,\sE,\sF)$ be an MMD space and let $\Psi$ is a scale function. 
	\begin{enumerate}[(a)]
		\item Let $m$ satisfy the volume doubling property and $(X,d,m,\sE,\sF)$ satisfy the Poincar\'e inequality \hyperlink{pi}{$\operatorname{PI}(\Psi)$} then there exists $C>0$ such that 
		\[
		\sup_{r \in (0,\infty)}  \frac{1}{\Psi(r)}\int_X  \fint_{B(x,r)} (f(x)-f(y))^2 \, m(dy) \, m(dx) \le C \sE(f,f),    
		\]
		for all $f \in \sF$.
		\item  Let $m$ satisfy the volume doubling property and  $(\sE,\sF)$ be a conservative Dirichlet form satisfy the heat kernel upper bound \eqref{e:uhke}. If $f \in L^2(X,m)$ satisfies
		\begin{equation} \label{e:fcrit}
			\limsup_{r \to 0}  \frac{1}{\Psi(r)}\int_X  \fint_{B(x,r)} (f(x)-f(y))^2 \, m(dy) \, m(dx)<\infty,
		\end{equation}
		then $f \in \sF$.
		There exists $C>0$  depending only the constants in the volume doubling property, \eqref{e:reg} and \eqref{e:uhke}  such that 
		\begin{equation}  \label{e:uener}
			\sE(f,f) \le C  \limsup_{r \to 0}  \frac{1}{\Psi(r)}\int_X   \fint_{B(x,r)} (f(x)-f(y))^2 \, m(dy) \, m(dx),
		\end{equation}
		for all $f \in \sF$. 
		\item Let     $(X,d,m,\sE,\sF)$ satisfy the assumptions in (b) above. Then there exists $C>0$ such that
		\begin{equation} \label{e:umeas}
			\Gamma(f,f)(K) \le C \lim_{\delta \downarrow 0}\limsup_{r \to 0}  \frac{1}{\Psi(r)}\int_{K_\delta}  \fint_{B(x,r)} (f(x)-f(y))^2 \, m(dy) \, m(dx),
		\end{equation}
		for all $f \in \sF$ and for any compact set $K \subset X$.
	\end{enumerate}	
\end{theorem}
\begin{proof} 
	We record some useful estimates concerning the function $\Phi$ defined in \eqref{e:defPhi}. By \cite[Lemma 2.10]{Mur20}, there exists $C_1>0$ such that 
	\begin{equation} \label{e:Phiscale}
		C_1^{-1} \left( \frac S s \right)^{\beta_2/(\beta_2-1)} \le \frac{\Phi(S)}{\Phi(s)} \le C_1 \left( \frac S s \right)^{\beta_1/(\beta_1-1)}, \q \mbox{for all $0 < s \le S$.}
	\end{equation}
	By \cite[(6.14)]{GT12}, there exists $c  \in (0,1)$ such that 
	\begin{equation} \label{e:GT12}
		\Phi \left( c_1 \frac{r}{\Psi(r)}\right) \le \frac{1}{\Psi(r)} \le \Phi\left( \frac{2 r }{\Psi(r)}\right), \quad \mbox{for all $r>0$.}
	\end{equation}
	Combining \eqref{e:Phiscale} and \eqref{e:GT12}, there exists $C_2 \in (1,\infty)$ such that 
	\begin{equation} \label{e:Phibnd}
		C_2^{-1} \le	 t \Phi\left( \frac{\Psi^{-1}(t)}{t}\right) \le C_2, \quad \mbox{for all $t>0$}.
	\end{equation}
	\begin{enumerate}[(a)]
		\item 
		Let $N \subset X$ be a $r$-net (that is, maximally $r$-separated subset). By the volume doubling property, there exists $C_1>0$ such that
		\begin{equation} \label{e:net}
			\frac{1}{m(B(x,r))} \one_{ \{d(x,y) <r\}} \le C_1 \sum_{n \in N} \frac{1}{m(B(n,2r))} \one_{B(n,2r)}(x) \one_{B(n,2r)}(y)   
		\end{equation}
		for all $x,y \in X$. Therefore for any $r>0$, we have 
		\begin{align*}
			\MoveEqLeft{\frac{1}{\Psi(r)}\int_X \frac{1}{m(B(x,r))}  \int_{B(x,r)} (f(x)-f(y))^2 \, m(dy) \, m(dx)} & \\
			&\le  \frac{C_1}{2\Psi(r)} \sum_{n \in N} \int_{B(n,2r)} \abs{f(x)-f_{B(n,2r)}}^2\, m(dx) \quad \mbox{(by \eqref{e:net})} \\
			& \le   \frac{C_1C_P \Psi(2r)}{2\Psi(r)}  \sup_{x \in X} \sum_{n \in N} \one_{B(n,2 A_P r)}(x) \sE(f,f) \quad \mbox{(by \hyperlink{pi}{$\operatorname{PI}(\Psi)$})}.
		\end{align*}
		The desired conclusion follows from \eqref{e:reg} and the bounded overlap of the family of balls $\{B(n,2 A_Pr): n \in N\}$.
		\item Let $f \in L^2(X,m)$. Since $(\sE,\sF)$ is conservative, we have by \eqref{e:semigroup}
		\begin{equation} \label{e:bl0}
			\sE(f,f)= \lim_{t \to 0} \frac{1}{2t}\int_X \int_X (f(x)-f(y))^2 p_t(x,y) \,m(dy)\,m(dx),
		\end{equation}
		where $p_t$ denotes the heat kernel.  Following \cite[Proof of Theorem 4.2]{GHL03} , we set
		\begin{align*}
			A(f,t,r) &:= \frac{1}{2t}\int_X \int_{X\setminus B(x,r)} (f(x)-f(y))^2 p_t(x,y) \,m(dy)\,m(dx), \\
			B(f,t,r) &:= \frac{1}{2t}\int_X \int_{ B(x,r)} (f(x)-f(y))^2 p_t(x,y) \,m(dy)\,m(dx),
		\end{align*}
		so that \begin{equation} \label{e:blh}
			\sE(f,f)= \lim_{t \downarrow 0} \left( A(f,t,r)+B(f,t,r) \right) \quad \mbox{for all $f \in L^2, r>0$}.
		\end{equation} 
		Define 
		\[
		W(f,r):= \frac{1}{\Psi(r)}\int_X  \frac{1}{m(B(x,r))}\int_{B(x,r)} (f(x)-f(y))^2 \,m(dy) \,m(dx),  
		\]
		and $W(f):= \limsup_{r \downarrow 0} W(f,r)$.
		First we claim that for any $f \in L^2(X,m), r >0$, we have 
		\begin{equation} \label{e:bl1}
			\limsup_{t \downarrow 0} A(f,t,r)=0.
		\end{equation}

		To this end, we estimate $A(f,t,r)$ as  
		\begin{align} \label{e:bl2}
			A(f,t,r) &= \frac{1}{2t} \int_X \int_{B(x,r)^c} (f(x)-f(y))^2  p_t(x,y) \, m(dy) \,m(dx) \nonumber \\
			& \le  \frac{1}{2t}\int_X \int_{B(x,r)^c} 2(f(x)^2+f(y)^2) p_t(x,y) \,m(dy) \,m(dx) \nonumber \\
			& \le \frac{2}{t}\int_X f(x)^2 \int_{B(x,r)^c} p_t(x,y) \,m(dy) \,m(dx) \q \mbox{(by symmetry of $p_t$).}
		\end{align}
		By \eqref{e:uhke}, for $m$-a.e. $x \in X$  and for all $0<t <1 \wedge \Psi(r) \wedge r$,  we have
		\begin{align} \label{e:bl3}
			\lefteqn{\frac{1}{t} \int_{B(x,r)^c} p_t(x,y) \,m(dy) } \nonumber \\&= \sum_{k=0}^\infty \frac{1}{t} \int_{B(x,2^{k+1}r) \setminus B(x,2^k r)} p_t(x,y) \,m(dy) \nonumber \\
			& \le \frac{C_1}{t} \sum_{k=0}^\infty \frac{m(B(x,2^{k+1}r))}{ m(B(x,\Psi^{-1}(t)))} \exp \left( -c_1 t \Phi\left( \frac{2^k r}{t}\right)\right)  \mbox{ (by \eqref{e:uhke},\eqref{e:Phiscale})} \nonumber \\
			& \le C_2 \sum_{k=0}^\infty \frac{2^{k\alpha} r^\alpha}{t^\beta} \exp \left( - c_2 \frac{2^{k (\gamma+1)}r^{\gamma+1}}{t^\gamma}\right) \q \mbox{(by \eqref{e:Phiscale}, \eqref{e:vd}, \eqref{e:reg}),}
		\end{align}
		where $C_1, C_2, \alpha, \beta, \gamma >0$ and do not depend on $x, r, t$.
		Combining \eqref{e:bl2}, \eqref{e:bl3} and letting $t \downarrow 0$, we obtain \eqref{e:bl1}. 
		
		Without loss of generality, it suffices to consider the case $W(f)<\infty$. Let $\eps>0$ be arbitrary, we choose $r_0 >0$ (depending on $f$ and $\eps$) such that 
		\begin{equation} \label{e:bl4}
			W(f,r_0)-\epsilon \le	W(f) \le   \sup_{0 < r \le r_0} W(f,r) \le W(f,r_0) + \eps.
		\end{equation}
	Similarly we follow \cite[Proof of Theorem 4.2]{GHL03} by considering dyadic annuli as above and using \eqref{e:uhke}, \eqref{e:Phibnd}, \eqref{e:Phiscale}, the doubling property of $m$ to estimate 
		\begin{align}
			B(f,t,r_0) &= \frac{1}{2t} \sum_{k=0}^{k_0} \int_X \int_{B(x,2^{-k}r_0) \setminus B(x,2^{-k-1}r_0) } (f(x)-f(y))^2 p_t(x,y)\, m(dy)\,m(dx) \nonumber \\
			& \quad + \frac{1}{2t} \int_X \int_{ B(x,2^{-k_0-1}r_0) } (f(x)-f(y))^2 p_t(x,y)\, m(dy)\,m(dx)  \label{e:bl5}
		\end{align}
		by choosing $k_0 \in \bN$ be the largest integer such that $2^{-k_0} r_0> \Psi^{-1}(t)$. The last term above can estimated   using \eqref{e:uhke}, \eqref{e:Phibnd}, \eqref{e:Phiscale} and the volume doubling property by 
		\begin{equation} \label{e:bl6}
			\frac{1}{2t} \int_X \int_{ B(x,2^{-k_0-1}r_0) } (f(x)-f(y))^2 p_t(x,y)\, m(dy)\,m(dx) \lesssim W(f, \Psi^{-1}(t)) \lesssim W(f,r_0)+\epsilon.
		\end{equation}
		For $0 \le k \le k_0$, we have  (by using \eqref{e:uhke}, \eqref{e:reg}, \eqref{e:Phibnd}, \eqref{e:Phiscale}, \eqref{e:vd}, \eqref{e:bl4})
		\begin{align} \label{e:bl7}
			\MoveEqLeft{\frac{1}{2t} }\int_X \int_{B(sx,2^{-k}r_0) \setminus B(x,2^{-k-1}r_0) } (f(x)-f(y))^2 p_t(x,y)\, m(dy)\,m(dx) \nonumber \\
			& \lesssim   W(f,2^{-k} r_0)\frac{\Psi(2^{-k  r_0})m(B(x,2^{-k} r_0))}{ \Psi(2^{-k_0 r_0})m(B(x,2^{-k_0} r_0))} \exp \left( -c_1 t \Phi\left(  \frac{2^{-k} r_0}{t}\right)\right) \nonumber \\
			& \lesssim     W(f,2^{-k} r_0)\frac{\Psi(2^{-k  r_0})m(B(x,2^{-k} r_0))}{ \Psi(2^{-k_0 r_0})m(B(x,2^{-k_0} r_0))} \exp \left( -c_2 t \Phi\left(  \frac{2^{-(k-k_0)} \Psi^{-1}(t)}{t}\right)\right) \nonumber \\
			& \lesssim   W(f,2^{-k} r_0) 2^{(k_0-k) \alpha}\exp \left( -c_3 2^{(k_0-k) \gamma} \right) \nonumber \\
			&\lesssim  2^{(k_0-k) \alpha}\exp \left( -c_3 2^{(k_0-k) \gamma} \right) (W(f,r_0) + \eps),
		\end{align}
		where $C_3, \alpha, \gamma, c_3$ depends only on the constants associated with \eqref{e:uhke}, \eqref{e:reg}, \eqref{e:Phibnd}, \eqref{e:Phiscale}, \eqref{e:vd}. Combining \eqref{e:bl5}, \eqref{e:bl6} and \eqref{e:bl7}, we obtain 
		\begin{equation} \label{e:bl8}
			B(f,t,r_0) \lesssim (W(f,r_0) + \epsilon) \left( 1+ \sum_{l=0}^\infty 2^{l \alpha } \exp(-c_3 2^{l \gamma})\right)  \le C_3 (W(f,r_0) + \epsilon)  
		\end{equation}
		for any $t \in (0,\Psi(r_0)/2)$ (note that $k_0 \to \infty$ as $t \downarrow 0$). 
		Using \eqref{e:bl8}, \eqref{e:bl4} and letting  $\epsilon \downarrow 0$, we obtain 
		\[
		\limsup_{t \downarrow 0} B(f,t,r_0) \le C_3 W(f).
		\]
		This along with \eqref{e:blh}, \eqref{e:bl0}  and \eqref{e:bl1}  implies the desired conclusion.
		\item  
		If suffices to consider the case $f \in L^2(X,m) \cap L^\infty(X,m) \cap \sF$ since any function $f \in \sF$ can be approximated by a sequence of bounded functions as given in Definition \ref{d:EnergyMeas}. 
		For any $\delta>0$ consider a function $\phi \in C_c(X) \cap \sF$ such that $\phi \equiv 1$ on $K$ and $\phi \equiv 0$ on $K_\delta^c$ (such a function exists by the regularity of the Dirichlet form). 
		We estimate the energy measure of $K$ by 
		\begin{align} \label{e:bl9}
			\MoveEqLeft{	\Gamma(f,f)(K)  \le \int \phi \,d\Gamma(f,f) = \sE(\phi f, f)- \frac{1}{2} \sE(\phi,f^2) }\nonumber \\
			& \stackrel{\eqref{e:bl0}}{=} \lim_{t \downarrow 0} \frac{1}{2t} \Bigg( \Bigg.    \int  \int  (\phi(x)f(x)-\phi(y)f(y))(f(x)-f(y))p_t(x,y) \,m(dy)\,m(dx) \nonumber \\
			& \qquad - \frac{1}{2}\int \int  (\phi(x)-\phi(y))(f(x)^2-f(y)^2)p_t(x,y) \,m(dy)\,m(dx)  \Bigg. \Bigg) \nonumber \\
			& = \lim_{t \downarrow 0} \frac{1}{4t} \int_X\int_X (\phi(x)+\phi(y)) (f(x)-f(y))^2 p_t(x,y)\,m(dy)\,m(dx) \nonumber \\
			& = \lim_{t \downarrow 0} \frac{1}{2t} \int_X\int_X  \phi(x)(f(x)-f(y))^2 p_t(x,y)\,m(dy)\,m(dx) \quad \mbox{(by symmetry)} \nonumber \\
			& \le \lim_{t \downarrow 0} \frac{1}{2t} \int_{K_\delta}\int_X   (f(x)-f(y))^2 p_t(x,y)\,m(dy)\,m(dx).
		\end{align}
		The desired conclusion follows by breaking the inner integral  over $X$ in \eqref{e:bl9} into $B(x,r), B(x,r)^c$  and then following the same argument as in (b). \qedhere
	\end{enumerate}
\end{proof}
\section{The extension map and its scale-invariant boundedness} \label{s:extend}
In this section, we define the extension map using the reflection map and study its scale-invariant and global boundedness properties. 

\subsection{Poincar\'e inequality on uniform domains} \label{s:poin}
In this subsection, we show that uniform domains inherit Poincar\'e inequality (see Theorem \ref{t:pi}). The strategy of the proof is essentially the same as \cite{GS}. 
It is also possible to adapt the slightly different approach presented in \cite[Theorem 4.4]{BS} based on an argument of \cite[Proof of Theorem 1]{HK} that uses weak-type estimates. 
Although our approach is a straightforward adaptation of \cite{GS}, we present the proof because some of the estimates in the proof play an important role later in bounds for the extension operator (see the proof of Proposition \ref{p:eqpi}).

We recall a lemma whose proof is an exercise in duality and maximal inequality. The proof can be found in \cite[Lemma 5.3.12]{Sal}.
\begin{lemma} \label{l:dilate}(\cite[Exercise 2.10]{Hei}, \cite[Lemma 3.25]{GS})
	Let $m$ be a doubling measure on $(Y,d)$. Let $\{B(x_i,r_i): i \in I\}$ be a countable collection of balls and let $a_i \ge 0$ for all $i \in I$. Let $\lambda \ge 1$ and $1<p< \infty$. Then there exist  $C \ge 1$ depending only on the doubling constant, $\lambda$, $p$ such that
	\[
	\int_Y  \left(\sum_{i \in I} a_i \one_{B(x_i,\lambda r_i)}\right)^p \,dm \le C \int_Y  \left(\sum_{i \in I} a_i \one_{B(x_i,  r_i)}\right)^p \,dm.
	\]
\end{lemma}

The following lemma controls the difference in the average of nearby balls using the energy measure. 
It follows by an application of the Poincar\'e inequality.
\begin{lemma} (Cf. \cite[Lemma 3.22]{GS}) \label{l:pineighbor}
	Let $(X,d,m,\sE,\sF)$ be an MMD space that satisfies the Poincar\'e inequality \hyperlink{pi}{$\on{PI}(\Psi)$}	and let $m$ be a doubling measure on $X$. Let $U$ be a  $A$-uniform domain for some $A \ge 1$. Let $\eps \in (0,1/4)$ be such that $A_P \left(3+ 6 \frac{1+4\eps}{1-2\eps}\right)  \frac{\eps}{1+\eps} <1$, where $A_P$ is the constant in \hyperlink{pi}{$\on{PI}(\Psi)$}. Let $\mathfrak{R}$ be an $\epsilon$-Whitney cover of $U$.
	Let $B_U(x_i,r_i), B_U(x_j,r_j) \in \mathfrak{R}$ be such that 
	$B_U(x_i,3r_i) \cap B_U(x_j,3r_j) \neq \emptyset$. 
	Then
	\[
	\abs{f_{B_U(x_i,3r_i)}- f_{B_U(x_j,3r_j)}}^2 \le C \frac{\Psi(r_i)}{m(B(x_i,r_i))} \int_{B(x_i, A_P Lr_i)} \,d\Gamma(f,f),  
	\]
	for all $f \in \sF_{\loc}(B(x_i, A_P Lr_i))$,
	where $L= \left(3+ 6 \frac{1+4\eps}{1-2\eps}\right) \le 27$.
\end{lemma}
\begin{proof}
	By Proposition \ref{p:whitney}(c), we have $B_U(x_j, 3r_j ) \subset B_U(x_i, 3r_i+6r_j) \subset B_U(x_i,L r_i)$.
	By the volume doubling property
	\[
	m(B_U(x_i,3r_i)) \asymp  m(B_U(x_j,3r_j)) \asymp m(B_U(x_i,Lr_i) \asymp m(B(x_i,r_i)).
	\]
	Hence, we have
	\begin{align*}
		 {\abs{f_{B_U(x_i,3r_i)}- f_{B_U(x_j,3r_j)}}^2} 
		&   \le  \fint_{B_U(x_j,3r_j)} \fint_{B_U(x_i,3r_i)} \abs{f(z)-f(y)}^2\,m(dy)\,m(dz) \\
		&   \lesssim   \fint_{B_U(x_i,Lr_i)} \fint_{B_U(x_i, Lr_i)} \abs{f(z)-f(y)}^2\,m(dy)\,m(dz) \\
		& \lesssim  \frac{1}{m(B(x_i,r_i))}\int_{B(x_i,Lr_i)} \abs{f-f_{B(x_i,Lr_i)}}^2 \,dm\\
		& \lesssim \frac{\Psi(r_i)}{m(B(x_i,r_i))} \int_{B(x_i, A_P Lr_i)} \,d\Gamma(f,f).
	\end{align*}
\end{proof}

The following Poincar\'e inequality is the main result of this subsection.
\begin{theorem} \label{t:pi}
	Let $(X,d,m,\sE,\sF)$ be an MMD space that satisfies the Poincar\'e inequality \hyperlink{pi}{$\on{PI}(\Psi)$}	and let $m$ be a doubling measure on $X$. Let $U$ be a  $A$-uniform domain for some $A \ge 1$. Then there exist  
	constants $C_{U},A_U\ge 1$ such that 
	for all $(x,r)\in X\times(0,\infty)$ and all $f \in \mathcal{F}_{\loc}(B_U(x,A_Ur))$,
	\begin{equation} \tag*{$\operatorname{PI}(\Psi)$}
		\int_{B_U(x,r)} (f -   f_{B_U(x,r)})^2 \,dm  \le C_{U} \Psi(r) \int_{B_U(x,A_U r)}d\Gamma(f,f),
	\end{equation}
	where $f_{B_U(x,r)}:= m(B_U(x,r))^{-1} \int_{B_U(x,r)} f\, dm$.
	
\end{theorem}
\begin{proof}

	\noindent\textbf{Case 1}: 
	$x \in \overline{U}, r \ge \delta_U(x)$.
	Let $\eps \in (0,1/4)$ be small enough such that $27 A_P \frac{\eps}{1+\eps}<1$ so that the assumption of Lemma \ref{l:pineighbor} is satisfied. Let $\mathfrak{R}$ be an $\epsilon$-Whitney cover of $U$.
	By Lemma \ref{l:central} we can choose a central ball $B_0=B_U(x_0,r_0)$ such that $r_0$ satisfies \eqref{e:bndr0}. 
	By \eqref{e:incRB}, we have
	\begin{align} \label{e:pi1}
		\MoveEqLeft{	\int_{B_U(x,r)} \abs{f-f_{B_U(x,r)}}^2 \,dm } \nonumber\\
		&\le 	\int_{B_U(x,r)} \abs{f-f_{B_U(x_0,3r_0)}}^2 \,dm \nonumber \\
		& \le \sum_{B_U(x_D,r_D) \in \mathfrak{R}(B_U(x,r))}  \int_{B_U(x_D,3r_D)}2  \abs{f-f_{B_U(x_0,3r_0)}}^2 \,dm \quad \mbox{(by \eqref{e:incRB})} \nonumber  \\
		&\le \sum_{B_U(x_D,r_D) \in \mathfrak{R}(B_U(x,r))}  \int_{B_U(x_D,3r_D)}4 \left( \abs{f-f_{B_U(x_D,3r_D)}}^2 \right. \nonumber \\ & \qq\qq\qq\qq\qq\qq\qq\qq \left. +\abs{f_{B_U(x_0,3r_0)}-f_{B_U(x_D,3r_D)}}^2 \right) \, dm. 
	\end{align}
	The first term above can be bounded using \hyperlink{pi}{$\operatorname{PI}(\Psi)$} as
	\begin{equation} \label{e:pi2}
		\int_{B_U(x_D,3r_D)}    \abs{f-f_{B_U(x_D,3r_D)}}^2\, dm \le   C_P \Psi(3r_D)  \int_{B_U(x_D,3A_Pr_D)}d\Gamma(f,f),
	\end{equation}
	where $A_P \ge 1$ be the constant as given in \hyperlink{pi}{$\operatorname{PI}(\Psi)$}.
	For the second term we use Lemma \ref{l:string} as follows.
	By Lemma \ref{l:string}, for any $D=B_U(x_D,r_D) \in \mathfrak{R}(B(x,r))$, there exists a finite collection of distinct balls $\bS(D)= \{B_j^D=B(x_j^D,r_j^D): 0 \le j \le l\}$ such that $B_0^D=B, B_l^D=D$, where $l=l(D)$ that satisfy the properties in Lemma \ref{l:string}(a) and (b). By \eqref{e:wu3}, there exists $C_2>1$ that depends only on $A,\epsilon$ such that 
	\begin{equation} \label{e:pi3}
		D \subset B_U(x_j^D, C_2 r_j^D).
	\end{equation}
	We bound the second term in \eqref{e:pi1}  using Lemma \ref{l:pineighbor} as
	\begin{align} \label{e:pi4}
		\MoveEqLeft{\abs{f_{B_U(x_0,3r_0)}-f_{B_U(x_D,3r_D)}} \one_D} \nonumber \\ & \stackrel{\eqref{e:pi3}}{\le} \sum_{j=1}^{l(D)} \abs{f_{B_U(x_{j-1}^D,3r_{j-1}^D)}-f_{B_U(x^D_j,3r_j^D)}}   \one_D  \one_{B_U(x_j^D, C_2 r_j^D)} \nonumber \\
		& \lesssim \sum_{j=1}^{l(D)} \left( \frac{\Psi(r_j^D)}{m(B_U(x_j^D,r_j^D))} \int_{B(x_j^D,A_PLr_j^D)} d\Gamma(f,f)\right)^{1/2}\one_D  \one_{B_U(x_j^D, C_2 r_j^D)} \nonumber \\
		& \lesssim  \sum_{B_U(x_j,r_j) \in \mathfrak{R}_1(B_U(x,r))}  a_j\one_D  \one_{B_U(x_j , C_2 r_j)},
	\end{align} 
	where $a_j:=\left( \frac{\Psi(r_j )}{m(B_U(x_j,r_j))} \int_{B(x_j,A_PLr_j)} d\Gamma(f,f)\right)^{1/2}$.
	Note that 
	\begin{align} \label{e:pi5}
		\MoveEqLeft{\sum_{B_U(x_D,r_D) \in \mathfrak{R}(B_U(x,r))} \int_{B_U(x_D,3r_D)}    \abs{f_{B_U(x_0,3r_0)}-f_{B_U(x_D,3r_D)}}^2 \, dm} \nonumber \\
		&=  \sum_{B_U(x_D,r_D) \in \mathfrak{R}(B_U(x,r))}  m(B_U(x_D,3r_D))    \abs{f_{B_U(x_0,3r_0)}-f_{B_U(x_D,3r_D)}}^2 \nonumber \\
		& \stackrel{\eqref{e:vd}}{\le} D_0^2     \sum_{B_U(x_D,r_D) \in \mathfrak{R}(B_U(x,r))}  m(B_U(x_D, r_D))    \abs{f_{B_U(x_0,3r_0)}-f_{B_U(x_D,3r_D)}}^2   \nonumber \\
		& \le D_0^2 \int   \sum_{B_U(x_D,r_D) \in \mathfrak{R}(B_U(x,r))}   \abs{f_{B_U(x_0,3r_0)}-f_{B_U(x_D,3r_D)}}^2 \one_{B_U(x_D,r_D)}\, dm \nonumber \\
		& \stackrel{\eqref{e:pi4}}{\lesssim } \int \sum_{D \in \mathfrak{R}(B_U(x,r))} \left( \sum_{B_U(x_j,r_j) \in \mathfrak{R}_1(B_U(x,r))} a_j \one_D  \one_{B_U(x_j, C_2 r_j)} \right)^2 \,dm \nonumber \\
		&\lesssim \int \left( \sum_{D \in \mathfrak{R}(B_U(x,r))} \one_{D} \right) \left( \sum_{B_U(x_j,r_j) \in \mathfrak{R}_1(B_U(x,r))} a_j  \one_{B_U(x_j, C_2 r_j)} \right)^2 \,dm.
	\end{align}
	Combining $A_P L > \epsilon^{-1}$, \eqref{e:wu2} and Proposition \ref{p:whitney}(d), there exist $C,C'>0$ such that 
	\begin{equation} \label{e:pi6}
	\sum_{B_U(x_j,r_j) \in \mathfrak{R}_1(B_U(x,r))} \one_{B_U(x_j,A_PLr_j)} \le C' \one_{B_U(x,Cr)}
	\end{equation}
	By Definition \ref{d:whitney}(a)   the balls in $\mathfrak{R}$ are disjoint and hence $\sum_{D \in \mathfrak{R}(B_U(x,r))} \one_D \le 1$. 
	Therefore by \eqref{e:pi5} for any $f \in \sF_{\loc}(B_U(x,Cr))$
	\begin{align} \label{e:pi7}
		\MoveEqLeft[6]{\sum_{B_U(x_D,r_D) \in \mathfrak{R}(B_U(x,r))} \int_{B_U(x_D,3r_D)}    \abs{f_{B_U(x_0,3r_0)}-f_{B_U(x_D,3r_D)}}^2 \, dm} \nonumber \\
		&\lesssim  \int \left( \sum_{B_U(x_j,r_j) \in \mathfrak{R}_1(B_U(x,r))} a_j  \one_{B_U(x_j, C_2 r_j)} \right)^2 \,dm  \nonumber  \\
		&\lesssim  \int \left( \sum_{B_U(x_j,r_j) \in \mathfrak{R}_1(B_U(x,r))} a_j  \one_{B_U(x_j,  r_j)} \right)^2 \,dm  \quad \mbox{(by Lemma \ref{l:dilate})} \nonumber \\
		&\lesssim  \sum_{B_U(x_j,r_j) \in \mathfrak{R}_1(B_U(x,r))}  \Psi(r_j) \int_{B(x_j,A_PLr_j)} d\Gamma(f,f) \nonumber \\
		&\stackrel{\eqref{e:wu2}}{\lesssim} \Psi(r) \sum_{B_U(x_j,r_j) \in \mathfrak{R}_1(B_U(x,r))}   \int_{B(x_j,A_PLr_j)} d\Gamma(f,f) \nonumber \\ 
		&\lesssim  \Psi(r) \int_{B_U(x,C r)} d\Gamma(f,f) \quad \mbox{(by \eqref{e:pi6}).}  
	\end{align}
	Combining the above estimate with \eqref{e:pi1}, \eqref{e:pi2} and \eqref{e:pi6}, we obtain the desired Poincar\'e inequality in this case.
	
	\noindent\textbf{Case 2}:  $x \in \overline{U}, A_P r \le \delta_U(x)$. In this case we have the desired Poincar\'e inequality by \hyperlink{pi}{$\on{PI}(\Psi)$} if we choose $A_U \ge A_P$.
	
	\noindent\textbf{Case 3}:   $x \in U,  A_P^{-1} \delta_U(x) < r <\delta_U(x)$.
	By considering the ball $B_U(x,A_Pr)$ and using case 1, we obtain Poincar\'e inequality by choosing $A_U=A_P C$.
\end{proof}
\begin{remark} \label{r:central2}
	As observed in Remark \ref{r:central1}, the assumption that $B_0= B_U(x_0,r_0)\in \mathfrak{R}(B_U(x,r))$ is central can be replaced with the condition that $r_0 \ge c_0 r$ for some $c_0 >0$ for the proof of \eqref{e:pi7}.  This follows from Remark \ref{r:central1} as Lemma \ref{l:string} still holds in this setting. 
\end{remark}
\subsection{Extension map and its boundedness} \label{ss:extend}
In this section, we define the extension map $E_Q:L^2(U) \to L^2(X)$ and obtain some of its basic properties.  
We will often make the following assumptions.
\begin{assumption} \label{a:main}
	Let   $(X,d,m,\sE,\sF)$ be an MMD space that satisfies the heat kernel estimate \hyperlink{hke}{$\on{HKE(\Psi)}$} for some scale function $\Psi$ and let $m$ be a doubling measure. Hence by Theorem \ref{t:hke}, $(X,d,m,\sE,\sF)$ satisfies \hyperlink{pi}{$\on{PI(\Psi)}$} and \hyperlink{cs}{$\on{CS(\Psi)}$} (and hence also  \hyperlink{cap}{$\on{cap}(\Psi)_\le$}).  Let $U$ be a uniform domain in $(X,d)$. Let $\eps>0$ be such that 
	\[
	27 A_P \eps < 1,
	\]
	where $A_P \ge 1$ be a constant such that  \hyperlink{pi}{$\on{PI(\Psi)}$}  holds. 
	Let $\mathfrak{R}, \mathfrak{S}$ denote $\epsilon$-Whitney covers of $U$ and $V:=(U^c)^\circ= (\overline U)^c$ respectively.
	Let $Q:   \wt{\mathfrak{S}} \to \mathfrak{R}$ be as given in Proposition \ref{p:reflect}. 
	Let $\{\psi_{B}: B \in \mathfrak{S}\}$ be a partition of unity of $V$ with controlled energy as given in Lemma \ref{l:partition}; that is,
	\[
	\sE(\psi_B,\psi_B) \lesssim \frac{m(B)}{\Psi(r(B))}.
	\]
	We define the extension map $E_Q: L^2(U,m) \to L^2(X,m)$ as
	\begin{equation} \label{e:EQ}
		E_Q( f) (x) = \begin{cases}
			f(x) &\mbox{if $x \in U$,} \\
			\sum_{B \in \wt{ \mathfrak{S}} }  \psi_B(x) \fint_{3 Q(B)} f \,dm & \mbox{otherwise},
		\end{cases}
	\end{equation}
	where $3Q(B)$ denotes the ball with the same center but three times the radius of that of $Q(B)$. In the case $V = \emptyset$, we interpret the above sum over an empty-set in \eqref{e:EQ} as $0$.
\end{assumption}
Recall from \eqref{e:defSB} the notation
\begin{equation} \label{e:locS}
	{\mathfrak{S}}(B(\xi,r))= \set{B_V(y,s) \in \mathfrak{S}: B_V(y,6s) \cap B(\xi,r)\neq \emptyset }, \quad  \wt{\mathfrak{S}}(B(\xi,r))= {\mathfrak{S}}(B(\xi,r)) \cap \wt{\mathfrak{S}}
\end{equation}
so that 
$E_Q f \equiv  \sum_{B \in \wt{\mathfrak{S}}(B(\xi,r)) }  \psi_B(x) \fint_{3 Q(B)} f \,dm$ on $B(\xi,r) \setminus U$ (since $\psi_B$ is supported on $6B$).

At this point, it is not clear why $E_Q (f)$ defined in \eqref{e:EQ} belong to $L^2(X,m)$ whenever $f \in L^2(U,m)$.
The following lemma shows that the extension operator enjoys a \emph{scale invariant}   boundedness property with respect to $L^2$ norm.
By \eqref{e:bdloc} below,  $E_Q$ can be viewed as a bounded operator from $L^2(B(\xi,A_1r) \cap U)$ to $L^2(B(\xi,r))$ where the bound is independent of $r>0$ and $\xi \in \partial U$. This motivates our terminology \emph{scale invariant} boundedness.
\begin{lemma} \label{l:l2b}
	Let $m$ be a doubling measure on $(X,d)$ and let $U$ be a uniform domain. Let $E_Q$ be the extension map as defined in \eqref{e:EQ}.
	The extension map $E_Q: L^2(U,m) \to L^2(X,m)$ is a bounded linear operator. There exists $C_1,A_1 \in (1,\infty)$ such that for any $f \in L^2(U,m), \xi \in \partial U, r >0$, we have
	\begin{equation} \label{e:bdloc}
		\int_{B(\xi,r)} \abs{E_Qf}^2\,dm \le C_1 \int_{B(\xi,A_1 r)\cap U} \abs{f}^2\,dm.
	\end{equation}
\end{lemma}
\begin{proof}
	The linearity of $E_Q$ is evident from the definition. 
	The boundedness of $E_Q$ follows from \eqref{e:bdloc} by letting $r \to \infty$.
	
	It remains to show \eqref{e:bdloc}. 
	By Propositions \ref{p:reflect}(a,b) and \ref{p:whitney}(d),
	there exist $A_1, C_0 \in (1,\infty)$ such that
	\begin{equation} \label{e:lb0}
		\sum_{B \in \wt{\mathfrak{S}}} \one_{3 Q(B)} \le C_0 \one_{B_U(x, A_1 r)}.
	\end{equation}
	Since $\{3B:B \in \wt{\mathfrak{S}}\}$ is a cover of $V$, we have
	\begin{align} \label{e:lb1}
		\int_{B(\xi,r)}\abs{E_Q f}^2 \, dm & \le \int_{U\cap B(\xi,r)} \abs{f}^2\,dm + \int_{V} \abs{\sum_{B \in \wt{\mathfrak{S}}(B(\xi,r)) } \psi_B(\cdot) \fint_{3Q(B)} f \,dm }^2 \,dm  \nonumber \\
		& \le \int_{U\cap B(\xi,r)} \abs{f}^2\,dm   + \int_{V} \abs{\sum_{B  \in \wt{\mathfrak{S}}(B(\xi,r)) } \one_{6B}(\cdot) \fint_{3Q(B)} f \,dm }^2 \,dm.  
	\end{align}
	For the second term we estimate using Proposition \ref{p:whitney}(d) and $\epsilon < 1/6$ we obtain  $$\sum_{B_V(x_i,r_i) \in \wt{\mathfrak{S}}(B(\xi,r)) }\one_{B_V(x_i,6r_i)} \lesssim 1,$$    the Cauchy-Schwarz inequality and volume doubling, to obtain 
	\begin{align} \label{e:lb2}
		\MoveEqLeft{ \int_{V} \abs{\sum_{B \in \wt{\mathfrak{S}}(B(\xi,r)) } \one_{B} \fint_{3Q(B)} f \,dm }^2 \,m(dx) } \nonumber \\
		& \lesssim \int_{V}\sum_{B  \in \wt{\mathfrak{S}}(B(\xi,r)) } \one_{6B}  \abs{\fint_{3Q(B)} f \,dm }^2 \,m(dx) \nonumber \\
		&\lesssim  \sum_{B \in \wt{\mathfrak{S}}(B(\xi,r))}  m(Q(B)) \abs{\fint_{3Q(B)} f \,dm }^2  \quad \mbox{(by  \eqref{e:prefa})} \nonumber \\
		& \lesssim \sum_{B \in \wt{\mathfrak{S}}(B(\xi,r))} \int_{3Q(B)} f^2\,dm \q \mbox{(by Cauchy-Schwarz inequality)} \nonumber \\
		& \lesssim \int_{U \cap B(\xi,A_1r)}  f^2 \,dm \quad \mbox{(by \eqref{e:lb0}).}
	\end{align}
	Combining \eqref{e:lb1} and \eqref{e:lb2}, we obtain \eqref{e:bdloc}. 
\end{proof}

\subsection{Energy bounds on the extension map} \label{s:heart}
We recall an elementary lemma about bounded degree graphs that is used to obtain energy bounds on the extended function. A  unweighted version of this lemma is contained in \cite[Proof of Lemma 7.5]{Soa}. 
\begin{lemma}  \label{l:fuzz}
	Let $G=(V_G,E_G)$ be a bounded degree graph with $\deg(x) \le M$ for all $x \in V_G$. Let $d_G$ denote the graph distance on $V_G$. Let $C \ge 1, m: V_G \to (0,\infty)$ be such that 
	\begin{equation}\label{e:gentle} 
	C^{-1}m(y) \le	m(x) \le Cm(y) \q \mbox{for all $x,y \in V_G$ with $d_G(x,y)=1$.}
	\end{equation}
	Then for any $L \ge 1, f: V_G \to \bR$, we have
	\begin{equation} \label{e:fuzz}
		\sum_{x \in V_G}	\sum_{\substack{y \in  V_G,\\ d_G(x,y) \le L}} \abs{f(x)-f(y)}^2 m(x) \le 	 L C^L M^{2(L+1)} \sum_{x \in V}	\sum_{\substack{y \in  V_G,\\ d_G(x,y) \le 1}} \abs{f(x)-f(y)}^2 m(x).
	\end{equation}
\end{lemma}
\begin{proof}
	For every ordered pair $(x,y) \in V_G \times V_G$ such that $d_G(x,y) \le L$, we pick a path $\gamma(x,y)$ of vertices $x=x_0,\ldots,x_n=y$ such that $n=d_G(x,y) \le L$. By the Cauchy-Schwarz inequality and \eqref{e:gentle}, we have 
	\begin{align} \label{e:fz1}
		\abs{f(x)-f(y)}^2 m(x) &\le d_G(x,y) \sum_{i=0}^{d_G(x,y)-1} \abs{f(x_i)-f(x_{i+1})}^2 m(x) \\
		& \le L C^L \sum_{i=0}^{d_G(x,y)-1} \abs{f(x_i)-f(x_{i+1})}^2 m(x_i).
	\end{align}
	For each edge $(z_1,z_2)$ the term $\abs{f(z_i)-f(z_{i+1})}^2 m(z_i)$ arises from at most $(1+M+M^2 +\ldots+M^{L})^2 \le M^{2(L+1)}$ (for all $M \ge 2$) pairs $(x,y) \in V \times V$ such that $d_G(x,y)\le L$. By summing \eqref{e:fz1} over all such pairs $(x,y)$ and using this observation, we obtain \eqref{e:fuzz}.
\end{proof}

Next, we would like to obtain an analogue of Lemma \ref{l:l2b} with $L^2$ norms replaced by energy measures. The following proposition provides such a scale-invariant boundedness estimate for energy of the extension. 
In \eqref{e:eqpi} below, we obtain Poincar\'e type inequality where the variance is over a ball on $X$ end the Dirichlet energy is over a ball on $U$. This variant of Poincar\'e inequality  is the key   ingredient of this work and plays an important role in estimating the energy of the extension $E_Q(f)$.
\begin{proposition} \label{p:eqpi} 
	Let   $(X,d,m,\sE,\sF)$ be an MMD space that satisfies the heat kernel estimate \hyperlink{hke}{$\on{HKE(\Psi)}$} for some scale function $\Psi$ and let $m$ be a doubling measure. Let $U$ be a uniform domain in $(X,d)$ and let $E_Q$ be the extension map as defined in \eqref{e:EQ} satisfying Assumption \ref{a:main}. Let $V= (U^c)^\circ$. Then we have the following:
	\begin{enumerate}[(a)]
		\item 
		There exist  $c_0 \in (0,1), K_0,C_1 \in (1,\infty)$ such that for any $\xi \in \partial U, r< c_0 \diam(U)$, such that 
		\begin{equation} \label{e:eqpi}
			\inf_{\alpha \in \bR}\int_{B(\xi,r)} \abs{E_Q f - \alpha}^2 \, dm \le C_1 \Psi(r) \int_{B(\xi,K_0 r) \cap U} d\Gamma_U (f,f), \q \mbox{for any $f \in \sF(U)$}.
		\end{equation}
		\item For any $f \in L^2(U)$, we have $E_Q(f) \in \sF_{\loc}(V)$. 
		There exits $c_1 \in (0,1), C_1, C_2, K_1 \in (1,\infty)$ such that for all $f \in \sF(U), \xi \in \partial U, 0< r < c_1 \diam(U)$, we have 
		\begin{equation} \label{e:ve1}
			\Gamma(E_Q(f),E_Q(f))(B_V(\xi,r)) \le C_1 \Gamma_U(f,f)(B_U(\xi,K_1 r)),
		\end{equation}
		and 
		\begin{equation} \label{e:ve2}
			\Gamma(E_Q(f),E_Q(f))(V) \le C_2 \left(\sE_U(f,f) + \frac{1}{\Psi(\diam(U))} \int_U f^2 \,dm \right),
		\end{equation}
		where by convention that $\frac{1}{\Psi(\diam(U))} =0$ if $\diam(U)=\infty$.
		\item There exists $C_3 \in (1,\infty)$ such that for any $f \in \sF(U)$, we have $E_Q(f) \in \sF$ and 
		\begin{equation} \label{e:enb}
			\sE(E_Q(f),E_Q(f)) \le  C_3 \left(\sE_U(f,f) + \frac{1}{\Psi(\diam(U))} \int_U f^2 \,dm \right).
		\end{equation}
		\item For all $f \in \sF(U)$, we have
		\begin{equation} \label{e:emb}
			\Gamma(E_Q(f),E_Q(f))(\partial U) =0.
		\end{equation} 
		\item There exists $C_4, K_2 \in (1,\infty), c_2 \in (0,1)$ all $f \in \sF(U)$ such that for all $x \in \overline{U}, 0<r<c_2 \diam(U,d)$, 
		\begin{equation} \label{e:se}
			\Gamma(E_Q(f),E_Q(f))(B(x,r)) \le C_4 \Gamma_U(f,f)(B_U(x,K_3r)).
		\end{equation}
	\end{enumerate}
\end{proposition}
\begin{proof}
	\begin{enumerate}[(a)]
		\item 
		Let $B_0=B_U(x_0,r_0) \in \mathfrak{R}(B_U(\xi,r))$ be a central ball as given by Lemma \ref{l:central}. 
		By \eqref{e:pi1}, \eqref{e:pi2} and \eqref{e:pi6} in the proof of Theorem \ref{t:pi}, we obtain 
		\begin{align} \label{e:eqp1} \nonumber
			\int_{B(\xi,r) \cap U} \abs{E_Q f - f_{B_U(x_0,3r_0)}}^2 \, dm &=	\int_{B(\xi,r) \cap U} \abs{ f - f_{B_U(x_0,3r_0)}}^2 \, dm \\
			&\lesssim  \Psi(r) \int_{B(\xi,K r) \cap U} d\Gamma_U(f,f).
		\end{align}
		Let $\wt{\mathfrak{S}}(B(\xi,r))$ be as defined in \eqref{e:locS} so that  
		$$E_Q f(x) =  \sum_{B \in \wt{\mathfrak{S}}(B(\xi,r)) }  \psi_B(x) \fint_{3 Q(B)} f \,dm \quad \mbox{ for all  } x\in B(\xi,r) \setminus U.$$ By \eqref{e:prefa}, there exist $K_1>0$ such that
		\begin{equation} \label{e:eqp2}
			\{Q(B): B  \in \wt{\mathfrak{S}}(B(\xi,r))\} \subset \mathfrak{R}(B(\xi,K_1r)),
		\end{equation}
		where $\mathfrak{R}(B(x,K_1r))$ is as defined in \eqref{e:defRB}.
		We estimate
		\begin{align}
			\MoveEqLeft{\int_{B_V(\xi,r)} \abs{E_Qf - f_{B_U(x_0,3r_0)}}^2 \, dm} \nonumber \\ & \le \int_{B_V(\xi,r)}  \abs{\sum_{B \in \wt{\mathfrak{S}}(B(\xi,r))} (f_{3Q(B)} - f_{B_U(x_0,3r_0)})  \psi_B(\cdot)}^2 \, dm \nonumber \\
			& \lesssim \int_{B_V(\xi,r)}  \sum_{B \in \wt{\mathfrak{S}}(B(\xi,r))} (f_{3Q(B)} - f_{B_U(x_0,3r_0)})^2  \psi_B(\cdot)^2 \, dm  \nonumber \\ &\quad \mbox{(by Proposition \ref{p:whitney}(d))}\nonumber \\
			&\lesssim   \sum_{B \in \wt{\mathfrak{S}}(B(\xi,r))} (f_{3Q(B)} - f_{B_U(x_0,3r_0)})^2   m(B) \nonumber \\ &\quad \mbox{(by doubling and $\psi_{B}(\cdot) \le \one_{6B}(\cdot)$)}\nonumber \\
			&\lesssim  \sum_{B \in \wt{\mathfrak{S}}(B(\xi,r))} (f_{3Q(B)} - f_{B_U(x_0,3r_0)})^2   m(Q(B))   \quad \mbox{(by \eqref{e:vd} and \eqref{e:prefa})} \nonumber \\
			&\lesssim  \sum_{B \in {\mathfrak{R}}(B_U(\xi,K_1 r))} (f_{3B} - f_{B_U(x_0,3r_0)})^2   m(B) \q \mbox{(Proposition \ref{p:reflect}(b), \eqref{e:eqp2})} \nonumber \\
			&\lesssim \Psi(r) \Gamma(f,f)(B_U(\xi, K_2 r)) \nonumber
		\end{align}
		The last line above follows from the same argument as \eqref{e:pi7} as explained in Remark \ref{r:central2} (recall that $27 A_P \eps <1$ from Assumption \ref{a:main}).
		\item For any open set $B_V(y_0,s_0) \in \mathfrak{S}$ and $f \in \sF(U)$, by Lemma \ref{l:partition}(b)  we obtain  
		\begin{equation}\label{e:eqp2b}
		E_Q (f)(\cdot) \equiv \sum_{B_V(y,s) \in  \wt{\mathfrak{S}}(B_V(y_0,3s_0))} f_{3 Q(B_V(y,s))} \psi_{B_V(y,s)}(\cdot) \q \mbox{ on $B_V(y_0,3s_0)$.}
		\end{equation}
		By Lemma \ref{l:reflect}(b), there exist $N_1 \in \bN$ such that 
		$\# \wt{\mathfrak{S}}(B_V(y_0,3s_0)) \le N_1$ for all  $B_V(y_0,s_0) \in \mathfrak{S}$. Therefore $E_Q(f) \in \sF (B_V(y_0,3s_0))$  for all  $B_V(y_0,s_0) \in \mathfrak{S}$ and consequently by Definition \ref{d:whitney}(iii) we conclude $E_Q(f) \in \sF_{\loc}(V)$.

		By \eqref{e:eqp2b}, strong locality, and Lemma \ref{l:partition}(b), we have the estimate 
		\begin{align} \label{e:eqp3}
			\Gamma(E_Q(f),E_Q(f))(B_V(y_0,3s_0)) &\le N_1 \inf_{\alpha \in \bR} \sum_{B \in \wt{\mathfrak{S}}(B_V(y_0,3s_0))} \abs{f_{3 Q(B)}-\alpha}^2 \sE(\psi_{B},\psi_{B}) \nonumber \\
			&\stackrel{\eqref{e:capbnd}}{\lesssim} \inf_{\alpha \in \bR} \sum_{B \in \wt{\mathfrak{S}}(B_V(y_0,3s_0))} \abs{f_{3 Q(B)}-\alpha}^2  \frac{m(B) }{\Psi(r(B))} \nonumber \\
			& \lesssim \inf_{\alpha \in \bR} \sum_{B \in \wt{\mathfrak{S}}(B_V(y_0,3s_0))} \abs{f_{3 Q(B)}-\alpha}^2  \frac{m(Q(B)) }{\Psi(r(Q(B)))},
		\end{align}
		where we use Proposition \ref{p:reflect}(a), \eqref{e:reg} and \eqref{e:vd} in the last line above.
		
		By Lemma \ref{l:reflect}(a,e), we choose  $c_0 \in (0,1)$ such that for all $\xi \in \partial U, 0<r < c_0 \diam(U), B_1 \in \mathfrak{S}(B_V(\xi,r)), B_2 \in \mathfrak{S}(B_1)$, we have  
		\begin{equation} \label{e:eqp4}
			B_1 \in \wt{\mathfrak{S}}, \q B_2 \in \wt{\mathfrak{S}}.
		\end{equation}
		Since $\{3B: B \in \mathfrak{S}\}$ cover $V$,    for all $\xi \in \partial U, 0<r < c_0 \diam(U)$ we obtain
		\begin{align} \label{e:eqp5}
			\MoveEqLeft{\Gamma(E_Q(f), E_Q(f))(B_V(\xi,r)) } \nonumber\\
			&\le 	\sum_{B \in \mathfrak{S}(B_V(\xi,r))} \Gamma(E_Q(f), E_Q(f))(3B) \nonumber \\ &\stackrel{\eqref{e:eqp4}}{=}\sum_{B \in \wt{\mathfrak{S}}(B_V(\xi,r))} \Gamma(E_Q(f), E_Q(f))(3B) \nonumber\\
			&\stackrel{\eqref{e:eqp3}, \eqref{e:eqp4}}{\lesssim} \sum_{B \in    \wt{\mathfrak{S}}(B_V(\xi,r))}  \sum_{B': B' \stackrel{\mathfrak{S}}{\sim} B}  \abs{f_{3Q(B')}-f_{3Q(B)}}^2 \frac{m(Q(B))}{\Psi(r(Q(B)))}.
		\end{align}
		By Lemma \ref{l:reflect}(c,d,e), there exist  $c_1 \in (0,c_0), K_1, L_1 \in (0,\infty)$ such that for all   $\xi \in \partial U, 0 < r< c_1\diam(U),  B_1 \in \mathfrak{S}(B_V(\xi,r))$, we have   
		\begin{align} \label{e:eqp6}
			\MoveEqLeft{B_2 \in  \wt{\mathfrak{S}}(B_1), \q	Q(B_1) \in\wt{\mathfrak{R}}(B_U(\xi,K_1r)),  
		}  \nonumber \\
		\MoveEqLeft{	D_{\mathfrak{S}}(Q(B_1),Q(B_2)) \le L_1,  \q\mbox{ for all  $B_2 \in {\mathfrak{S}}(B_1)$;}} \nonumber\\
			 \MoveEqLeft{B_3 \in \wt{\mathfrak{R}}(B_U(\xi,K_1r))  \q \mbox{ for all $ B_3 \in \mathfrak{R}$ with $D_{\mathfrak{R}}(Q(B_1),B_3) \le L_1$}}.
		\end{align}
		By Proposition \ref{p:whitney}(c) and \eqref{e:vd}, there exists $C_0 \ge 1$ such that for all $B_1, B_2 \in \mathfrak{R}$ satisfying $D_{\mathfrak{R}}(B_1,B_2) = 1$, we have
		\begin{equation} \label{e:eqp4b}
			C_0^{-1}	\frac{m(B_2)}{\Psi(r(B_2))}	\le \frac{m(B_1)}{\Psi(r(B_1))} \le C_0 	\frac{m(B_2)}{\Psi(r(B_2))}.
		\end{equation}
		By \eqref{e:eqp5} and \eqref{e:eqp6}, we estimate 
		\begin{align}
			\MoveEqLeft{\Gamma(E_Q(f), E_Q(f))(B_V(\xi,r)) } \nonumber\\
			& \lesssim \sum_{B_1 \in \mathfrak{R}(B_U(\xi, K_1r))}  \sum_{\substack{B_2 \in \mathfrak{R}(B_U(\xi, K_1r)) \\ D_{\mathfrak{R}}(B_1,B_2) \le L_1}} \abs{f_{3B_1} - f_{3B_2}}^2 \frac{m(B_1)}{\Psi(r(B_1))} \nonumber \\
			& \lesssim \sum_{B_1 \in \mathfrak{R}(B_U(\xi, K_1r))}  \sum_{\substack{B_2 \in \mathfrak{R}(B_U(\xi, K_1r)) \\ D_{\mathfrak{R}}(B_1,B_2) =1}} \abs{f_{3B_1} - f_{3B_2}}^2 \frac{m(B_1)}{\Psi(r(B_1))} \nonumber \\ &\q \mbox{(by \eqref{e:eqp4b} and Lemma \ref{l:fuzz})} \nonumber \\
			& \lesssim \sum_{B_1 \in \mathfrak{R}(B_U(\xi, K_1r))} \Gamma(f,f)\left(A_P \left(3+ 6 \frac{1+4\eps}{1-2\eps}\right)  B_1 \right) \nonumber\\   &\q \mbox{(by Lemmas \ref{l:pineighbor} and \ref{l:reflect}(b))} \nonumber \\
			&\lesssim \Gamma(f,f)\left( \bigcup_{B_1 \in \mathfrak{R}(B_U(\xi, K_1r))}A_P \left(3+ 6 \frac{1+4\eps}{1-2\eps}\right)  B_1 \right) \nonumber \\ &\q \mbox{(by Proposition \ref{p:whitney}(d)).} \label{e:eqp7}
		\end{align}
		By \eqref{e:srbnd1}, we obtain
		\[
		\bigcup_{B_1 \in \mathfrak{R}(B_U(\xi, K_1r))}A_P \left(3+ 6 \frac{1+4\eps}{1-2\eps}\right)  B_1 \subset B_U(\xi, Kr), \quad 
		\]
		where $$K= K_1 \left(1+ \frac{\eps}{1-2\eps}\left( 3 + A_P\left(3+ 6 \frac{1+4\eps}{1-2\eps} \right)\right)\right).$$ This along with \eqref{e:eqp7} yields  \eqref{e:ve1}.
		
		It remains to show \eqref{e:ve2}. If $\diam(U)=\infty$, then  by letting $r \to \infty$  in  \eqref{e:ve1} and using monotone convergence theorem, we obtain $\Gamma(E_Q(f),E_Q(f))(V) \le C_1 \sE_U(f,f)$ for all $f \in \sF(U)$.  
		
		Hence it suffices to consider the case $\diam(U)< \infty$. 
		Let $c_1 \in (0,1), K_1 \in (1,\infty)$ be such that \eqref{e:ve1} holds.
		Let $r_0:= \frac{c_1}{4}\diam(U)$ and $N \subset X$ be a maximal $r_0$-separated subset of $\partial U$.
		Then 
		\begin{equation}\label{e:eqp8}
			\partial U \subset \cup_{n\in N} B(n,r_0), \qquad d(x, \partial U) \ge r_0  \quad \mbox{for all $x \in X \setminus \left(\cup_{n \in N} B(n,2r_0)\right)$.}
		\end{equation}
		By \eqref{e:ve1} and the metric doubling property, we have 
		\begin{equation} \label{e:epq9}
			\sum_{n \in N} \Gamma(E_Q(f),E_Q(f))(B_V(n,2r_0)) \lesssim \sum_{n \in N} \Gamma(f,f)(B_U(n,2K_1r_0)) \lesssim \sE_U(f,f)
		\end{equation}
		for all $f \in \sF(U)$. 
		By  \eqref{e:whitc}, \eqref{e:defStil}, and Lemma \ref{l:partition}(b), 
		for all $B_V(y,s) \in \mathfrak{S}$  
		\begin{equation} \label{e:epq1}
		s \ge \frac{1+7\eps}{1-5\eps} \frac{\eps}{6A(1+\eps)} \diam(U) \mbox{ implies }	E_Q(f) \equiv 0 \, \mbox{ on $B_V(y,3s)$.}
		\end{equation}
		By \eqref{e:whitb} and \eqref{e:eqp8}, for any $B_V(y,s) \in \mathfrak{S}$   such that $$B_V(y,3s) \cap  \left(X \setminus \left(\cup_{n \in N} B(n,2r_0)\right) \right) \neq \emptyset,$$ we have 
		\begin{equation} \label{e:epq2}
			s = \frac{\eps}{1+\epsilon} \delta_{V}(y) > \frac{\eps}{1+4\eps} r_0 = \frac{ c_1 \eps}{4(1+4\eps)} \diam(U).
		\end{equation}
		For any $B_V(y,s) \in \mathfrak{S}$ such that $s > \frac{ c_1 \eps}{4(1+4\eps)} \diam(U)$, by the metric doubling property, \eqref{e:eqp3} we have 
		\begin{align} \label{e:epq3}
			\Gamma(E_Qf,E_Qf)(B_V(y,3s)) &\stackrel{\eqref{e:eqp3}}{\lesssim}\sum_{B \in  \wt{\mathfrak{S}}(B_V(y,s))}  \frac{m(Q(B))}{\Psi(Q(B))}\abs{f_{3Q(B)}}^2 \nonumber \\
			&\lesssim  \frac{1}{\Psi(\diam(U))}  \int_U f^2 \,dm
		\end{align}
		for all $f \in L^2(U,m)$. Since the metric doubling property implies that there are only a bounded number of balls $B_V(y,s) \in \mathfrak{S}$ such that $$s > \frac{ c_1 \eps}{4(1+4\eps)} \diam(U) \mbox{ and }s < \frac{1+7\eps}{1-5\eps} \frac{\eps}{6A(1+\eps)} \diam(U).$$ Therefore by Definition \ref{d:whitney}(c), \eqref{e:epq9}, \eqref{e:epq1}, and \eqref{e:epq3}, we obtain  \eqref{e:ve2}.
		\item Our approach to show that $E_Qf \in \sF$ is to use the criterion \eqref{e:fcrit} in Theorem \ref{t:besov}. 
		To this end, we choose $N_1 \subset \partial U$ a maximal $r$-separated subset of $\partial U$ and $N \supset N_1$ such that $N$ is a maximal $r$-separated subset of $X$ (the existence of such $N_1, N$ follows from Zorn's lemma). Let $C_P, A_P \ge 0$ be the constants associated with the Poincar\'e inequality \hyperlink{pi}{$\operatorname{PI}(\Psi)$}. Define 
		\begin{equation} \label{e:ebd}
			N_2= \{n \in N: B(n,2A_pr) \cap \partial U = \emptyset\}.
		\end{equation}
		Since $\partial U \subset \cup_{n \in N_1}B(n,r),  X =\cup_{n \in N}B(n,r)$, we have 
		\begin{equation} \label{e:eb0}
			X=\left( \cup_{n \in N_1} B(n,2(A_P r+1)r)\right) \cup \left(\cup_{n \in N_2} B(n,r)\right).
		\end{equation}
		Therefore 
		\begin{align} \label{e:eb1}
			\MoveEqLeft{\one_{ \{d(x,y) <r\}}} \nonumber \\&\le \sum_{n \in N_1}\one_{B(n,(2 A_P+3) r)}(x)\one_{B(n,(2 A_P+3) r)}(y) + \sum_{n \in N_2} \one_{B(n,2 r)}(x)\one_{B(n,2r)}(y).
		\end{align}
		By \eqref{e:vd} and \eqref{e:eb1}, we have 
		\begin{align}
			\label{e:eb2}
			\MoveEqLeft{\frac{\one_{ \{d(x,y) <r\}}}{m(B(x,r))}} \nonumber\\ &\lesssim \sum_{n \in N_1}\frac{\one_{B(n,(2 A_P+3) r)}(x)\one_{B(n,(2 A_P+3) r)}(y)}{m(B(n,(2 A_P+3) r))} + \sum_{n \in N_2} \frac{\one_{B(n,2 r)}(x)\one_{B(n,2r)}(y)}{m(B(n, 2r))}.
		\end{align}
		Let $K_0 \in (1,\infty), c_0 \in (0,1)$ be such that \eqref{e:eqpi} holds.  Let $f \in \sF(U)$ and define 
		\[
		W(E_Q(f),r):= \frac{1}{\Psi(r)}\int_X   \fint_{B(x,r)} (E_Q(f)(x)-E_Q(f)(y))^2 \,m(dy) \,m(dx). 
		\]
		By \eqref{e:eb2}, for all $r>0, f \in \sF(U)$,
		\begin{align} \label{e:eb3}
			\MoveEqLeft{W(E_Q(f),r)} \nonumber \\& \stackrel{\eqref{e:eb2}}{\lesssim} \frac{1}{\Psi(r)}\sum_{n \in N_1} \int_{B(n,(2A_P+3)r)}  \abs{E_Q(f)(x)-(E_Q(f))_{B(n,(2A_P+3)r)}}^2 \,m(dx) \nonumber \\
			&\qq+ \sum_{n \in N_2} \frac{1}{\Psi(r)} \int_{B(n,2r)}  \abs{E_Q(f)(x)-(E_Q(f))_{B(n,2r)}}^2\,m(dx).
		\end{align}
		Since $\partial V \subset \partial U$, we have $B(n,2A_Pr) \subset U \cup V$ for all $n \in N_2$.
		Noting that $E_Q(f)$ belongs to $\sF_{\loc}(U \cup V)$ and using the Poincar\'e inequality \hyperlink{pi}{$\operatorname{PI}(\Psi)$} along with metric doubling property we obtain  for all $r>0, f \in \sF(U)$, 
		\begin{align} \label{e:eb4}
			\MoveEqLeft{\sum_{n \in N_2} \frac{1}{\Psi(r)} \int_{B(n,2r)}  \abs{E_Q(f)(x)-(E_Q(f))_{B(n,2r)}}^2\,m(dx)} \nonumber \\
			&\lesssim \sum_{n \in N_2} \frac{\Psi(2r)}{\Psi(r)} \int_{B(n,2A_P r)}   d\Gamma(E_Q(f),E_Q(f)) \q \mbox{(by \hyperlink{pi}{$\operatorname{PI}(\Psi)$})} \nonumber \\
			& \lesssim \Gamma_U(f,f)(U) + \Gamma(E_Q(f),E_Q(f))(V) \q \mbox{(by \eqref{e:reg} and metric doubling)} \nonumber \\
			& \lesssim \sE_U(f,f)+ \frac{1}{\Psi(\diam(U))} \int_U f^2\,dm.
		\end{align}
		For all $0<r< \frac{c_0}{2A_P+3}\diam U$, by \eqref{e:eqpi} we have
		\begin{align} \label{e:eb5}
			\MoveEqLeft{\frac{1}{\Psi(r)}\sum_{n \in N_1} \int_{B(n,(2A_P+3)r)}  \abs{E_Q(f)(x)-(E_Q(f))_{B(n,(2A_P+3)r)}}^2 \,m(dx)}  \nonumber \\
			&\lesssim \frac{\Psi((2A_P+3)r)}{\Psi(r)} \sum_{n \in N_1} \int_{B(n,(2A_P+3)K_0r) \cap U} d\Gamma(f,f) \nonumber \\
			& \lesssim \sE_U(f,f) \q \mbox{(by \eqref{e:reg} and metric doubling).} 
		\end{align}
		Combining \eqref{e:eb3}, \eqref{e:eb4} and \eqref{e:eb5}, we obtain
		\[
		\limsup_{r\downarrow 0} W(E_Q(f),r) \lesssim \sE_U(f,f) + \frac{1}{\Psi(\diam(U))} \int_U f^2\,dm, \q \mbox{for all $f \in \sF(U)$}.
		\]
		This along with Theorem \ref{t:besov}(b) completes the proof of \eqref{e:enb}.
		\item Since $X = U  \cup V \cup \partial U$, by \eqref{e:ve1} it suffices to estimate $\Gamma(E_Q(f),E_Q(f))(B(\xi,r) \cap \partial U)$. We estimate $\Gamma(E_Q(f),E_Q(f))(B(\xi,r) \cap \partial U)$ using Theorem \ref{t:besov}(c). 
		We choose $c_0,c_1 \in (0,1), K_0, K_1 \in (1,\infty)$   such that \eqref{e:eqpi} and \eqref{e:ve1} hold.

		Let $\delta>0, r>0$. As in (c) above, we choose $N_1 \subset \partial U$ a maximal $r$-separated subset of $\partial U$ and $N \supset N_1$ such that $N$ is a maximal $r$-separated subset of $X$. Let $N_2$ be as given by \eqref{e:ebd}. 
		We denote by $ (\partial U)_\delta= \set{y\in X: d(y,\partial U) < \delta}$, the $\delta$-neighborhood of $\partial U$. 
		By \eqref{e:eb0}, we have 
		\begin{align} \label{e:emb1}
			\one_{(\partial U)_\delta}(x)\one_{ \{d(x,y) <r\}} \le \sum_{n \in N_1}\one_{B(n,(2 A_P+3) r)}(x)\one_{B(n,(2 A_P+3) r)}(y) \nonumber \\
			\qq + \sum_{\substack{ n \in N_2, \\ n \in (\partial U)_{\delta+r} }} \one_{B(n,2 r)}(x)\one_{B(n,2r)}(y).
		\end{align}
		By \eqref{e:vd} and \eqref{e:emb1}, we obtain 
		\begin{align}
			\label{e:emb2}
			\frac{\one_{(\partial U)_\delta}(x)\one_{ \{d(x,y) <r\}}}{m(B(x,r))} \lesssim \sum_{n \in N_1}\frac{\one_{B(n,(2 A_P+3) r)}(x)\one_{B(n,(2 A_P+3) r)}(y)}{m(B(n,(2 A_P+3) r))}  \nonumber \\
			\qquad + \sum_{\substack{ n \in N_2, \\ n \in (\partial U)_{\delta+r} }} \frac{\one_{B(n,2 r)}(x)\one_{B(n,2r)}(y)}{m(B(n,2r))}.
		\end{align}
		By \eqref{e:emb2}, for all $r>0, f \in \sF(U)$,
		\begin{align} \label{e:emb3}
			\MoveEqLeft{\frac{1}{\Psi(r)}\int_{(\partial U)_\delta}  \fint_{B(x,r)} (E_Q(f)(x)-E_Q(f)(y))^2 \, m(dy) \, m(dx)} \nonumber \\ & \stackrel{\eqref{e:emb2}}{\lesssim} \frac{1}{\Psi(r)}\sum_{n \in N_1} \int_{B(n,(2A_P+3)r)}  \abs{E_Q(f)(x)-(E_Q(f))_{B(n,(2A_P+3)r)}}^2 \,m(dx) \nonumber \\
			&\qq+ \sum_{\substack{ n \in N_2, \\ n \in (\partial U)_{\delta+r} }} \frac{1}{\Psi(r)} \int_{B(n,2r)}  \abs{E_Q(f)(x)-(E_Q(f))_{B(n,2r)}}^2\,m(dx).
		\end{align}
		Since $\partial V \subset \partial U$, we have $B(n,2A_Pr) \subset U \cup V$ for all $n \in N_2$.
		Combining $E_Q(f) \in \sF$ with the Poincar\'e inequality \hyperlink{pi}{$\operatorname{PI}(\Psi)$} and metric doubling property, for all $r>0, f \in \sF(U)$, we obtain
		\begin{align} \label{e:emb4}
			\MoveEqLeft{\sum_{\substack{ n \in N_2, \\ n \in (\partial U)_{\delta+r} }} \frac{1}{\Psi(r)} \int_{B(n,2r)}  \abs{E_Q(f)(x)-(E_Q(f))_{B(n,2r)}}^2\,m(dx)} \nonumber \\
			&\lesssim \sum_{\substack{ n \in N_2, \\ n \in (\partial U)_{\delta+r} }} \frac{\Psi(2r)}{\Psi(r)} \int_{B(n,2A_P r)}   d\Gamma(E_Q(f),E_Q(f)) \q \mbox{(by \hyperlink{pi}{$\operatorname{PI}(\Psi)$})} \nonumber \\
			& \lesssim \Gamma_U(f,f)(U \cap (\partial U)_{\delta +  (2A_P+1)r}) + \Gamma(E_Q(f),E_Q(f))(V \cap (\partial U)_{\delta +  (2A_P+1)r}) \\
			&\qq \mbox{(by \eqref{e:reg} and metric doubling).} \nonumber
		\end{align}
		Since $V \cap (\partial U)_{\delta +  (2A_P+1)r} \subset \cup_{n \in N_1} B_V(n,\delta + 2(A_P+1)r)$, for any $\delta,r>0$ such that $\delta + 2(A_P+1)r < c_1 \diam(U)$, we obtain 
		\begin{align} \label{e:emb4b}
			\MoveEqLeft{\Gamma(E_Q(f),E_Q(f))(V \cap (\partial U)_{\delta +  (2A_P+1)r})}  \nonumber \\& \le \sum_{n \in N_1}	 \Gamma(E_Q(f),E_Q(f))(B_V(n,\delta + 2(A_P+1)r) ) \nonumber \\
			& \lesssim \sum_{n \in N_1}	 \Gamma_U(f,f)(B_U(n,K_1(\delta + 2(A_P+1)r)) ) \q \mbox{(by \eqref{e:ve1})} \nonumber \\
			& \lesssim \Gamma_U(f,f )\left(U \cap (\partial U)_{K_1(\delta + 2(A_P+1)r)} \right)  \mbox{(by metric doubling).} 
		\end{align}
		For all $0<r< \frac{c_0}{2A_P+3}\diam U$, by \eqref{e:eqpi} we have
		\begin{align} \label{e:emb5}
			\MoveEqLeft{\frac{1}{\Psi(r)}\sum_{n \in N_1} \int_{B(n,(2A_P+3)r)}  \abs{E_Q(f)(x)-(E_Q(f))_{B(n,(2A_P+3)r)}}^2 \,m(dx)}  \nonumber \\
			&\lesssim \frac{\Psi((2A_P+3)r)}{\Psi(r)} \sum_{n \in N_1} \int_{B(n,(2A_P+3)K_0r) \cap U} d\Gamma(f,f) \nonumber \\
			& \lesssim \Gamma_U(f,f) (U  \cap (\partial U)_{(2A_P+3)K_0 r}) \q \mbox{(by \eqref{e:reg} and metric doubling).} 
		\end{align}
		Combining \eqref{e:emb3}, \eqref{e:emb4}, \eqref{e:emb4b} and \eqref{e:eb5}, we obtain 
		\begin{align} 
			\MoveEqLeft{\frac{1}{\Psi(r)}\int_{(\partial U)_\delta} \frac{1}{m(B(x,r))}  \int_{B(x,r)} (E_Q(f)(x)-E_Q(f)(y))^2 \, m(dy) \, m(dx)} \nonumber \\ & \lesssim   \Gamma_U(f,f )\left(U \cap (\partial U)_{K_1(\delta + 2(A_P+1)r)} \right)  + \Gamma_U(f,f) (U  \cap (\partial U)_{(2A_P+3)K_0 r}) \nonumber
		\end{align}
		for all $f \in \sF(U)$ and for all $\delta, r>0$ such that $$r< \frac{c_0}{2A_P+3}\diam(U)\mbox{ and }\delta + 2(A_P+1)r < c_1 \diam(U).$$ 
		Therefore by dominated convergence theorem, for all $0<\delta<c_1 \diam(U), f \in \sF(U)$, we have 
		\begin{align*}
			\MoveEqLeft{\limsup_{r \downarrow 0} \frac{1}{\Psi(r)}\int_{(\partial U)_\delta} \frac{1}{m(B(x,r))}  \int_{B(x,r)} (E_Q(f)(x)-E_Q(f)(y))^2 \, m(dy) \, m(dx)} \\
			&\lesssim   \Gamma_U(f,f )\left(U \cap (\partial U)_{2K_1\delta} \right).
		\end{align*}
		By letting $\delta \downarrow 0$ and using dominated convergence theorem, we have 
		\begin{equation}\label{e:emb6}
		\lim_{\delta \downarrow 0} \limsup_{r \downarrow 0} \frac{1}{\Psi(r)}\int_{(\partial U)_\delta}    \fint_{B(x,r)} (E_Q(f)(x)-E_Q(f)(y))^2 \, m(dy) \, m(dx)=0
		\end{equation}
		for all $f \in \sF(U)$.
		Therefore by Theorem \ref{t:besov}(c) and \eqref{e:emb6}, we have
		\[
		\Gamma(E_Q(f),E_Q(f))(\overline{B(\xi,r) \cap \partial U})=0
		\]
		for all $f \in \sF(U), \xi \in \partial U, r>0$. Letting $r \to \infty$, we obtain \eqref{e:emb}.
		\item Let $c_1 \in (0,1), C_1,K_1 \in (1,\infty)$ be chosen so that \eqref{e:ve1} holds. Let $c_2 = c_1/2$. 
		If $x \in \overline{U}, 0 < r < c_2 \diam(U,d)$ satisfies $\delta_U(x) > r$, \eqref{e:se} follows since $\Gamma(E_Q(f),E_Q(f))(B(x,r))= \Gamma(f,f)(B_U(x,r))$ by strong locality. 
		
		It suffices to consider the case $x \in \overline{U}, 0 < r < c_2 \diam(U,d)$ with $\delta_U(x) \le r$. Let $\xi \in \partial U$ be such that $d(x,\xi)= \delta_U(x) \le r$. Since $B(x,r) \subset B(\xi,2r)$, we obtain 
		\begin{align*}
			\Gamma(E_Q(f),E_Q(f))(B(x,r)) &\le \Gamma(E_Q(f),E_Q(f))(B(\xi,2r))  \\
			&\le (C_1 +1)\Gamma(f,f)(B_U(\xi,2K_1 r)) \quad \mbox{( \eqref{e:ve1} and \eqref{e:emb})} \\
			&\le (C_1 +1)\Gamma(f,f)(B_U(x,(2K_1 +1)r))\\
			& \qquad  \mbox{(since $B_U(\xi,2K_1 r) \subset B_U(x,(2K_1 +1)r)$.}
		\end{align*}
		This concludes the proof of \eqref{e:se}. \qedhere
	\end{enumerate}
\end{proof}
Next, we complete the proof of Theorem \ref{t:extend}. 
\begin{proof}[Proof of Theorem \ref{t:extend}]
Let $E$ denote the operator $E_Q$ in Proposition \ref{p:eqpi}. 
The   estimate \eqref{e:ebndloc1} follows from \eqref{e:se}.
The estimates \eqref{e:ebndloc2}  follows from Lemma \ref{l:l2b} and the same argument as the proof of  \eqref{e:se}.  The estimate \eqref{e:bndglob2}  follows from \eqref{e:ebndloc2}   by letting $r \to \infty$.
The global bound on the energy  \eqref{e:bndglob1} is a consequence  of \eqref{e:ve2} and \eqref{e:emb}. 
\end{proof}

\subsection{Energy measure  on the boundary of a uniform domain} \label{ss:energyb}
In this subsection, we prove Theorem \ref{t:benergy}. 
The following  lemma shows that $E_Q$ maps continuous function with compact support to a function that has a continuous representative with compact support.
\begin{lemma} \label{l:cont}
	In the setting of Assumption \ref{a:main}, for any $f \in C_c(X)$, the function $\wt{E_Q}f$ defined as
	\[
	\wt{E_Q}f(x)= \begin{cases}
		f(x) &\mbox{if $x \in \overline{U}$,} \\
		\sum_{B \in \wt{ \mathfrak{S}} }  \psi_B(x) \fint_{3 Q(B)} f \,dm & \mbox{if $x \in V=(\overline{U})^c$},
	\end{cases}
	\]
	satisfies $\wt{E_Q}f=E_Qf$ $m$-almost everywhere and $\wt{E_Q}f \in C_c(X)$.
\end{lemma}
\begin{proof}
	Since $E_Q f \equiv \wt{E_Q}f$ on $(\partial U)^c$ and $m(\partial U)=0$ by Lemma \ref{l:doubling}, we obtain $\wt{E_Q}f=E_Qf$ $m$-almost everywhere.
	
	Next, we show that $\wt{E_Q}f$ has compact support. Let $\xi \in \partial U$. Then there exists $R >0$ such that $\supp(f) \subset B(\xi,R)$. By  \eqref{e:whitc}, \eqref{e:prefa}, Lemma \ref{l:partition}, there exists $K>0$ such that $\supp(\wt{E_Q}f) \subset B(\xi,KR)$. By the metric doubling property and completeness of $(X,d)$, every closed and bounded set is compact. Therefore $\wt{E_Q}f$ has compact support.

	It remains to show that $\wt{E_Q}f \in C_c(X)$. Since $\wt{E_Q}f \equiv f$ on $\overline{U}$, it suffices to show that $\restr{\wt{E_Q}f}{U^c} \in C_c(U^c)$. 
	By Lemmas \ref{l:reflect}(b) and \ref{l:partition}(b), for any $B_V(y,s) \in \mathfrak{S}$, the sum defining $\wt{E_Qf}$ is a finite sum of continuous functions on $B_V(y,3s)$. Therefore $\wt{E_Q}f$ is continuous on $V= \cup_{B_V(y,s) \in \mathfrak{S}} B_V(y,3s)$. 
	
	It remains to show that $\restr{\wt{E_Q}f}{U^c}$ is continuous at all points in $\partial U$. To this end, consider $\xi \in \partial U$. By  \eqref{e:whitc} and \eqref{e:prefa}, there exist $s_0>0, K_0 >0$ such that 
	\[
	\sup_{y \in B_V(\xi,s)} \abs{\wt{E_Q}f(y)-\wt{E_Q}f(\xi)} \le 	\sup_{z \in B_U(\xi,K_0 s)} \abs{f(z)-f(\xi)}, \quad \mbox{for all $\xi \in \partial U, s \in (0, s_0)$.}
	\] 
	The continuity of  $\restr{\wt{E_Q}f}{U^c}$ follows from the above estimate and the continuity of $\restr{f}{\overline{U}}$.
\end{proof}

In order to prove Theorem \ref{t:benergy}, 
we recall some elementary facts about energy measures. 
For any Borel set $B \subset X$ and for any $f,g \in \sF$, by \cite[(2.1)]{Hin10}\footnote{There is an additional factor of $2$ in \cite[(2.1)]{Hin10} due to the slightly different definition of energy measure there.} we have 
\begin{equation} \label{e:approx}
	\abs{\sqrt{\Gamma(f,f)(B)} -  \sqrt{\Gamma(g,g)(B)}}^2 \le \Gamma(f-g,f-g)(B) \le \sE(f-g,f-g).
\end{equation}
For any $f \in \sF$, by \cite[Theorem 4.3.8]{CF}\footnote{In \cite[Theorem 4.3.8]{CF} the authors did not mention that quasi-continuous version of $f$ is needed to define the pushforward measure. However, this is required because changing $f$ on a set of $m$-measure zero can affect the push-forward measure. Nevertheless, due to the smoothness of energy measures, if $\wt{f}$ is \emph{any} quasi-continuous version of $f$ the measure $\nu:=\wt{f}_*(\Gamma(f,f))$ is well-defined; that is, $\nu$ is independent of the choice of the quasi-continuous representative  \cite[Theorem 2.1.3, Lemmas 2.1.4 and 3.2.4]{FOT}.}
\[
\wt{f}_*(\Gamma(f,f)) \ll \lambda,
\]
where $\lambda$ denotes the  Lebesgue measure on $\mathbb{R}$,  $\wt{f}$ is a quasi-continuous version of $f$, $\wt{f}_*\mu$ denotes the pushforward of the measure $\mu$ under the map $\wt{f}$. In particular, 
\begin{equation} \label{e:elevel}
	\Gamma(f,f)(\wt{f}_*^{-1}(\{0\}))=0, \quad \mbox{for any $f \in \sF$ and for any quasi-continuous version $\wt{f}$ of $f$.}
\end{equation}

\begin{proof}[Proof of Theorem \ref{t:benergy}]
By \eqref{e:approx} and the regularity of the Dirichlet form it suffices to show that $\Gamma(f,f)(\partial U)=0$ for all $f \in C_c(X) \cap \sF$. 
To this end, note that $f - \wt{E_Q}(f) \equiv 0$ on $\overline{U}$ and $f-\wt{E_Q}f \in C_c(X) \cap \sF$ by Lemma \ref{l:cont}. By the continuity of  $f-\wt{E_Q}f$ and \eqref{e:elevel}, we have 
\begin{equation} \label{e:est}
	\Gamma(f-\wt{E_Q}f,f-\wt{E_Q}f)(\overline{ U}) =0, \quad \mbox{for all $f \in C_c(X) \cap \sF$}.
\end{equation}
Combining this with Proposition \ref{p:eqpi} and \eqref{e:approx}, we obtain 
\begin{align*}
	{\Gamma(f,f)(\partial U)} & \stackrel{\eqref{e:emb}}{=}	\abs{\sqrt{\Gamma(f,f)(\partial U)} -  \sqrt{\Gamma(\wt{E_Q}(f),\wt{E_Q}(f))(\partial U)}}^2 \\
	& \stackrel{\eqref{e:approx}}{\le} \Gamma(f-\wt{E_Q}(f),f-\wt{E_Q}(f))(\partial U) \stackrel{\eqref{e:est}}{=}0 \quad \mbox{for all $f \in C_c(X) \cap \sF$.}\qedhere
\end{align*}
\end{proof}
\subsection{Quasicontinuous extension of a quasicontinuous function} \label{ss:compare}
Recall from Lemma \ref{l:cont} that every continuous function in the domain of the reflected diffusion admits a continuous extension to the domain of the ambient diffusion. We will show a similar property for quasicontinuous functions. 

Let $(X,d,m,\sE,\sF)$ be an MMD space and let   $U$ be a uniform domain satisfying the assumptions of Theorem \ref{t:extend}.
We recall the notion of $1$-capacity of a set. Given an MMD space $(X,d,m,\sE,\sF)$ and a Borel set $A$, we define its \emph{$1$-capacity} as
\[
\Cap_1(A)= \inf \set{\sE(f,f) + \norm{f}_2^2 : \mbox{$f \in \sF \cap \contfunc(X)$, $f \equiv 1$ on a neighborhood of $A$}},
\]
where $\norm{f}_2$ denotes the $L^2(X,m)$ norm. Let $(\overline{U},d,m,\sE_U,\sF(U))$ denote the MMD space for the corresponding reflected diffusion and let $\Cap_1^U$ denote the $1$-capacity corresponding to the reflected Dirichlet space. 
We say that an increasing sequence of closed subsets $\{F_k\}$ of $X$ is said to be \emph{nest} for $(X,d,m,\sE,\sF)$ if $\lim_{k \to \infty} \Cap_1(X \setminus F_k)=0$.
We say that a function $u \in \sF$ is quasicontinuous with respect to $(X,d,m,\sE,\sF)$ if for any $\epsilon>0$, there exists an open subset $G \subset X$ such that $\Cap_1(G)<\epsilon$ and the restriction $\restr{u}{X \setminus G}$ is finite and continuous on $X \setminus G$.

We record a useful property of the extension map that allows us to compare potential theoretic notions of the reflected Dirichlet space with that of the ambient space. 

\begin{lemma} \label{l:open}
	Let $m$ be a doubling measure on $(X,d)$ and let $U$ be a uniform domain. Let $E_Q$ be the extension map as defined in \eqref{e:EQ}.
	The extension map $E_Q: L^2(U,m) \to L^2(X,m)$ is a bounded linear operator.
	Then there exist $c_1,c_2 \in (0,1)$ such that the following holds. For any $f \in L^2(U,m), \xi \in \partial U, 0<r< c_1 \diam(U,d)$, such that whenever $f \ge 1$ $m$-almost everywhere on $B_U(\xi,r)$, we have 
	\[
	(E_Q (f))(x) \ge 1 \quad \mbox{for $m$-almost every $x \in B(\xi, c_2r)$.}
	\]
\end{lemma}
\begin{proof}
	By Proposition \ref{p:reflect}(a), there exist  $c_1,c_2 \in (0,1)$ such that whenever $B \in {\mathfrak{S}}$ satisfies $6B \cap B(\xi, c_2r) \neq \emptyset$ for some $\xi \in \partial U, 0<r<c_1 \diam(U,d)$, then $B \in \wt{\mathfrak{S}}$ and $3Q(B) \subset B_U(\xi,r)$. This along with Lemma \ref{l:partition}(b) and \eqref{e:EQ} implies that $(E_Q(f))(x) \ge 1$ for all $x\in V \cap B(\xi,c_2r)$. 
	The desired conclusion follows form Lemma \ref{l:doubling}.
\end{proof}

 The following proposition compares basic potential theoretic notions of $1$-capacity, nest, and quasicontinuous functions between $(X,d,m,\sE,\sF)$ and $(\overline{U},d,m,\sE_U,\sF(U))$. It can be viewed as analogues of \cite[Theorem 4.4.3(ii)]{FOT} and \cite[Theorem 3.3.8(i,iv)]{CF}.
\begin{proposition} \label{p:qcts}
	Let $(X,d,m,\sE,\sF)$ be an MMD space that satisfies the heat kernel estimate \hyperlink{hke}{$\on{HKE(\Psi)}$} for some scale function $\Psi$ and let $m$ be a doubling measure.  Let $U$ be a uniform domain $U$ and let $(\overline{U},d,m,\sE_U,\sF(U))$ denote the MMD space for reflected diffusion given in Theorem \ref{t:hkeunif}. Then we have the following.
	\begin{enumerate}[(i)]
		\item There exists $C \in (1,\infty)$ such that
		\begin{equation}
			 \Cap_1^U(A) \le \Cap_1(A) \le C \Cap_1^U(A) \quad \mbox{for any Borel set $A \subset \ol{U}$}.
		\end{equation}

		\item If $\{F_k\}$ is a nest for the MMD space $(X,d,m,\sE,\sF)$, then $\{F_k \cap \overline{U}\}$ is a nest for $(\overline{U},d,m,\sE_U,\sF(U))$.
		\item A function is quasicontinuous with respect to $(\overline{U},d,m,\sE_U,\sF(U))$ if and only if it is a restriction of a quasicontinuous function with respect to $(X,d,m,\sE,\sF)$.
	\end{enumerate}
\end{proposition}
\begin{proof}
	\begin{enumerate}[(i)]
		\item Since  $\Cap_1^U(A) \le \Cap_1(A)$ is trivial by restriction of $f \in \sF$ such that $f \ge 1$ to $\ol{U}$.
		
	 Let $\epsilon>0$ and $E_Q$ denote the extension operator defined in \eqref{e:EQ} and satisfying the properties in Theorem \ref{t:extend}.
		Since  there exists an open set $G \subset \overline{U}$ with $G \supset A$ and  $f \in \sF(U)$ such that $f \ge 1$ $m$-almost everywhere in $G$ and $\sE_U(f,f) + \int_{U}f^2\,dm < \Cap_1^U(A)  +\epsilon$. By \eqref{e:bndglob1}, \eqref{e:bndglob2}, there exists $C\in (1,\infty)$ depending only on the constants involved in the assumptions such that 
		\begin{equation} \label{e:comp1}
			\sE_1(E_Q(f),E_Q(f)) \le C \left( \sE_U(f,f)+ \int_{\ol{U}} f^2\,dm \right) < C (\Cap_1^U(A)+ \epsilon).
		\end{equation}
		By Lemma \ref{l:open}, we have that there exists an open set $\wt{G}$ in $X$ such that $A \subset G \subset \wt{G} \subset X$ and $E_Q(f) \ge 1$ $m$-almost everywhere in $\wt{G}$. Therefore by \eqref{e:comp1}, we have $\Cap_1(A) < C (\Cap_1^U(A)+\epsilon)$. Since $\epsilon>0$ is arbitrary, we obtain the desired conclusion.
		\item 
		Since $\overline{U} \setminus (F_k\cap\overline{U} )\subset  X \setminus F_k$ and $\{F_k\}$ is a nest with respect to $(X,d,m,\sE,\sF)$, we have
		\[
		\lim_{k \to \infty} \Cap_1^U \left( \overline{U} \setminus (F_k\cap\overline{U} )\right)	\le \lim_{k \to \infty} C \Cap_1\left( \overline{U} \setminus (F_k\cap\overline{U} )\right) =0.
		\]
		Therefore  $\{F_k \cap \overline{U}\}$ is a nest for $(\overline{U},d,m,\sE_U,\sF(U))$.
	The if part is an immediate consequence of (i). 
	It suffices to show that every  quasicontinuous function in $\sF(U)$ has a quasicontinuous extension  in $\sF$. To this end, let $f:\overline{U} \to \bR$ be quasicontinuous for $(\overline{U},d,m,\sE_U,\sF(U))$ and let $E_Q(f) \in \sF$ denote the extension map as  mentioned in (i). 
	By \cite[Theorem 2.1.3]{FOT}, there exists a quasicontinuous modification $\wt{E_Q(f)}$ of $E_Q(f)$ with respect to the Dirichlet form $(\sE,\sF)$. 
	Using the if part, the restriction $\restr{\wt{E_Q(f)}(x)}{\ol{U}}$ is quasicontinuous with respect to the Dirichlet form $(\sE_U,\sF(U))$.
	Hence by \cite[Lemma 2.1.4]{FOT}, the set $A= \{x \in \overline{U}: \wt{E_Q(f)}(x) \neq f(x)\}$ satisfies $\Cap_1^U(A)=0$ and hence by (i) also satisfies $\Cap_1(A)=0$. Therefore the function
	\[
	\wh{E_Q(f)}(x)= \begin{cases}
		f(x) & \mbox{if $x \in \ol{U}$,}\\
		\wt{E_Q(f)}(x) &\mbox{if $x \in (\ol{U})^c$,}
	\end{cases}
	\]
	is a quasicontinuous extension of $f$. \qedhere
	\end{enumerate}
	
\end{proof}

\section{Heat kernel estimates} \label{s:hke}

\subsection{A simpler cutoff Sobolev inequality}
We introduce a simplified version of cutoff Sobolev inequality and show that this simpler version is equivalent to \hyperlink{cs}{$\operatorname{CS}(\Psi)$}   in Definition \ref{d:PI-CS}(b).

\begin{definition} \label{d:css}
	We say that $(X,d,m,\sE,\sF)$ satisfies the \textbf{simplified cutoff Sobolev inequality} 
	\hypertarget{csa}{$\operatorname{CSS}(\Psi)$},
	if there exist  $C_S>0,A_1, A_2, C_1>1$ such that the following holds: for all $x \in X$ and $0<R< \diam(X,d)/A_2$,
	there exists a cutoff function $\phi \in \sF$ for $B(x,R) \subset B(x,A_1R)$ such that for all $f \in \sF$,
	\begin{equation}\tag*{$\operatorname{CSS}(\Psi)$}
		\int_{B(x,A_1R)} \wt{f}^2\, d\Gamma(\phi,\phi) \le C_1 \int_{B(x,A_1R) }  d\Gamma(f,f)
		+ \frac{C_1}{\Psi(R)} \int_{B(x,A_1R) } f^2\,dm;
	\end{equation}
	where $\wt{f}$ is a quasi-continuous version of $f \in \sF$.
\end{definition}
We note that \hyperlink{csa}{$\operatorname{CSS}(\Psi)$} is different from  \hyperlink{cs}{$\operatorname{CS}(\Psi)$} in the following aspects:
\begin{enumerate}[(a)]
	\item We only consider cutoff functions for $B(x,R) \subset B(x,A_1 R)$ for $0<R< \diam(X,d)/A_2$ in \hyperlink{csa}{$\operatorname{CSS}(\Psi)$} instead of $B(x,R) \subset B(x,R+r)$ for all $0<r<R$ in \hyperlink{cs}{$\operatorname{CS}(\Psi)$}.
	\item The integrals on \hyperlink{csa}{$\operatorname{CSS}(\Psi)$}  are over the larger ball as opposed to annuli in \hyperlink{cs}{$\operatorname{CS}(\Psi)$}.
	\item The first term in the right side has coefficient $C_1$ in \hyperlink{csa}{$\operatorname{CSS}(\Psi)$} as opposed to $\frac 1 8$ in \hyperlink{cs}{$\operatorname{CS}(\Psi)$}.
\end{enumerate}
It is immediate to see that  \hyperlink{cs}{$\operatorname{CS}(\Psi)$} implies \hyperlink{csa}{$\operatorname{CSS}(\Psi)$}. 
Next, we show the converse that  \hyperlink{csa}{$\operatorname{CSS}(\Psi)$} implies   \hyperlink{cs}{$\operatorname{CS}(\Psi)$}. The proof follows from similar \emph{self-improvement} properties of cutoff Sobolev inequality in \cite[Lemma 5.1]{AB} and \cite[Proposition 5.11]{BM}.
\begin{lemma} \label{l:improve} 
	Let $(\sX,d,m,\sE,\sF)$ satisfy the volume doubling property 
	and  \hyperlink{csa}{$\operatorname{CSS}(\Psi)$} for some regular scale function  $\Psi$. Then $(X,d,m,\sE,\sF)$ satisfies \hyperlink{cs}{$\operatorname{CS}(\Psi)$}.
\end{lemma}

\begin{proof}
	By \cite[Lemma 5.1]{AB}, it suffices to show that there exists $C_2>0$ such that for all $x \in X, 0<r<R$, there exists a cutoff function $\phi \in \sF$ for $B(x,R) \subset B(x,R+r)$ such that for all $f \in \sF$
	\begin{align} \label{e:si1}
		\MoveEqLeft{\int_{B(x,R+r) \setminus B(x,R)} \wt{f}^2\, d\Gamma(\phi,\phi)} \nonumber \\
		&\le  C_2 \int_{B(x,R+r) \setminus B(x,R)}   \, d\Gamma(f,f)
		+ \frac{C_2}{\Psi(r)} \int_{B(x,R+r) \setminus B(x,R)} f^2\,dm,
	\end{align}
	where $\wt{f}$ is a quasi-continuous version of $f$.

	Let $x \in X, 0<r<R$ be arbitrary. 
	If $B(x,R+r)=X$ we simply choose $\phi\equiv 1$ as the cutoff function. In this case, by strong locality we have $\Gamma(\phi,\phi) \equiv 0$.
	
	It remains to consider the case $B(x,R+r) \neq X$. In this case $2r < R+r \le \diam(X,d)$ implies $r < \diam(X,d)/2$. 
	Let $A_1, A_2,C_1>1$ denote the constants such that \hyperlink{csa}{$\operatorname{CSS}(\Psi)$} holds. 
	Define 
	\begin{equation}
		\label{e:defL}   L:= \max(A_2,4(A_1+1)).
	\end{equation}
	Let $N=\{z_i: i \in I\}$ denote a maximal $r/L$-separated subset of $X$. For each $i \in I$, let $B_i= B(z_i, r/L), B_i^*=B(z_i, A_1r/L)$. Since $r/L< \diam(X,d)/A_2$,
	by \hyperlink{csa}{$\operatorname{CSS}(\Psi)$}, there exists a cutoff function $\vp_i$ for $B_i \subset B_i^*$ such that for all $f \in \sF$
	\begin{equation}\label{e:si2}
		\int_{B_i^*} \wt{f}^2 \, d\Gamma(\vp_i,\vp_i) 
		\le C_1 \int_{B_i^*} d\Gamma(f,f) + \frac{C_1}{\Psi(r/L)} \int_{B_i^*} f^2\, dm,
	\end{equation}
	where $\wt{f}$ is a quasi-continuous version of $f$.
	Define $J \subset I$ as 
	\[
	J :=\{i \in I: z_i \in B(x,R+r/2)\}.
	\]
	Since $\cup_{i \in B_i} X$, and $N$ is a $r/L$-separated, we have $B(x,R+r/2-r/L) \subset \cup_{j \in J} B_j$ and $\cup_{j \in J} B_j^* \subset B(x,R+r/2+rA_1/L)$. Therefore the function 
	\[
	\phi= \max_{j \in J} \vp_j
	\]
	satisfies (by \cite[Theorem 1.4.2]{FOT}(i),(ii))
	\[
	\phi \in \sF,  0 \le \phi \le 1, \q \phi \equiv 1 \mbox{ on $B(x,R+r/2-r/L)$},\,\supp(\phi) \subset  B(x,R+r/2+rA_1/L).
	\]
	Note that by strong locality, we have $$\supp(\Gamma(\phi,\phi)) \subset B(x,R+r/2+rA_1/L) \setminus B(x, R+r/2-r/L).$$
	Let $\psi \in C(X) \cap \sF$ be a cutoff function for $B(x,R+r/2+rA_1/L) \setminus B(x, R+r/2-r/L) \subset B(x,R+r/2+r(A_1+1)/L) \setminus B(x, R+r/2-2r/L)$. 
	Let $$V:=B(x,R+r/2+r(A_1+1)/L) \setminus B(x, R+r/2-2r/L).$$
	For any quasi-continuous function $f \in \sF \cap L^\infty(m)$,
	\begin{align} \label{e:si3}
		\MoveEqLeft{\int_{B(x,R+r)\setminus B(x,r)} f^2 \,d\Gamma(\phi,\phi)} \nonumber \\ &= 	\int_{B(x,R+r)\setminus B(x,r)} (\psi f)^2 \,d\Gamma(\phi,\phi) \nonumber \\
		& \le \sum_{j \in J} \int_{B(x,R+r)\setminus B(x,r)}  (\psi f)^2 \,d\Gamma(\vp_j,\vp_j)   \mbox{ (by \cite[Lemma 5.10]{BM})} \nonumber\\
		&\le \sum_{j \in J} \int_{B_j^*} \one_{V} f^2 \, d\Gamma(\vp_j,\vp_j) \mbox{ (since $\supp(\vp_j) \subset B_j^*$, and  $\psi \le \one_V$).}  
	\end{align}
	Let $J_1:= \{j \in J: d(x,z_j) \ge R+r/2-(A_1+2)r/L\}$. Since $\supp(\Gamma(\vp_i,\vp_i)) \subset B_i^*$ and $V \cap B_j^*= \emptyset$ for all $j \in J \setminus J_1$, we have 
	\begin{align} \label{e:si4}
		\sum_{j \in J} \int_{B_j^*} \one_{V} f^2 \, d\Gamma(\vp_j,\vp_j)  &= \sum_{j \in J_1} \int_{B_j^*} \one_{V} f^2 \, d\Gamma(\vp_j,\vp_j) \nonumber \\
		& \le C_1 \int_X \sum_{j \in J_1} \one_{B_j^*} \,d \Gamma(f,f) + \frac{C_1}{\Psi(r/L)} \int_X \sum_{j \in J_1} \one_{B_j^*} f^2 \,dm.
	\end{align}
	By the metric doubling property and $R+r/2- (A_1+2)r/L \le d(x,z_j) < R+r/2$ for all $j \in J_1$, we have 
	\begin{equation} \label{e:si5}
		\sum_{j \in J_1} \one_{B_j^*} \lesssim \one_{B(x,R+r/2+A_1r/L) \setminus B(x,R+r/2- 2(A_1+1)r/L)} \stackrel{\eqref{e:defL}}{\lesssim} \one_{B(x,R+r) \setminus B(x,R)}.
	\end{equation}
	By \eqref{e:si3}, \eqref{e:si4}, \eqref{e:si5} and \eqref{e:reg}, we obtain \eqref{e:si1} for all $f \in \sF \cap L^\infty(m)$. By approximating any $f \in \sF$, with $f_n:= (n \wedge f)\vee (-n) \in \sF \cap L^\infty$ and letting $n \to \infty$, we obtain \eqref{e:si1} for all $f \in \sF$.
\end{proof}
\subsection{Extension map and the cutoff Sobolev inequality}
\begin{proposition} \label{p:cs}
	Let $(X,d,m,\sE,\sF)$ be an MMD space satisfying the volume doubling property and the heat kernel estimate \hyperlink{hke}{$\on{HKE}(\Psi)$}, where $\Psi$ is a scale function. Let $U \subset X$ 
	be a uniform domain. Then the MMD space $(\overline{U}, d, m, \sE_U, \sF(U))$ corresponding to the reflected process on $U$ satisfies \hyperlink{csa}{$\on{CSS}(\Psi)$}. 
\end{proposition}
\begin{proof}
	By Theorem \ref{t:hke}, $(X,d,m,\sE,\sF)$ satisfies \hyperlink{cs}{$\on{CS}(\Psi)$}. Therefore there exists $C_1>0$ such that for all $x \in X, R>0, f \in \sF$, there exists a cutoff function $\phi \in \sF$ for $B(x,R) \subset B(x,2R)$ 
	\begin{equation} \label{e:cs1}
		\int_{B(x,2R) \setminus B(x,R)} \wt{f}^2 \, d\Gamma(\phi,\phi) \le \frac{1}{8} \int_{B(x,2R)} d\Gamma(f,f) + \frac{C_1}{\Psi(R)} \int_{B(x,2R)} f^2 \, dm,
	\end{equation}
	where $\wt{f}$ is a quasi-continuous version of $f$.

	By \eqref{e:cs1}, we the desired estimate in \hyperlink{csa}{$\on{CSS}(\Psi)$} for  $(\overline{U}, d, m, \sE_U, \sF(U))$ for any choice of $A_1 \ge 2$ provided $x \in \overline{U}, R> 0$ satisfies  $\delta_U(x) > 2R$.
	
	It suffices to consider the case $x \in \overline{U}, R \le \delta_U(x) \le 2R$. We choose $\phi$ such that \eqref{e:cs1} holds. Let $E: \sF(U) \to \sF$ be an extension operator as given in Theorem \ref{t:extend}.
	Let $x \in \overline{U}, 0 < r \lesssim \diam(X,d)$ and $f \in \sF(U)$.
	Let $\wt{f}$ be a quasi-continuous modification of $f$ for the Dirichlet form $(\sE_U,\sF(U))$ on $L^2(\ol{U},\restr{m}{\ol{U}})$.
	By Proposition \ref{p:qcts}(iii), there exists $\wt{E(f)}$   a quasi-continuous modification of the extension $E(f) \in \sF$ with respect to $(\sE,\sF)$ on $L^2(X,m)$ such that  $\restr{\wt{E(f)}}{\ol{U}}= \wt{f}$. We estimate
	\begin{align*}
	\MoveEqLeft{\int_{B_{\overline{U}}(x,2R)\setminus B_{\overline{U}}(x,R)} \wt{f}^2 \, d\Gamma(\phi,\phi)} \nonumber \\ &\le \int_{B(x,2R) \setminus B(x,R)} \abs{ \wt{E(f)}}^2 \, d\Gamma(\phi,\phi) \\
		&\stackrel{\eqref{e:cs1}}{\lesssim} \int_{B(x,2R)} d\Gamma(E(f),E(f)) + \frac{1}{\Psi(R)} \int_{B(x,2R)} \abs{E(f)}^2\,dm \\
		&\stackrel{\eqref{e:ebndloc1}, \eqref{e:ebndloc2}}{\lesssim} \int_{B_U(x,2KR)} d\Gamma(f,f) + \frac{1}{\Psi(R)} \int_{B_U(x,2KR)} f^2\,dm,
	\end{align*}
	where $K$ is as given in Theorem \ref{t:extend}. Since $\phi$ is a cutoff function for $B(x,R) \subset B(x,2R)$ it is a cutoff function for $B(x,R) \subset B(x,2KR)$ and by strong locality \cite[Theorem 4.3.8]{CF}, we have $\Gamma(\phi,\phi)\left((B(x,2R)\setminus B(x,R))^c\right)=0$ and hence
	 $$\int_{B_{\overline{U}}(x,2KR)\setminus B_{\overline{U}}(x,R)} \wt{f}^2 \, d\Gamma(\phi,\phi)=\int_{B_{\overline{U}}(x,2R)\setminus B_{\overline{U}}(x,R)} \wt{f}^2 \, d\Gamma(\phi,\phi).$$ Combining the above two displays, we have \hyperlink{csa}{$\on{CSS}(\Psi)$} by choosing $A_1= 2K$. 
\end{proof}
\begin{remark}
	In \cite{Lie}, the author studies killed diffusion (with Dirichlet boundary condition instead of Neumann boundary condition considered in this work). In the Dirichlet boundary condition setting, the proof of cutoff Sobolev inequality does not require any extension operator, since every function in the domain can be extended as zero outside the domain. 
\end{remark}
We have the ingredients to prove Theorem \ref{t:hkeunif}. 
\begin{proof}[Proof of Theorem \ref{t:hkeunif}] 
By Theorem \ref{t:extend} along with Lemma \ref{l:regular}, we obtain that $(\sE_U,\sF(U))$ is a regular Dirichlet form on $L^2(\overline U,m)$. By Theorems \ref{t:hke}, \ref{t:pi}, Lemma \ref{l:doubling}, and Proposition \ref{p:cs}, the MMD space   $(\overline{U},d,m,\sE_U,\sF(U))$  corresponding to reflected diffusion satisfies the heat kernel estimate  \hyperlink{hke}{$\on{HKE(\Psi)}$}.
\end{proof}

\subsection{Concluding remarks.} \label{s:conclude}
One of the most important open problems concerning heat kernel estimates is the \emph{resistance conjecture} which states that the characterization of sub-Gaussian heat kernel estimates in Theorem \ref{t:hke}(b) can be replaced with the simpler conditions Poincar\'e inequality \hyperlink{pi}{$\on{PI(\Psi)}$} and the capacity upper bound \hyperlink{cap}{$\operatorname{cap}(\Psi)_\le$} 	 \cite[Conjecture 4.15]{GHL14}. N.~Kajino\footnote{personal communication.} observed that if the resistance conjecture were true, then Theorem \ref{t:hkeunif} will have a simpler proof because it is easy to verify the capacity upper bound \hypertarget{cap}{$\operatorname{cap}(\Psi)_\le$} on uniform domain (and inner uniform domains) using the corresponding property on the ambient space.
Although there is some partial progress on the resistance conjecture, it 
remains open in general \cite{Mur23}. This connection to heat kernel estimate for reflected diffusion is further motivation to solve resistance conjecture. Following the work of Gyrya and Saloff-Coste \cite{GS}, it is natural to conjecture that Theorem \ref{t:hkeunif} should also be true for inner uniform domains (here the instrinc metric in $\ol U$ should be used for heat kernel estimates in $\ol U$). This would require developing a more instrinsic approach that does not rely on any extension operator.

Next, we compare the role of Lipschitz functions in the fields of `analysis/diffusion on fractals' \cite{Bar98,Kig01} and `analysis on metric spaces' \cite{Cheeg, Hei,HKST}.  In the latter setting, Lipschitz functions play a central role while in the former case  Lipschitz functions do not play much of a role.   J.~Heinonen \cite[Chapter 6]{Hei} writes  ``Lipschitz functions are the smooth functions of metric spaces". One justification for the above quote is that the  Sobolev spaces considered in the `analysis on metric spaces' field have a dense set of Lipschitz functions \cite[Theorem 8.2.1]{HKST}, \cite{Eri}. In fact, the proof of extension property for Newton-Sobolev spaces in \cite{BS} uses such a density of Lipschtiz functions. On the other hand, our proof of Theorem \ref{t:extend} does not use any Lipschitz functions.

In the setting of MMD spaces satisfying sub-Gaussian heat kernel bound \hyperlink{hke}{$\on{HKE(\Psi)}$}, it is known that if the space time scaling is Gaussian (that is, $\Psi(r)=r^2$), then Lipschitz functions are dense in the domain of the Dirichlet form \cite[Lemma 2.11]{ABCRST}, \cite[Remark 2.11]{KM20}.
For general space-scaling, it is believed that   Lipschitz functions are not necessarily dense in the domain of the Dirichlet form   in general. To explain this, we recall an old conjecture of Barlow and Perkins \cite{BP} concerning Brownian motion on the Sierpi\'nski gasket which is still open. In \cite[Section 9]{BP}, the authors conjecture that there are no non-constant $\alpha$-H\"older functions in the domain of the generator\footnote{the generator of the Brownian motion is a self-adjoint operator that is the analogous to the Laplacian.} for the Brownian motion on Sierpinski gasket for any  $\alpha > \frac{\log (5/3)}{\log 2} = .736966....$. One could also make an analogous conjecture  for the domain of the Dirichlet form. In particular, we conjecture that there are no non-constant Lipschitz functions (with respect to the Euclidean metric) in the domain of the Dirichlet form for the  Brownian motion on Sierpinski gasket. 

\noindent \textbf{Acknowledgments.} 
I am grateful to Naotaka Kajino for introducing me to the problem of obtaining heat kernel estimates for reflected diffusion on uniform domains, explaining the issue of proving cutoff Sobolev inequality, for the remark concerning resistance conjecture in \textsection \ref{s:conclude}, and for many helpful discussions. 
I understood that a scale-invariant version of extension property is relevant for this problem after listening to Pekka Koskela's lecture on extension domains at the workshop on Heat Kernels, Stochastic Processes and Functional Inequalities at Oberwolfach in November 2022. I thank the organizers of this workshop for the opportunity and Pekka Koskela for the inspiring lecture. I benefited from discussions with Zhen-Qing Chen on reflected diffusion. I thank Nageswari Shanmugalingam for the reference \cite{Raj}. The author is grateful to the anonymous referees and Ryosuke Shimizu for their helpful comments and suggestions.



\begin{thebibliography}{}
	\bibitem[Aik]{Aik} H.~Aikawa,  Boundary Harnack principle and Martin boundary for a uniform domain. {\em J. Math. Soc. Japan} {\bf 53} (2001), no. 1, 119--145.
	
	\bibitem[ABCRST]{ABCRST} P.\ Alonso-Ruiz, F.\ Baudoin, L.\ Chen, L.\ Rogers, N.\ Shanmugalingam and A.\ Teplyaev,
	Besov class via heat semigroup on Dirichlet spaces II: BV functions and Gaussian heat kernel estimates,
	to appear in \emph{Calc.\ Var.\ Partial Differential Equations} {\bf 59}  (2020), no. 3, Paper No.103, 32 pp.
	\bibitem[AB]{AB} S. Andres, M. T. Barlow. 
	Energy inequalities for cutoff-functions and some applications, 
	{\em J. Reine Angew. Math.} {\em 699} (2015), 183--215. 
	
	\bibitem[Aro]{Aro} D.~G.~Aronson, Bounds for the fundamental solution of a parabolic equation, {\em Bull. Amer.
		Math. Soc.} {\bf 73} 1967 890--896.
	
	\bibitem[Bar98]{Bar98} M. T. Barlow. 
	Diffusions on fractals,
	{\em Lecture Notes in Math.} {\bf 1690}, 1--121, Springer, Berlin, 1998. 
	
	\bibitem[BB96]{BB96} M.~T.~Barlow, R.~F. Bass, Brownian Motion and Harmonic Analysis on
	Sierpinski Carpets, {\em Canad. J. Math.}   {\bf 51} (4), 1999, 673--744.
	
	
	\bibitem[BB04]{BB04}
	M.T. Barlow, R.F. Bass. Stability of parabolic Harnack inequalities. 
	{\em Trans. Amer. Math. Soc.} {\bf 356} (2004) no. 4, 1501--1533.
	
	
	
	\bibitem[BBK]{BBK}
	M.T. Barlow, R.F. Bass, T. Kumagai. 
	Stability of parabolic Harnack inequalities on metric measure spaces, 
	{\em J. Math. Soc. Japan} (2) {\bf 58} (2006), 485--519.  (correction in \arxiv{2001.06714})
	
	\bibitem[BH]{BH} M.~T.~Barlow and B.~M.~Hambly, Transition density estimates for Brownian
	motion on scale irregular Sierpinski gaskets. {\em Ann. IHP.} {\bf 33}, 531--557 (1997).
	
	\bibitem[BP]{BP} M.~T.~Barlow, E.~A.~Perkins, Brownian motion on the Sierpi\'nski gasket, {\em Probab. Theory
		Related Fields} {\bf 79} (1988), no. 4, 543--623.
	
	\bibitem[BM]{BM}
	M.\ T.\ Barlow and M.\ Murugan,
	Stability of the elliptic Harnack inequality,
	{\em Ann. of Math.}\ (2) {\bf 187} (2018), 777--823 
	
	\bibitem[BST]{BST} O.\ Ben-Bassat, R.\ S.\ Strichartz and A.\ Teplyaev,
	What is not in the domain of the Laplacian on Sierpinski gasket type fractals,
	\emph{J.\ Funct.\ Anal.}\ \textbf{166} (1999), no.\ 2, 197--217.
	
	\bibitem[BCR]{BCR} I.~Benjamini, Z.-Q.~Chen, S.~Rohde, 
	Boundary trace of reflecting Brownian motions. 
	{\em	Probab. Theory Related Fields} {\bf 129} (2004), no. 1, 1--17.
	
	\bibitem[BS]{BS}  J.~Bj\"orn, N.~Shanmugalingam.  
	Poincar\'e inequalities, uniform domains and extension properties for Newton-Sobolev functions in metric spaces.  
	{\em J. Math. Anal. Appl.} {\bf 332} (2007),  190--208.
	
	\bibitem[BHK]{BHK} M.~Bonk, J.~Heinonen, P.~Koskela,  Uniformizing Gromov hyperbolic spaces. {\em Ast\'erisque} No. {\bf 270} (2001), viii+99 pp. 
	
	\bibitem[Cal61]{Cal61}  A.-P.~Calder\'on,
	Lebesgue spaces of differentiable functions and distributions. 1961 {\em Proc. Sympos. Pure Math.}, Vol. {\bf IV} pp. 33--49, 1961, American Mathematical Society, Providence, R.I.
	
	\bibitem[Cal72]{Cal72} A.-P.~Calder\'on,
	Estimates for singular integral operators in terms of maximal functions.
	{\em Studia Math.} {\bf 44} (1972), 563--582.
	
	\bibitem[CQ]{CQ} S.~Cao, H.~Qiu.  Uniqueness and convergence of resistance forms on unconstrained Sierpinski carpets  {\tt \href{https://math.nju.edu.cn/DFS/file/2023/09/28/202309281328226917qhv2q.pdf}{(preprint)}} 2023.
	
	\bibitem[Cheeg]{Cheeg} J.~Cheeger,  Differentiability of Lipschitz functions on metric measure spaces. {\em Geom. Funct. Anal.} {\bf 9} (1999), no. 3, 428--517. 
	
	
	\bibitem[Che]{Che} Z.-Q.~Chen, On reflecting diffusion processes and Skorokhod decompositions. {\em Probab. Theory
		Relat. Fields}, {\bf 94} (1993), 281--351
	
	\bibitem[CF]{CF} Z.-Q.\ Chen and M.\ Fukushima,
	\emph{Symmetric Markov Processes, Time Change, and Boundary Theory},
	London Math.\ Soc.\ Monogr.\ Ser., vol.\ 35,
	Princeton University Press, Princeton, NJ, 2012. 
	
	\bibitem[CKKW]{CKKW} Z.-Q.~Chen, P.~Kim, T.~Kumagai, J.~Wang,  Heat kernels for reflected diffusions with jumps on inner uniform domains. {\em Trans. Amer. Math. Soc.} {\bf 375} (2022), no. 10, 6797--6841.
	
	
	\bibitem[Eri]{Eri} S.~Eriksson-Bique, 
	Density of Lipschitz functions in energy. 
	{\em Calc. Var. Partial Differential Equations} {\bf 62} (2023), no. 2, Paper No. 60, 23 pp.
	
	\bibitem[FHK]{FHK} P.~J.~Fitzsimmons, B.~M.~Hambly, T.~Kumagai: Transition density estimates for Brownian motion on affine nested fractals. {\em Comm. Math. Phys.} {\bf 165}, 595--620 (1995).
	
	\bibitem[F]{F} M.~Fukushima,  A construction of reflecting barrier Brownian motions for bounded domains.
	{\em	Osaka J. Math} {\bf 4}, 183--215 (1967) 
	
	\bibitem[FOT]{FOT} 
	M. Fukushima, Y. Oshima, and M. Takeda.
	{\em Dirichlet Forms and Symmetric Markov Processes}.
	De Gruyter Studies in Mathematics, \textbf{19}, Berlin, 2011. 
	
	\bibitem[Geh]{Geh} F.~W.~Gehring,  Uniform domains and the ubiquitous quasidisk. {\em Jahresber. Deutsch. Math.-Verein.} {\bf 89} (1987), no. 2, 88--103. 
	
	\bibitem[GH]{GH}  F.~W.~Gehring, K.~Hag, The ubiquitous quasidisk. With contributions by O.~J.~Broch. Mathematical Surveys and Monographs, {\bf  184}. {\em American Mathematical Society, Providence}, RI, 2012. xii+171 pp. 
	
	\bibitem[GV]{GV} V.~M.~Gol'dshtein, S.~K.~Vodop'janov, Prolongement des fonctions differentiables hors de domaines plans, {\em C. R. Acad. Sci. Paris Ser. I Math.}  {\bf 293} (12) (1981) 581--584.
	
	\bibitem[Gri]{Gri} A. Grigor'yan. The heat equation on noncompact Riemannian 
	manifolds. 
	(in Russian) {\it Matem. Sbornik. \bf 182} (1991), 55--87.
	(English transl.) {\it Math. USSR Sbornik \bf 72} (1992), 47--77.
	
	\bibitem[GHL03]{GHL03} A.~Grigor'yan, J.~Hu Jiaxin, K.-S.~Lau,  Heat kernels on metric-measure spaces and an application to semi-linear elliptic equations, {\em Trans. AMS} {\bf 355} (2003) no.5, 2065--2095
	
	\bibitem[GHL14]{GHL14} A. Grigor'yan, J. Hu, K.-S. Lau. 
	Heat kernels on metric measure spaces. 
	Geometry and analysis of fractals, 147--207, 
	{\em Springer Proc. Math. Stat.}, {\bf 88}, Springer, Heidelberg, 2014.
	
	\bibitem[GHL15]{GHL15} A.~Grigor'yan, J.~Hu Jiaxin, K.-S.~Lau,  Generalized capacity, Harnack inequality and heat kernels of Dirichlet forms on metric spaces, {\em J. Math. Soc. Japan}, {\bf 67} (2015) 1485--1549 
	
	\bibitem[GT12]{GT12} A.\ Grigor'yan and A.\ Telcs,
	Two-sided estimates of heat kernels on metric measure spaces,
	\emph{Ann.\ Probab.}\ \textbf{40} (2012), no.\ 3, 1212--1284.
	
	\bibitem[GrS]{GrS}  A. Grigor'yan, L. Saloff-Coste, Stability results for Harnack inequalities, {\em Ann. Inst. Fourier
		(Grenoble)} {\bf 55 }(2005), no. 3, 825--890.
	
	\bibitem[GyS]{GS} P.~Gyrya, L.~Saloff-Coste. 
	Neumann and Dirichlet heat kernels in inner uniform domains. 
	{\em Ast\'erisque} No. {\bf 336 }(2011), viii+144 pp.
	
	\bibitem[HaK]{HK} P.~Haj\l asz, P.~Koskela,  Sobolev meets Poincar\'e. {\em C. R. Acad. Sci. Paris S\'er. I Math.} {\bf 320} (1995), no. 10, 1211--1215.
	\bibitem[HKT]{HKT} P.~Haj\l asz,   P.~Koskela, H.~Tuominen, 
	Sobolev embeddings, extensions and measure density condition.  
	{\em J. Funct. Anal.} {\bf 254} (2008), no. 5, 1217--1234.
	\bibitem[HeK]{HeK} D.~A.~Herron, P.~Koskela,  
	Uniform and Sobolev extension domains.
	{\em Proc. Amer. Math. Soc.} {\bf 114} (1992), no. 2, 483--489.
	
	\bibitem[HKST]{HKST}  
	J. Heinonen, P. Koskela, N. Shanmugalingam, J. T. Tyson.
	Sobolev spaces on metric measure spaces. 
	An approach based on upper gradients. {\em New Mathematical Monographs}, {\bf 27}. Cambridge University Press, Cambridge, 2015. xii+434  
	
	\bibitem[Hei]{Hei} J.~Heinonen.  Lectures on Analysis on Metric Spaces, {\em Universitext. Springer-Verlag}, New York, 2001. x+140 pp. 
	
	\bibitem[Hin05]{Hin05} M.\ Hino,
	On singularity of energy measures on self-similar sets,
	\emph{Probab.\ Theory Related Fields} \textbf{132} (2005), no.\ 2, 265--290.
	
	\bibitem[Hin10]{Hin10} M.\ Hino,
	Energy measures and indices of Dirichlet forms, with applications to derivatives on some fractals,
	\emph{Proc.\ Lond.\ Math.\ Soc.}\ (3) \textbf{100} (2010), no.\ 1, 269--302.
	
	\bibitem[Hin13]{Hin13} M.\ Hino,
	Upper estimate of martingale dimension for self-similar fractals, {\em Probab. Theory Related Fields}, vol. {\bf 156} (2013), 739--793. 
	
	\bibitem[Jon80]{Jon80} P.~W.~Jones, 
	Extension theorems for BMO. {\em Indiana Math. J.}, {\bf 29} (1980), 41--66.
	
	\bibitem[Jon81]{Jon81} P.~W.~Jones, 
	Quasiconformal mappings and extendability of functions in Sobolev spaces.
	{\em Acta Math.} {\bf 147} (1981), no. 1-2, 71--88.
	
	\bibitem[Jons]{Jons} A.~Jonsson, Brownian motion on fractal and function spaces.
	{\em Math. Z.}, {\bf 222}, 495--504 (1996).
	
	\bibitem[KM20]{KM20} N. Kajino, M. Murugan, On singularity of energy measures for symmetric diffusions with full off-diagonal heat kernel estimates. Ann. Probab. 48 (2020), no. 6, 2920?2951.
	
	\bibitem[KM23]{KM23} N. Kajino, M. Murugan, On the conformal walk dimension: Quasisymmetric uniformizition for symmetric diffusions, 
	{\em	Invent. Math.} {\bf 231} (2023), no. 1, 263--405.
	
	\bibitem[KM23+]{KM23+} N. Kajino, M. Murugan, Heat kernel estimates for boundary trace
	of reflected diffusions on uniform domains  \arxiv{2312.08546} (preprint).
	
	\bibitem[Kig01]{Kig01}
	J.\ Kigami, \emph{Analysis on Fractals}, Cambridge Tracts in Math., vol.\ 143,
	Cambridge University Press, Cambridge, 2001.
	
	
	\bibitem[Kig12]{Kig12} J.\ Kigami,
	Resistance forms, quasisymmetric maps and heat kernel estimates,
	\emph{Mem.\ Amer.\ Math.\ Soc.}\ \textbf{216}, no.\ 1015 (2012).
	
	\bibitem[KoSc]{KS} N.~J.~Korevaar, R.~M.~Schoen,  Sobolev spaces and harmonic maps for metric space targets. {\em Comm. Anal. Geom.} {\bf 1} (1993), no. 3-4, 561--659.
	
	\bibitem[Kum]{Kum} T.~Kumagai, Estimates of transition densities for Brownian motion on nested fractals, \emph{Probab. Theory Related Fields} {\bf 96} (1993), no. 2, 205--244.
	
	\bibitem[KuSt]{KuSt}   T.~Kumagai, K.~T.~Sturm,
	Construction of diffusion processes on fractals, d-sets, and general metric measure spaces.
	{\em J. Math. Kyoto Univ.} {\bf 45} (2005), no. 2, 307--327.
	
	
	\bibitem[Kus89]{Kus89} S.\ Kusuoka, 
	Dirichlet forms on fractals and products of random matrices,
	\emph{Publ.\ Res.\ Inst.\ Math.\ Sci.}\ \textbf{25} (1989), no.\ 4, 659--680. 
	
	\bibitem[LY]{LY} P.~Li, S.~T.~Yau, On the parabolic kernel of the Schrdinger operator, {\em Acta Math.} {\bf 156}
	(1986), no. 3-4, 153--201.
	
	
	\bibitem[Lie]{Lie} J.~Lierl,  The Dirichlet heat kernel in inner uniform domains in fractal-type spaces. {\em Potential Anal.} {\bf 57} (2022), no. 4, 521--543
	
	\bibitem[Lin]{Lin} T. Lindstr\o m, Brownian motion on nested fractals. {\em Mem. A.M.S.} {\bf 420}, 1990.
	
	\bibitem[LS]{LS} P.-L.~Lions, A.-S.~Sznitman,  Stochastic differential equations with reflecting boundary conditions. {\em Comm. Pure Appl. Math.} {\bf 37} (1984), no. 4, 511--537. 
	
	\bibitem[MS]{MS} O.~Martio, J.~Sarvas, 
	Injectivity theorems in plane and space.
	{\em Ann. Acad. Sci. Fenn. Ser. A I Math.} {\bf 4} (1979), no. 2, 383--401.
	
	\bibitem[Mar]{Mar} O.~Marito, Definitions for uniform domains, {\em Ann. Acad. Sci. Fenn. Ser. A I Math.} {\bf 4} (1979), 383--401.
	
	
	\bibitem[Mur20]{Mur20} M. Murugan,
	On the length of chains in a metric space, \emph{J. Funct. Anal.} {\bf  279} (2020), no. 6, 108627, 18 pp.
	
	\bibitem[Mur23+]{Mur23} M. Murugan,  A note on heat kernel estimates, resistance bounds and Poincar\'e inequality \emph{Asian J. Math.}
	(to appear).
	
	\bibitem[Raj]{Raj} T.~Rajala,
	Approximation by uniform domains in doubling quasiconvex metric spaces. {\em Complex Anal. Synerg.} {\bf 7} (2021), no. 1, Paper No. 4, 5 pp.
	
	\bibitem[Rog]{R} L.~G.~Rogers. Degree-independent Sobolev extension on locally uniform domains. {\em J. Funct. Anal. }{\bf 235} (2006), no. 2, 619--665
	
	\bibitem[Sal92]{Sal92}	L.~Saloff-Coste, Uniformly elliptic operators on Riemannian manifolds,
	{\em J. Differential Geom.} {\bf 36} no. 2 (1992), 417--450.
	
	\bibitem[Sal02]{Sal} L.~Saloff-Coste. Aspects of Sobolev-type inequalities. 
	{\em London Mathematical Society Lecture Note Series},{\bf 289}. Cambridge University Press, Cambridge, 2002. x+190 pp.  
	
	\bibitem[Sal10]{Sal10} L.~Saloff-Coste. The heat kernel and its estimates, Probabilistic approach to geometry,
	405--436, {\em Adv. Stud. Pure Math.},{\bf 57}, Math. Soc. Japan, Tokyo, 2010.
	
	\bibitem[Shv]{Shv}	P.~Shvartsman,  On Sobolev extension domains in $\mathbb{R}^n$. {\em J. Funct. Anal.} {\bf 258} (2010), no. 7, 2205--2245.
	
	\bibitem[S]{S} A.~V.~Skorokhod,  Stochastic equations for diffusion processes in a bounded region 1, 2, {\em Theor. Veroyatnost. i Primenen}, {\bf 6}, 1961, pp. 264--274, {\bf 7}, 1962, pp. 3--23.
	
	\bibitem[Soa]{Soa} P.~M.~Soardi, 
	Potential theory on infinite networks.
	{\em	Lecture Notes in Mathematics}, {\bf 1590}. Springer-Verlag, Berlin, 1994. viii+187 pp.
	
	\bibitem[Ste]{Ste} E.~M.~Stein,
	Singular integrals and differentiability properties of functions.
	Princeton Mathematical Series, No. 30 {\em Princeton University Press}, Princeton, N.J. 1970 xiv+290 pp.
	
	\bibitem[Stu]{Stu96} K.-T.\ Sturm,
	Analysis on local Dirichlet spaces --- III. The parabolic Harnack inequality,
	\emph{J.\ Math.\ Pures Appl.}\ (9) \textbf{75} (1996), no.\ 3, 273--297.
	
	\bibitem[Tan]{Tan} H.~Tanaka, 
	Stochastic differential equations with reflecting boundary condition in convex regions.
	{\em	Hiroshima Math. J.} {\bf 9} (1979), no. 1, 163--177.
	
	\bibitem[V\"ai]{Vai}
	J. V\"ais\"al\"a, Uniform domains. {\em Tohoku Math. J.} {\bf 40} (1988), 101--118
	
	\bibitem[VSC]{VSC} N.~Th.~Varopoulos, L.~Saloff-Coste, T.~Couhlon, Analysis and Geometry of Groups, {\em Cambridge Tracts in Mathematics}, {\bf 100}. Cambridge University Press, Cambridge, 1992. xii+156 pp.
	
	\bibitem[Whi]{Whi} H.~Whitney.
	Analytic extensions of differentiable functions defined in closed sets.
	{\em Trans. Amer. Math. Soc.} {\bf 36} (1934), no. 1, 63--89.
\end{thebibliography}
\end{document}